\documentclass[12pt, reqno]{amsart}
\usepackage{mathrsfs}
\usepackage{amssymb,amsthm,amsmath}
\usepackage[numbers,sort&compress]{natbib}
\usepackage{amssymb,amsmath}
\usepackage{amsfonts}
\usepackage{mathrsfs}
\usepackage{latexsym}
\usepackage{amssymb}
\usepackage{amsthm}
\usepackage{color}
\usepackage{pdfsync}
\usepackage{indentfirst}

\date{today}

\usepackage{hyperref}
\hypersetup{hypertex=true,
colorlinks=true,
linkcolor=blue,
anchorcolor=g,
citecolor=red}

\usepackage{amsmath}
\usepackage{amsthm}

\newtheorem{remark}{Remark}[section]

\newtheorem{theorem}{Theorem}[section]
\newtheorem{proposition}{Proposition}[section]
\newtheorem{lemma}{Lemma}[section]

\newcommand{\beq}{\begin{equation}}
\newcommand{\eeq}{\end{equation}}
\newcommand{\ben}{\begin{eqnarray}}
\newcommand{\een}{\end{eqnarray}}
\newcommand{\beno}{\begin{eqnarray*}}
\newcommand{\eeno}{\end{eqnarray*}}

\numberwithin{equation}{section}

\begin{document}
\title[Stability of 3D Boussinesq system]{Stability threshold of Couette flow for 3D Boussinesq system in Sobolev spaces}
%Stability and transition threshold of the Couette flow for 3D Boussinesq system 
%in Sobolev space}
\author{Shikun~Cui}
\address[Shikun~Cui]{School of Mathematical Sciences, Dalian University of Technology, Dalian, 116024,  China}
\email{cskmath@163.com}
\author{Lili~Wang}
\address[Lili~Wang]{School of Mathematical Sciences, Dalian University of Technology, Dalian, 116024,  China}
\email{wanglili\_@mail.dlut.edu.cn}
\author{Wendong~Wang}
\address[Wendong~Wang]{School of Mathematical Sciences, Dalian University of Technology, Dalian, 116024,  China}
\email{wendong@dlut.edu.cn}
\date{\today}
\vspace*{-0.6cm}
\maketitle

\vspace{-0.4cm}

\begin{abstract}
In this paper,
we investigate  the nonlinear stability and transition threshold for the 3D Boussinesq system in Sobolev space under the high Reynolds number and small thermal
diffusion in $\mathbb{T}\times\mathbb{R}\times\mathbb{T} $.
It is proved that if the initial velocity $v_{\rm in}$ and the initial temperature $ \theta_{\rm in} $ satisfy $ \|v_{\rm in}-(y,0,0)\|_{H^{2}}\leq \varepsilon\nu, \|\theta_{\rm in}\|_{H^{2}}\leq \varepsilon\nu^{2} $, respectively for some $ \varepsilon>0 $ independent of  the Reynolds number or thermal
diffusion, then the solutions of 3D Boussinesq system are global in time.
\end{abstract}


\parskip5pt
\parindent=1.5em
\section{Introduction}
Consider the following 3D Boussinesq system in $ (x,y,z)\in\Omega=\mathbb{T}\times\mathbb{R}\times\mathbb{T} $:
\begin{equation}\label{ini}
\left\{
\begin{array}{lr}
	\partial_{t}v-\nu\triangle v+v\cdot\nabla v+\nabla p=\theta ge_{2}, \\
	\partial_{t}\theta-\mu\triangle\theta+v\cdot\nabla\theta=0,\\
	\nabla\cdot v=0,\\
	v|_{t=0}=v_{\rm in}(x,y,z),~~\theta|_{t=0}=\theta_{\rm in}(x,y,z), 
\end{array}
\right.
\end{equation}
where $ v=(v_{1}(t,x,y,z), v_{2}(t,x,y,z), v_{3}(t,x,y,z)) $ denotes the velocity of the fluid, $ p $ is the pressure, $ \theta $ is the temperature, $ g $ is the gravitational constant, $ e_{2}=(0,1,0) $ represents the unit vector, $ \nu $ is the viscosity coefficient, and $ \mu $ is the thermal diffusivity. For simplicity, we focus on $\nu=\mu$ and $g=1.$

The Boussinesq system is widely used to describe heat transfer phenomena and consist of a coupled system that combines the Navier-Stokes equations for fluid motion with a diffusion equation for temperature distribution, and it is also employed to simulate atmospheric and oceanic flows \cite{G1982,M2003,P1987}, as well as to model Rayleigh-B$\acute{{\rm e}}$nard convection in both laminar and turbulent flow regimes, as demonstrated in works such as \cite{CD1996,DC1996,G1998}. Due to its extensive applications in physics and its mathematical significance
the theory of Boussinesq equations has gained considerable attention.
Our aim is to investigate the stability threshold problem constructed by Bedrossian-Germain-Masmoudi \cite{Bedro1} as follows:
{\it Given a norm
$\|\cdot\|_X$, find a $\beta=\beta(X)$ so that
$$~\|u_{\rm in}\|_X\leq Re^{-\beta}\Longrightarrow {\rm stability},$$
$$\quad\|u_{\rm in}\|_X\gg Re^{-\beta}\Longrightarrow {\rm instability}.$$}
The exponent $\beta$ is referred to as the transition threshold in the applied literature. 

%Since the famous work of Reynolds in 1883 \cite{Re1883}, understanding the stability of shear flow in a stratified medium is a topic of interest in various fields, including fluid dynamics, geophysics, astrophysics, mathematics, and others.  This field is mainly concerned with how the laminar flows become unstable and transition to turbulence \cite{SH2001,Y2012}. In order to better understand the mechanism of transition, an important question firstly proposed by Trefethen et al. \cite{TTRD1993} is to study the transition threshold problem, which is concerned with how much disturbance will lead to the instability of the flow and the dependence of disturbance on the Reynolds number. 

If $ \theta=0 $, the system (\ref{ini}) is reduced to the Navier-Stokes equations. Significant progresses have been made on the stability of Couette flow in Navier-Stokes equations. At this time the possible transition threshold depends on the domain or $X$-norm.
%To quantify the stability threshold, the most commonly studied dynamical system is the NS equations. 
Let us briefly recall some developments on this topic.

{\it The 2D Navier-Stokes (NS) equations via Couette flow.}  In $ \mathbb{T}\times\mathbb{R}$, 
Bedrossian-Masmoudi-Vicol showed that if $X$ is taken as Gevrey-$m$ with $m < 2$ and $\beta=0$, then the solution is stable \cite{BMV2016}. Bedrossian-Vicol-Wang proved that Couette flow is stable in Sobolev norms for $\beta=\frac12$ by Fourier-multiplier method in \cite{BVW2018};  recently the sharp case of $\beta=\frac13$ was proved by  Masmoudi-Zhao in  \cite{MZ2022} and Wei-Zhang in  \cite{WZ2023}, respectively.
In $ \mathbb{T}\times[-1,1]$, $\beta=\frac12$ was proved by Chen-Li-Wei-Zhang in  \cite{CLWZ2020} with no-slip boundary data by establishing some sharp resolvent estimates for the linearized operator under the Navier-slip boundary condition and non-slip boundary condition.
%  and  $\beta=\frac13 $ was proved by Wei-Zhang in \cite{WZ0} by a new method of proving the optimal enhanced dissipation and inviscid damping estimates. 

{\it The 3D Navier-Stokes (NS) equations via Couette flow.}  In $ \mathbb{T}\times\mathbb{R}\times\mathbb{T}$, Fourier multipliers that precisely encode the interplay between the dissipation and possible growth are designed, and $\beta=\frac32$ was proved by  Bedrossian-Germain-Masmoudi in  \cite{Bedro1} for Sobolev spaces; $\beta=1$ was proved by the same authors in  \cite{BGM1} for Gevrey  spaces. Later, $\beta=1$ was proved by Wei-Zhang in  \cite{wei2} for Sobolev spaces. In $ \mathbb{T}\times[-1,1]\times\mathbb{T}, $   $\beta=1$ was proved by Chen-Wei-Zhang in  \cite{Chen1} for the results of 3D Couette flow.
For Poiseuille flow, Kolmogorov flow and other shear flows, we refer to \cite{ZEW2020, ZOTTO2023, DL2022, CDLZ,WZZ2020,LWZ2020,LZ2024,BH2024} and the references therein.

%One may refer to \cite{,MZ2022,WZ2023} for the results of 2D Couette flow in $ \mathbb{T}\times\mathbb{R}, $ to for the results of 2D Couette flow in $ \mathbb{T}\times[-1,1], $ to \cite{Bedro1,BGM1,BGM2,wei2} for the results of 3D Couette flow in $ \mathbb{T}\times\mathbb{R}\times\mathbb{T}, $ and to

{\it The 2D Boussinesq equations via Couette flow.}
Recently, the impact of the temperature field in the Boussinesq equations on the mixing mechanism of Couette flow has attracted some attentions.  Masmoudi, Said-Houari and Zhao \cite{MSZ2022} proved the stability of the Couette flow in $ \mathbb{T}\times\mathbb{R} $ without thermal diffusivity for the initial perturbation in Gevrey-$ \frac{1}{s}, (\frac13<s\leq 1). $ Bedrossian-Bianchini-Coti Zelati-Dolce \cite{BBZD2023} investigated the long-time properties of the two-dimensional inviscid Boussinesq equations in $ \mathbb{T}\times\mathbb{R} $ near a stably
stratified Couette flow for an initial Gevrey perturbation.
Deng-Wu-Zhang \cite{DWZ2021} proved that the solution with only 
vertical dissipation on $ \mathbb{T}\times\mathbb{R} $ remains close to the Couette at the same order provided the initial perturbation 
from the Couette flow is no more than the viscosity to 
a suitable power (in the Sobolev space $ H^{b} $ with $ b>\frac43$, $\beta=\frac23$).
Zhang-Zi \cite{Zhangzi2023} studied the stability near  the Couette flow $ v_{s}=(y,0)^{T} $ and $\theta_{s}=1 $, and showed that if 
$$ \|u_{\rm in}\|_{H^{N+1}}+\nu^{-\frac12}\|\theta_{\rm in}\|_{H^{N}}+\nu^{-\frac13}\||\partial_{x}|^{\frac13}\theta\|_{H^{N}}\lesssim\nu^{\frac13},\quad N>7 ,\quad{\rm in }\quad\mathbb{T}\times\mathbb{R}, $$
then solution is global in time.
%Zhang-Zi \cite{Zhangzi2023} showed that if the initial perturbations $ u_{\rm in} $ and
%$ \theta_{\rm in} $ to the Couette flow $ v_{s}=(y,0)^{T} $ and $\theta_{s}=1 $, respectively, satisfy 
%$ \|u_{\rm in}\|_{H^{N+1}}+\nu^{-\frac12}\|\theta_{\rm in}\|_{H^{N}}+\nu^{-\frac13}\||\partial_{x}|^{\frac13}\theta\|_{H^{N}}<<\nu^{\frac13}, N>7 $ in $ \mathbb{T}\times\mathbb{R} $, then the resulting solution remains close to the Couette flow in $ L^{2} $ at the same order for all time. 
Niu-Zhao \cite{NZ2024} improved the size of temperature from $ \nu^{\frac56} $ to $ \nu^{\frac23}$  in the paper of Zhang-Zi \cite{Zhangzi2023} by using the quasi-linearization.
Zhai-Zhao \cite{ZZ2023} studied the nonlinear asymptotic stability of the Couette flow in the stably stratified regime on $ \mathbb{T}\times\mathbb{R} $ with the Richardson number greater than $ \frac14. $ 
Arbon \cite{A2025} considered the quantitative asymptotic stability of the stably stratified Couette flow solution to the fully dissipative nonlinear Boussinesq system on $ \mathbb{R}^{2} $ with large Richardson number greater that $ \frac14 $. Moreover, we refer to \cite{WGF2024} for the result of Poiseuille flow.

{\it The 3D Boussinesq equations via Couette flow.}
For 3D Boussinesq equations,  Coti Zelati-Del Zotto   \cite{CotiDel2025} established linear enhanced dissipation results for the three-dimensional Boussinesq
equations around a stably stratified Couette flow by introducing a change of variables grounded in a Fourier space symmetrization
framework in $ \mathbb{T}\times\mathbb{R}\times\mathbb{T}$. Furthermore, 
Coti Zelati-Del Zotto-Widmayer \cite{Zelati2024} established a bound for the nonlinear
transition threshold, which is quantitatively larger than the inverse Reynolds number,  around a stably stratified Couette flow of 
\begin{equation*}
v_{s}=(y,0,0),\quad \partial_{y}p_{s}=g(1+\alpha y), \quad \theta_{s}=1+\alpha y, \quad\alpha>0 .
\end{equation*}
%\ben\label{eq: stratified Couette}
%v_{s}=(y,0,0),\quad \partial_{y}p_{s}=g(1+\alpha y), \quad \theta_{s}=1+\alpha y, \quad\alpha>0 .
%\een
Introduce $v = v_s + u,~\theta= \theta_s-\sqrt{\frac{\alpha}{g}}\Theta$, then 
the perturbations $(u, \Theta)$  in (\ref{ini}) satisfy
\begin{equation*}\label{ini00}
\left\{
\begin{array}{lr}
	\partial_{t}u+y\partial_x u+u_2e_1-\nu\triangle u+u\cdot\nabla u+\nabla p=-\gamma\Theta ge_{2}, \\
	\partial_{t}\Theta-\gamma u_2-\mu\triangle\Theta+v\cdot\nabla\Theta=0,\\
	\nabla\cdot u=0,
\end{array}
\right.
\end{equation*}
where $\gamma=\sqrt{\alpha g}$ is the Brunt-V\"{a}is\"{a}l\"{a} frequency, reflecting the strength of the response of the fluid to
displacements in the direction of gravity.  
By exploring the coupling between velocity and temperature, the 3D lift-up effect can be suppressed.
In particular, they quantify the size of the
stability transition as $\beta=\frac{11}{12}<1$ when $\gamma>\frac12.$
More references on Navier-Stokes equations with rotation, MHD model and other models, we refer to \cite{LISS2020,HSX2024-1,HSX2024-2,ZZZ2022,{RZZ2025}, CotiDelW2025} and the references therein. 

Motivated by \cite{Zelati2024}, we are interested  the case of $\alpha=0,$ since the symmetrization
framework fails  and  the 3D lift-up effect takes effect at this time. 

%The choice of sign α > 0 hereby assures that with respect to the direction of gravity, warmer fluid is on top of colder fluid.

%\textcolor[rgb]{0,0,0}{
%	Bedrossian-Germain-Masmoudi proved $\beta\leq \frac{3}{2}$ for the 3D Couette flow \cite{Bedro1} in Sobolev space and $\beta\leq 1$ in Gevrey class \cite{BGM2022}.}
%In Sobolev space, we refer to Wei-Zhang \cite{wei2} and Chen-Wei-Zhang \cite{Chen1} for  recent results of $\beta\leq 1$. More references on MHD, Boussinesq equations or other models, we refer to \cite{HSX2024-1,HSX2024-2,LISS2020,NZ2024,ZZZ2022} and the references therein.
%For the stability of the 2D Navier-Stokes equations and related models, there are  very rich research progress on this topic, and we refer to \cite{BMV2016,BVW2018,WZ2023,MZ2022,DWZ2021,BHIW2023} and the references therein.

%The purpose of this paper is to understand the stability and large-time behavior of perturbations near the Couette flow.
In details, let
\begin{equation*}
\tilde{v}=(y,0,0),\quad \tilde{p}=y+c,\quad \tilde{\theta}=0,
\end{equation*}
be a stationary solution of (\ref{ini}), and we introduce the perturbations
$$u=v-\tilde{v},\quad P=p-\tilde{p},\quad \Theta=\theta-\tilde{\theta}, $$  satisfying $u\big|_{t=0}=u_{\rm in}=(u_{1,\rm in}, u_{2,\rm in}, u_{3,\rm in})$ and $ \Theta|_{t=0}=\Theta_{\rm in}$. Then rewrite the Boussinesq equations (\ref{ini}) into
\begin{equation}\label{ini1}
\left\{
\begin{array}{lr}
	\partial_tu-\nu\triangle u+y\partial_{x}u+\left(
	\begin{array}{c}
		u_2 \\
		0 \\
		0 \\
	\end{array}
	\right)
	+u\cdot\nabla u+\nabla P^{N_1}+\nabla P^{N_2}+\nabla P^{N_{3}}=\left(
	\begin{array}{c}
		0 \\
		\Theta \\
		0 \\
	\end{array}
	\right), \\
	\partial_{t}\Theta-\nu\triangle\Theta+y\partial_{x}\Theta+u\cdot\nabla\Theta=0,
	\\
	\nabla \cdot u=0,
\end{array}
\right.
\end{equation}
where the pressure $P^{N_1}$, $ P^{N_{2}} $ and $P^{N_3}$ are determined by
\begin{equation}\label{pressure_1}
\left\{
\begin{array}{lr}
	\triangle P^{N_1}=-2\partial_xu_2, \\
	\triangle P^{N_2}=-{\rm div}~(u\cdot\nabla u),\\
	\triangle P^{N_{3}}=\partial_{y}\Theta.
\end{array}
\right.
\end{equation}

Our main result is stated as follows.
\begin{theorem}\label{main result}
Assume that $u_{\rm in}\in H^{2}(\mathbb{T}\times\mathbb{R}\times\mathbb{T})$ and $ \Theta_{\rm in}\in H^{2}(\mathbb{T}\times\mathbb{R}\times\mathbb{T}). $ There are constants $ \varepsilon>0, \nu\in(0,1), $ such that if $ \|u_{\rm in}\|_{H^{2}}\leq \varepsilon\nu$ and $ \|\Theta_{\rm in}\|_{H^{2}}\leq \varepsilon\nu^{2}, $ then the solution $ (u,\Theta)$ to the system (\ref{ini1})-(\ref{pressure_1}) with the initial data $(u_{\rm in}, \Theta_{\rm in} )$ is global in time. 
\end{theorem}

\begin{remark} For the stratified Couette flow considered by 
Coti Zelati-Del Zotto-Widmayer in \cite{Zelati2024}, which shows that warmer fluid is on top of colder fluid for $\alpha>0$, the instability effect of the lift-up term can be well suppressed. It is difficult to apply 
the similar symmetrization
framework to overcome the lift-up effect for the perturbations under the constant temperature. Here we use the quasi-linearization technique introduced by Wei-Zhang in \cite{wei2}, and the main obstacle comes from some coupled terms. We introduce some new good unknowns  and the decomposition of $Q=u_{2,\neq}+\kappa u_{3,\neq}$ (see (\ref{Q}) for some details) that effectively
capture the phenomena of enhanced dissipation and inviscid damping to close the energy estimates.
\end{remark}

\begin{remark} The system  (\ref{ini1}) is reduced to the Navier-Stokes system when $\Theta=0$. 
Through a formal asymptotic
analysis of the Navier–Stokes equations, it was found in \cite{Cha2002} that for
streamwise initial perturbations and oblique initial perturbations, $\beta=1$, which is rigorously proven in \cite{wei2}.
In fact, there will be an  increase at a level of $O(\nu^{-1})$ for the first velocity $u_{1,0}$ due to the 3D lift-up effect.
For the current Boussinesq system and $ \|u_{\rm in}\|_{H^{2}}\leq \varepsilon\nu$, it follows from the zero-mode estimates (for example, see (\ref{U20 U30 end})) that
\begin{equation*}
	\begin{aligned}
		\partial_{t}\left(\|u_{2,0}\|_{L^{2}}^{2}+\|u_{3,0}\|_{L^{2}}^{2} \right)+\nu\left(\|\nabla u_{2,0}\|_{L^{2}}^{2}+\|\nabla u_{3,0}\|_{L^{2}}^{2} \right)
		\leq C\nu^{-1}\left(\||u_{\neq}|^{2}\|_{L^{2}}^{2}+\|\nabla\Theta_{0}\|_{L^{2}}^{2} \right),
	\end{aligned}
\end{equation*}
which implies that 
\beno 
\nu^{-1}\|\nabla\Theta_{0}\|_{L^2_tL^{2}}^{2}\sim \nu^{-2}\|\Theta_{0}\|_{L^\infty_tL^{2}}^{2}\sim \|(u_{2,0},u_{3,0})\|_{L^\infty_tL^{2}}^{2}\sim \nu^2
\eeno
and it should be  $ \|\Theta_{\rm in}\|_{H^{2}}\leq \varepsilon\nu^{2}. $ 
Therefore, it seems that the assumptions of $\|u_{\rm in}\|_{H^{2}}\leq \varepsilon\nu$ and $\|\Theta_{\rm in}\|_{H^{2}}\leq \varepsilon\nu^{2}$ are optimal. 

%which along with {Lemma \ref{lem 2}} imply that
%	\begin{equation}\label{U20U30''}
	%		\begin{aligned}
		%			\|u_{2,0}\|_{Y_{0}}^{2}+\|u_{3,0}\|_{Y_{0}}^{2}
		%			&\leq C\left(\|(u_{\rm in})_0\|_{L^{2}}^{2}+\nu^{-2}E_{4}^{4}+
		%			\nu^{-2}E_{3}^{2}\right).
		%		\end{aligned}
	%	\end{equation}
\end{remark}	

\begin{remark}
Here,  we consider an infinite domain $(x,y,z)\in\mathbb{T}\times\mathbb{R}\times\mathbb{T} $ to avoid the
boundary effect. Thus, it remains open whether the transition threshold conjecture
holds for the Boussinesq system in the finite domain. Recently, Chen-Wei-Zhang \cite{Chen1} developed the resolvent estimates method
and proved that transition threshold for the 
Couette flow in the finite domain $\mathbb{T}\times\mathbb{I}\times\mathbb{T}$ is $\beta=1.$ It's interesting that whether $\beta<1$ holds for the stratified Couette flow as in \cite{Zelati2024}.
%We are optimistic that using their method will effectively solve this problem.
\end{remark}

%\begin{remark}
%Compared with 
%
%\end{remark}

Here are some notations used in this paper.

\noindent\textbf{Notations}:
\begin{itemize}
%\item Summation notation is assumed: the repea ted upper and lower indices are summed over $ i,j\in\{1,2,3\} $ and $ \alpha, \beta\in\{2,3\}. $

\item The Fourier transform is defined by
\begin{equation*}
	f(t,x,y,z)=\sum_{k_{1},k_{3}\in\mathbb{Z}}\frac{1}{2\pi}\int_{k_{2}\in\mathbb{R}}\widehat{f}_{k_{1},k_{2},k_{3}}(t)e^{ik_{2}y}dk_{2}e^{i(k_{1}x+k_{3}z)},
\end{equation*}
where $\widehat{f}_{k_1,k_{2},k_3}(t)=\frac{1}{|\mathbb{T}|^2}\int_{\mathbb{T}\times\mathbb{T}}\int_{\mathbb{R}}{f}(t,x,y,z)e^{-ik_{2}y}dy{\rm e}^{-i(k_1x+k_3z)}dxdz.$

\item 
For a given function $f=f(t,x,y,z)$,  the zero and non-zero modes are denoted by
$$P_0f=f_0=\frac{1}{|\mathbb{T}|}\int_{\mathbb{T}}f(t,x,y,z)dx \ {\rm and}\ P_{\neq}f=f_{\neq}=f-f_0.$$
Especially, we use $u_{k,0}$ and $u_{k,\neq}$ to represent the zero mode
and non-zero mode of the velocity $u_k(k=1,2,3)$, respectively.
Similarly, we use $\omega_{k,0}$ and $\omega_{k,\neq}$ to represent the zero mode
and non-zero mode of the vorticity $\omega_k (k=1,2,3)$.

\item The norm of the $L^p$ space is defined by
$\|f\|_{L^p(\mathbb{T}\times\mathbb{R}\times\mathbb{T})}=\big(\int_{\mathbb{T}\times\mathbb{R}\times\mathbb{T}}|f|^p dxdydz\big)^{\frac{1}{p}},$
and $\langle\cdot,\cdot\rangle$ denotes the standard $L^2$ scalar product.
\item The time-space norm  $\|f\|_{L^qL^p}$ is defined by
$\|f\|_{L^qL^p}=\big\|  \|f\|_{L^p(\mathbb{T}\times\mathbb{R}\times\mathbb{T})}\ \big\|_{L^q(0,t)}.$
\item For $ a\geq 0 $ and $ k\in\mathbb{N^{+}}, $ we introduce the following norms 
\begin{equation*}
	\begin{aligned}
		&\|f\|_{Y_{0}}^2=\|f\|^2_{L^{\infty}L^{2}}+\nu\|\nabla f\|^2_{L^{2}L^{2}},\\
		&\|f\|_{Y_{0}^{k}}^{2}=\|f\|_{L^{\infty}H^{k}}^{2}+\nu\|\nabla f\|_{L^{2}H^{k}}^{2},
	\end{aligned}
\end{equation*}
and
\begin{equation*}
	\begin{aligned}
		\|f\|_{X_{a}}^2
		=\|{\rm e}^{a\nu^{\frac{1}{3}}t}f\|^2_{L^{\infty}L^{2}}+\|e^{a\nu^{\frac13}t}\nabla\triangle^{-1}\partial_{x}f\|_{L^{2}L^{2}}^{2}
		+\nu^{\frac{1}{3}}\|{\rm e}^{a\nu^{\frac{1}{3}}t}f\|^2_{L^{2}L^{2}}
		+\nu\|{\rm e}^{a\nu^{\frac{1}{3}}t}\nabla f\|^2_{L^{2}L^{2}}.
	\end{aligned}
\end{equation*}

\item Throughout this paper, we denote by $ C $ a positive constant independent of $\nu$, $t$,  $x$, $y$, $z$ and the initial data, and it may be different from line to line.
\end{itemize}

\section{Key ingredients and  proof of Theorem \ref{main result}}
\subsection{Reformulation of the perturbation system }
%Let $ \mathcal{L}:=\partial_{t}-\nu\triangle+y\partial_x, $ it follows from (\ref{ini1}) that
%\begin{equation*}
%\left\{
%\begin{array}{lr}
%\mathcal{L}u+\left(
%\begin{array}{c}
%u_{2} \\
%0 \\
%0 \\
%\end{array}
%\right)+u\cdot\nabla u+\nabla P^{N_{1}}+\nabla P^{N_{2}}+\nabla P^{N_{3}}=\left(
%\begin{array}{c}
%0 \\
%\Theta \\
%0 \\
%\end{array}
%\right),\\\mathcal{L}\Theta+u\cdot\nabla\Theta=0,
%\end{array}
%\right.
%\end{equation*}
%and the pressure $ P^{N_{1}}, $ $ P^{N_{2}} $ and $ P^{N_{3}} $ are given by
%\begin{equation*}
%\triangle P^{N_{1}}=-2\partial_{x}u_{2},~~\triangle P^{N_{2}}=-{\rm div}(u\cdot\nabla u)=-\partial_{i}u_{j}\partial_{j}u_{i},~~\triangle P^{N_{3}}=\partial_{y}\Theta.
%\end{equation*}
Firstly, inspired by Chen-Wei-Zhang's work \cite{Chen1}, we decompose 
\begin{equation*}
u_{1,0}=\widehat{u_{1,0}}+\widetilde{u_{1,0}},
\end{equation*}
which satisfies
\begin{equation}\label{decom u10}
\left\{
\begin{array}{lr}
	\partial_{t}\widehat{u_{1,0}}-\nu\triangle\widehat{u_{1,0}}+u_{2,0}+u_{2,0}\partial_{y}\widehat{u_{1,0}}+u_{3,0}\partial_{z}\widehat{u_{1,0}}=0, \\
	\partial_{t}\widetilde{u_{1,0}}-\nu\triangle\widetilde{u_{1,0}}+u_{2,0}\partial_{y}\widetilde{u_{1,0}}+u_{3,0}\partial_{z}\widetilde{u_{1,0}}+(u_{\neq}\cdot\nabla u_{1,\neq})_{0}=0,\\
	\widehat{u_{1,0}}|_{t=0}=0,\quad \widetilde{u_{1,0}}|_{t=0}=(u_{1,{\rm in}})_0.
\end{array}
\right.
\end{equation}
In this way, $\widetilde{u_{1,0}}$ will not be affected by the 3D lift-up effect and $\widetilde{u_{1,0}}$ has a better decay, thus $ \widetilde{u_{1,0}}\partial_{x} $ could be viewed as a perturbation.

The nonlinear pressure $ P^{N_{2}} $ can be decomposed into five parts as in \cite{wei2}:
\begin{equation}\label{P_N2}
P^{N_{2}}=P^{0}+P^{1}+P^{2}+P^{3}+P^{4},
\end{equation}
where
\begin{equation}\label{P_N2 fenjie}
\left\{
\begin{array}{lr}
	\triangle P^{0}=-2\left(\partial_{y}\widetilde{u_{1,0}}\partial_{x}u_{2,\neq}+\partial_{z}\widetilde{u_{1,0}}\partial_{x}u_{3,\neq} \right),
	\\
	\triangle P^{1}=-2\left(\partial_y \widehat{u_{1,0}}\partial_x u_{2,\neq}+\partial_z \widehat{u_{1,0}}\partial_x u_{3,\neq} \right),
	\\\triangle P^{2}=-\partial_{i}u_{j,0}\partial_{j}u_{i,0}, \\
	\triangle P^{3}=-2\partial_{\alpha}u_{\beta,0}\partial_{\beta}u_{\alpha,\neq},\quad\alpha, \beta\in\{2,3\},\\
	\triangle P^{4}=-\partial_{i}u_{j,\neq}\partial_{j}u_{i,\neq}.
\end{array}
\right.
\end{equation}
For the above decomposition, $ P^{2} $ corresponds to the interaction between zero modes, $ P^{4} $ corresponds to the interaction between non-zero modes, and $ P^{0}, P^{1}, P^{3} $ correspond to the interaction between zero and non-zero modes where the pressure $P^{1}$ contains the worse interaction due to the 3D lift-up effect. 
We denote
\begin{equation*}
P^{5}=P^{N_{1}}+P^{1},\quad V(t,y,z)=y+\widehat{u_{1,0}}(t,y,z).
\end{equation*}
As $ \partial_{y}V=1+\partial_{y}\widehat{u_{1,0}} $ and $ \partial_{z}V=\partial_{z}\widehat{u_{1,0}}, $ one  obtains 
\begin{equation}\label{p5}
\begin{aligned}
	\triangle P^{5}=&-2\partial_{x}u_{2}-2\left( \partial_{y}\widehat{u_{1,0}}\partial_{x}u_{2,\neq}+\partial_{z}\widehat{u_{1,0}}\partial_{x}u_{3,\neq}\right)=-2\left(\partial_{y}V\partial_{x}u_{2,\neq}+\partial_{z}V\partial_{x}u_{3,\neq} \right).
\end{aligned}
\end{equation}
Direct calculations indicate that
\begin{equation}\label{23}
\left\{
\begin{array}{lr}
	\partial_{t}u_{2}-\nu\triangle u_{2}+V\partial_{x}u_{2}+\partial_{y}P^{5}+h_{2}+\widetilde{u_{1,0}}\partial_{x}u_{2}=\Theta-\partial_{y}P^{N_{3}},\\\partial_{t}u_{3}-\nu\triangle u_{3}+V\partial_{x}u_{3}+\partial_{z}P^{5}+h_{3}+\widetilde{u_{1,0}}\partial_{x}u_{3}=-\partial_{z}P^{N_{3}},
\end{array}
\right.
\end{equation}
where $h_{j}$ can be viewed as good terms satisfying
\begin{equation}\label{hj}
h_{j}=\left(u_{2,0}\partial_{y}+u_{3,0}\partial_{z} \right)u_{j}+u_{\neq}\cdot\nabla u_{j}+\partial_{j}(P^{0}+P^{2}+P^{3}+P^{4})\quad{\rm for}~~j=2,3.
\end{equation}
Next, we denote
\begin{equation*}
\mathcal{L}_{V}:=\partial_{t}-\nu\triangle+V\partial_{x},
\end{equation*}
which can be viewed as a perturbation of $ \mathcal{L}=\partial_{t}-\nu\triangle+y\partial_{x} $ under the condition
\begin{equation*}
\|\widehat{u_{1,0}}\|_{H^{4}}+\nu^{-1}\|\partial_{t}\widehat{u_{1,0}}\|_{H^{2}}<\delta
\end{equation*}
for some small constant $ \delta. $ Then it follows from (\ref{23}) that
\begin{equation}\label{U_neq}
\left\{
\begin{array}{lr}
	\mathcal{L}_{V}u_{2,\neq}+\partial_{y}P^{5}+h_{2,\neq}+{\widetilde{u_{1,0}}\partial_{x}u_{2,\neq}}=\Theta_{\neq}-\partial_{y}P_{\neq}^{N_{3}},\\
	\mathcal{L}_{V}u_{3,\neq}+\partial_{z}P^{5}+h_{3,\neq}+{\widetilde{u_{1,0}}\partial_{x}u_{3,\neq}}=-\partial_{z}P^{N_{3}}_{\neq}.
\end{array}
\right.
\end{equation}
%Understanding the nonlinear interactions between different modes is crucial in the computational process. 
%For the nonlinear terms $ h_{j} $ in (\ref{hj}), we further write $ h_{j}=\sum_{k=1}^{7}h_{j,k}, $ where
%$$ h_{j,1}=\left(u_{2,0}\partial_{y}+u_{3,0}\partial_{z} \right)u_{j},\quad h_{j,2}=u_{\neq}\cdot\nabla u_{j,0},\quad h_{j,3}=u_{\neq}\cdot\nabla u_{j,\neq}, $$
%$$h_{j,k+2}=\partial_{j}P^{k}\quad {\rm for}\quad k\in\{2,3,4\},\quad h_{j,7}=\partial_{j}P^{0}.$$
%Due to $ P_{\neq}^{2}=0, $ we have $ (h_{j,4})_{\neq}=\partial_{j}P_{\neq}^{2}=0. $
% Besides, for simplicity,  we let $ H_{2}=(h_{2}+\kappa h_{3})_{\neq} $ and denote
%\begin{equation*}
%H_{2}=\sum_{j=1}^{7}H_{2,j},\quad H_{2,j}=(h_{2,j}+\kappa h_{3,j})_{\neq}.
%\end{equation*}

\subsection{Construction of energy functional}
We introduce the following energy functional of the system (\ref{ini1}):
\begin{equation*}
\begin{aligned}
	&E_{1,1}=\|\widehat{u_{1,0}}\|_{L^{\infty}H^{4}}+\nu^{\frac12}\|\nabla\widehat{u_{1,0}}\|_{L^{2}H^{4}}+\nu^{-1}\|\partial_{t}\widehat{u_{1,0}}\|_{L^{\infty}H^{2}},
	\\&E_{1,2}=\|\widetilde{u_{1,0}}\|_{L^{\infty}H^{2}}+\nu^{\frac12}\|\nabla\widetilde{u_{1,0}}\|_{L^{2}H^{2}},
	\\&E_{2}=\|\triangle u_{2,0}\|_{Y_{0}}+\|u_{3,0}\|_{Y_{0}}+\|\nabla u_{3,0}\|_{Y_{0}}+\|\min(\nu^{\frac23}+\nu t, 1)^{\frac12}\triangle u_{3,0}\|_{Y_{0}},
	\\&E_{3}=\|\Theta_{0}\|_{Y_{0}}+\|\nabla\Theta_{0}\|_{Y_{0}}+\|\partial_{z}\nabla\Theta_{0}\|_{Y_{0}},
	\\&E_{4,1}=\|\triangle u_{2,\neq}\|_{X_a}+\|(\partial_x^2+\partial_z^2) u_{3,\neq}\|_{X_a},
	\\&E_{4,2}=\nu^{\frac13}\big(\|e^{a\nu^{\frac13}t}\nabla\omega_{2,\neq}\|_{L^{\infty}L^{2}}
	+\nu^{\frac12}\|e^{a\nu^{\frac13}t}\triangle\omega_{2,\neq}\|_{L^{2}L^{2}} \big),
	\\&E_{5}=\|\partial_{x}^{2}\Theta_{\neq}\|_{X_{b}}+\|\partial_{z}^{2}\Theta_{\neq}\|_{X_{a}},
	\\&E_{6}=\|\partial_{x}^{2}u_{2,\neq}\|_{X_{b}}+\|\partial_{x}^{2}u_{3,\neq}\|_{X_{b}},
\end{aligned}	
\end{equation*}
where $ a $ and $ b $ are positive constants with $ 0<a<b<2a $. Moreover, we denote 
\begin{equation*}
\begin{aligned}
	&E_{1}=E_{1,1}+\nu^{-\frac23}E_{1,2}, \quad E_{4}=E_{4,1}+E_{4,2}.
\end{aligned}
\end{equation*}

\begin{remark} Let us give some explanations about the energy functional:

$\bullet$ $ E_{1} $ is introduced to control the zero mode $ u_{1,0} $. Due to the 3D lift-up effect, $ E_{1} $ is expected to be $ o(1) $, which implies 
$ \|({u}_{1,\rm in})_0\|_{H^{2}}\leq \varepsilon_{0}\nu^{\frac23}$ is enough in Theorem \ref{main result} for the initial data of the component $ u_{1,0}$.

$ \bullet $ $ E_{2} $ is introduced to control good components $ u_{2,0}, u_{3,0}. $ Since there is no lift-up in the equations of $ u_{2,0} $ and $ u_{3,0}, $ $ E_{2} $ is expected to be $ o(\nu). $

$ \bullet $ $ E_{3} $ is introduced to control the zero mode of $ \Theta $. Due to the linear transfer mechanism  of energies between $ u_{2,0} $ and $ \Theta_{0} $, $ E_{3} $ is expected to be $ o(\nu^{2}). $

$ \bullet $ $ E_{4,1} $ is introduced to control good components $ \triangle u_{2,\neq} $ and $ (\partial_{x}^{2}+\partial_{z}^{2})u_{3,\neq}$, which is more simplified than that in \cite{wei2}, since there is an additional term $\nu^{\frac23}\|\triangle u_{3,\neq}\|_{X_b}$.  $ E_{4,2} $ is introduced to control the energy $E_{1,2}$. $ E_{4} $ is expected to be $ o(\nu). $

$\bullet$ $E_5$ is introduced to control the non-zero modes of $\Theta$ and control the energy $E_3$, and it is expected to be $ o(\nu^{2}). $
It is noted that the time weight of $\partial_{x}^{2}\Theta_{\neq}$ is $b$ rather than $a,$ which is an ingredient for 
dealing with 3D lift-up effect in the temperature $\Theta.$

$\bullet$ $E_6$ is introduced to deal with the 3D lift-up effect in the velocity $u$ and is expected to be $ o(\nu).$
We also need to note that the time weight of $\partial_{x}^{2}u_{2,\neq}$ and $\partial_{x}^{2}u_{3,\neq}$ is also $b$ rather than $a.$ 
\end{remark}
%
%
%$\bullet$ \underline{Estimate $ E_{6}. $} 

\subsection{Proof of Theorem \ref{main result}}
First, we define $T$ to be the largest time such that the following bootstrap assumptions hold:
\begin{equation}\label{ass}
E_{1}\leq\varepsilon_{0},~~ E_{2}\leq\varepsilon_{0}\nu,~~ E_{3}\leq \varepsilon_{0}\nu^{2},~~  E_{4}\leq\varepsilon_{0}\nu,~~ E_{5}\leq \varepsilon_{0}\nu^{2},
~~  E_{6}\leq\varepsilon_{0}\nu,
\end{equation}
where $ \varepsilon_{0} $ is determined later. 

Using  Prop. \ref{prop:E1}, Prop. \ref{prop:E21}, Prop. \ref{prop:E22}, Prop. \ref{prop:E3},  Prop. \ref{propE4} and Prop. \ref{E5 E6},
we can obtain that 
\begin{equation*}
	\begin{aligned}
			&E_{1}\leq C\big(\nu^{-\frac23}\|u_{\rm in}\|_{H^{2}}+\nu^{-1}E_{2}+\nu^{-2}E_{4}^{2} \big),\\
			&E_{2}\leq C\left(\|u_{\rm in}\|_{H^{2}}+\nu^{-1}E_{4}^{2}+\nu^{-1}E_{3} \right),\\
			&E_{3}\leq C\left(\|\Theta_{\rm in}\|_{H^{2}}+\nu^{-1}E_{2}E_{3}+\nu^{-1}E_{4}E_{5} \right),\\
			&E_{4}\leq C\left(\|u_{\rm in}\|_{H^{2}}+\nu^{-1}E_{4}^{2}+\nu^{-1}E_{5}\right),\\
			&E_{5}\leq C\big(\|\Theta_{\rm in}\|_{H^{2}}+\nu^{-\frac23}E_{3}E_{6}+\nu^{-1}E_{3}E_{4}\big),\\
			&E_6\leq C\left(\|u_{\rm in}\|_{H^{2}}+\nu^{-1}E_{4}^{2}+\nu^{-1}E_{5}\right).
	\end{aligned}
\end{equation*}
It is straightforward to observe that by choosing a sufficiently small $\varepsilon_0$, 
we conclude the following bounds for $t\in[0,T]:$
$$E_{3}\leq C\varepsilon\nu^{2},\quad E_{5}\leq C\varepsilon\nu^{2}.$$
Moreover, we also have
\begin{equation*}
E_{2}\leq C\varepsilon\nu,\quad E_{4}\leq C\varepsilon\nu, \quad E_{6}\leq C\varepsilon\nu.
\end{equation*}
Finally, there holds
\begin{equation*}
E_{1}\leq C\varepsilon.
\end{equation*}
By setting $\varepsilon$ small enough satisfying $C\varepsilon\leq \frac{\varepsilon_0}{2},$
we obtain that 
\begin{equation*}
	E_{1}\leq\frac{\varepsilon_{0}}{2},~~ E_{2}\leq\frac{\varepsilon_{0}}{2}\nu,~~ 
	E_{3}\leq \frac{\varepsilon_{0}}{2}\nu^{2},~~E_{4}\leq\frac{\varepsilon_{0}}{2}\nu,~~ 
	E_{5}\leq \frac{\varepsilon_{0}}{2}\nu^{2},~~E_{6}\leq\frac{\varepsilon_{0}}{2}\nu.
\end{equation*}
The argument as above implies that $T=+\infty.$

%Then, our aim is to improve the above assumptions:
%for $t\in[0,T],$ there are
%$$ E_{1}\leq \frac{\varepsilon_{0}}{2},~~ E_{2}\leq \frac{\varepsilon_{0}}{2}\eta,~~ E_{3}\leq \frac{\varepsilon_{0}}{2}\eta^{2},~~ 
%E_{4}\leq \frac{\varepsilon_{0}}{2}\eta,~~ E_{5}\leq \frac{\varepsilon_{0}}{2}\eta^{2},~~ E_{6}\leq \frac{\varepsilon_{0}}{2}\eta.$$

\section{The velocity estimates in terms of the energy}

In this section, we give some embedding inequalities controlled by energy functional $E_1, \cdots, E_6$, which will be frequently used below.

\begin{lemma}\label{lem omega}
For $k\geq 0$, it holds that
\begin{equation}\label{estimate_u}
	\begin{aligned}
		&\|\nabla^{k}(\partial_{x},\partial_{z})\partial_{x}u_{\neq}\|_{L^{2}}
		\leq C\left(\|\nabla^{k}(\partial_{x}^{2}+\partial_{z}^{2})u_{3,\neq}\|_{L^{2}}+\|\nabla^{k}\triangle u_{2,\neq}\|_{L^{2}}\right),\\
		&\|e^{\frac{a+b}{2}\nu^{\frac13}t}\partial_{x}u_{\neq}\|_{Y_{0}}^{2}+\|e^{\frac{a+b}{2}\nu^{\frac13}t}\partial_{z}u_{3,\neq}\|_{Y_{0}}^{2}\leq~CE_{4}E_{6},\\	&\|\triangle(\partial_{x},\partial_{z})u_{\neq}\|_{L^{2}}\leq C\left(\|\triangle\omega_{2,\neq}\|_{L^{2}}+\|\nabla\triangle u_{2,\neq}\|_{L^{2}}\right).
	\end{aligned}
\end{equation}
\end{lemma}
\begin{proof}
\underline{{\bf Estimate $ (\ref{estimate_u})_{1}. $}} Notice that $ {\rm{div}}~u_{\neq}=\partial_{x}u_{1,\neq}+\partial_{y}u_{2,\neq}+\partial_{z}u_{3,\neq}=0, $ then for any $ k\geq 0, $ we get
\begin{equation*}
	\begin{aligned}
		&\|\nabla^{k}(\partial_{x},\partial_{z})\partial_{x}u_{\neq}\|_{L^{2}}\leq&C\left(\|\nabla^{k}(\partial_{x}^{2}+\partial_{z}^{2})u_{3,\neq}\|_{L^{2}}+\|\nabla^{k}\triangle u_{2,\neq}\|_{L^{2}} \right).
	\end{aligned}
\end{equation*}

\underline{{\bf Estimate $ (\ref{estimate_u})_{2}.$}} Due to $ {\rm div}~u_{\neq}=0,$ direct calculations indicate that
\begin{equation*}
	\begin{aligned}
		&\|\nabla^{k}\partial_{x}u_{\neq}\|_{L^{2}}+\|\nabla^{k}\partial_{z}u_{3,\neq}\|_{L^{2}}
		\leq C\left(\|\nabla^{k+1}u_{2,\neq}\|_{L^{2}}+\|\nabla^{k}(\partial_{x},\partial_{z})u_{3,\neq}\|_{L^{2}}\right)
		\\\leq&C\big(\|\nabla^{k}\triangle u_{2,\neq}\|_{L^{2}}^{\frac12}
		\|\nabla^{k} u_{2,\neq}\|_{L^{2}}^{\frac12}+\|\nabla^{k}(\partial_{x}^{2}
		+\partial_{z}^{2})u_{3,\neq}\|_{L^{2}}^{\frac12}\|\nabla^{k} u_{3,\neq}\|_{L^{2}}^{\frac12} \big)
	\end{aligned}
\end{equation*}
and
\begin{equation*}
	\begin{aligned}
		&\|e^{\frac{a+b}{2}\nu^{\frac13}t}\partial_{x}u_{\neq}\|_{Y_{0}}^{2}+\|e^{\frac{a+b}{2}\nu^{\frac13}t}\partial_{z}u_{3,\neq}\|_{Y_{0}}^{2}
		\\\leq&C\big(\|\triangle u_{2,\neq}\|_{X_{a}}
		\|u_{2,\neq}\|_{X_{b}}+\|(\partial_{x}^{2}+\partial_{z}^{2})u_{3,\neq}\|_{X_{a}}\|u_{3,\neq}\|_{X_{b}}\big)\leq CE_{4}E_{6}.
	\end{aligned}
\end{equation*}

\underline{{\bf Estimate $ (\ref{estimate_u})_{3}. $}} By $ \|(\partial_{x},\partial_{z})(f_{1},f_{2})\|_{L^{2}}^{2}=\|(\partial_{z}f_{1}-\partial_{x}f_{2},\partial_{x}f_{1}+\partial_{z}f_{2})\|_{L^{2}}^{2}, $ we get
\begin{equation*}
	\|\triangle(\partial_{x},\partial_{z})(u_{1,\neq},u_{3,\neq})\|_{L^{2}}^{2}=\|\triangle\omega_{2,\neq}\|_{L^{2}}^{2}+\|\triangle\partial_{y}u_{2,\neq}\|_{L^{2}}^{2}.
\end{equation*}
\end{proof}

\begin{lemma}\label{zero mode}
It holds that
\begin{equation*}
	\|u_{2,0}\|_{H^{2}}+\|\nabla u_{2,0}\|_{H^{1}}+\|u_{3,0}\|_{H^{1}}+\|\partial_{z}u_{3,0}\|_{H^{1}}\leq CE_{2},
\end{equation*}
\begin{equation*}
	\|u_{j,0}\|_{L^{\infty}L^{\infty}}+\nu^{\frac12}\|\nabla u_{j,0}\|_{L^{2}L^{\infty}}\leq CE_{2},   \quad {\rm for}~~j\in\{2,3\}.
\end{equation*}
\end{lemma}
\begin{proof}
As $ \partial_{y}u_{2,0}+\partial_{z}u_{3,0}=0 $ and $ \frac{1}{|\mathbb{T}|}\int_{\mathbb{T}}u_{2,0}dz=0, $ we have
\begin{equation*}
	\begin{aligned}
		&\|u_{2,0}\|_{H^{2}}+\|\nabla u_{2,0}\|_{H^{1}}+\|u_{3,0}\|_{H^{1}}+\|\partial_{z}u_{3,0}\|_{H^{1}}
		\\\leq&C\left(\|\triangle u_{2,0}\|_{Y_{0}}+\|u_{3,0}\|_{Y_{0}}+\|\nabla u_{3,0}\|_{Y_{0}} \right)\leq CE_{2}.
	\end{aligned}
\end{equation*}

Thanks to ${\rm div}~u_{0}=0$ and Lemma \ref{lem zero nonzero},
direct calculations indicate that
\begin{equation*}
	\begin{aligned}
		\|u_{2,0}\|_{L^{\infty}L^{\infty}}+\|u_{3,0}\|_{L^{\infty}L^{\infty}}\leq C(\|u_{2,0}\|_{L^{\infty}H^{2}}+\|u_{3,0}\|_{L^{\infty}H^{1}})\leq CE_{2}
	\end{aligned}
\end{equation*}
and
\begin{equation*}
	\begin{aligned}
		&\nu^{\frac12}\|\nabla u_{2,0}\|_{L^{2}L^{\infty}}\leq C\nu^{\frac12}\|\nabla u_{2,0}\|_{L^{2}H^{2}}
		\leq C\|\triangle u_{2,0}\|_{Y_{0}}\leq CE_{2},\\
		&\nu^{\frac12}\|\nabla u_{3,0}\|_{L^{2}L^{\infty}}
		\leq C\nu^{\frac12}\left(\|\nabla u_{3,0}\|_{L^{2}H^{1}}
		+\|\partial_{z}\nabla u_{3,0}\|_{L^{2}H^{1}} \right)\leq CE_{2}.
	\end{aligned}
\end{equation*}	
\end{proof}

\begin{lemma}\label{U10}
It holds that
\begin{equation*}
	\begin{aligned}
		&\|u_{1,0}\|_{H^{2}}\leq CE_{1}\min\big(\nu t+\nu^{\frac23}, 1 \big),\\
		&\|\kappa\nabla u_{3,0}\|_{H^{1}}\leq CE_1E_{2}.			
	\end{aligned}
\end{equation*}
\end{lemma}
\begin{proof}
As $ \widehat{u_{1,0}}(0)=0, $ then it follows from $ \widehat{u_{1,0}}(t)=\int_{0}^{t}\partial_{t}\widehat{u_{1,0}}(s)ds $ that
\begin{equation*}
	\begin{aligned}
		\|\widehat{u_{1,0}}(t)\|_{H^{2}}\leq\int_{0}^{t}\|\partial_{t}\widehat{u_{1,0}}(s)\|_{H^{2}}ds\leq E_{1,1}\nu t.
	\end{aligned}
\end{equation*}
On the other hand, $ \|\widehat{u_{1,0}}(t)\|_{H^{2}}\leq \|\widehat{u_{1,0}}(t)\|_{H^{4}}\leq E_{1,1}, $ thus 
\begin{equation}\label{widehat u10 H2}
	\|\widehat{u_{1,0}}\|_{H^{2}}\leq E_{1,1}\min(\nu t, 1).
\end{equation}
From this, along with $ \|\widetilde{u_{1,0}}\|_{H^{2}}\leq E_{1,2}\leq \nu^{\frac23}E_{1}, $ we get
\begin{equation*}
	\|u_{1,0}\|_{H^{2}}\leq \|\widehat{u_{1,0}}\|_{H^{2}}+\|\widetilde{u_{1,0}}\|_{H^{2}}\leq CE_{1}\big(\min(\nu t, 1)
	+\nu^{\frac23} \big)\leq CE_{1}\min\big(\nu t+\nu^{\frac23}, 1 \big).
\end{equation*}

Due to Lemma \ref{kappa}, there hold
\begin{equation*}
	\|\kappa\|_{H^{1}}\leq C\|\widehat{u_{1,0}}\|_{H^{2}}\leq CE_{1}\min\{\nu t, 1\}, ~~{\rm and}~~ \|\kappa\|_{H^{3}}\leq C\|\widehat{u_{1,0}}\|_{H^{4}}\leq CE_{1},
\end{equation*}
which imply that
\begin{equation*}
	\begin{aligned}
		\|\kappa\nabla u_{3,0}\|_{H^{1}}\leq&C\|\kappa\|_{H^{2}}\left(\|\nabla u_{3,0}\|_{L^{2}}+\|\triangle u_{3,0}\|_{L^{2}} \right)
		%			\\\leq& C\|\kappa\|_{H^{1}}^{\frac12}\|\kappa\|_{H^{3}}^{\frac12}\left(\|\nabla u_{3,0}\|_{L^{2}}+\|\triangle u_{3,0}\|_{L^{2}} \right)
		%			\\\leq&CE_{1}\min(\nu t, 1)^{\frac12}\big(\|\nabla u_{3,0}\|_{L^{2}}+\|\triangle u_{3,0}\|_{L^{2}} \big)
		\\\leq&CE_{1}\big(\|\nabla u_{3,0}\|_{L^{2}}+\|\min(\nu^{\frac23}+\nu t, 1)^{\frac12}\triangle u_{3,0}\|_{L^{2}} \big)\leq CE_{1}E_{2}.
	\end{aligned}
\end{equation*}
\end{proof}

\section{Nonlinear interactions}
In this section, we are  devoted to estimate the nonlinear estimates.
To estimate
$E_1, \cdots, E_6$ in the energy functional, it is necessary to consider the estimates of all nonlinear terms in (\ref{ini1}), \eqref{decom u10} and \eqref{23}. 
For simplicity, we sometimes use the following auxiliary energy
\begin{equation*}
E_{7}=\sum_{j=2}^{3}\left(\|\partial_{x}^{2}u_{j,\neq}\|_{X_{b}}^{2}+\|\partial_{x}(\partial_{z}-\kappa\partial_{y})u_{j,\neq}\|_{X_{b}}^{2} \right)+\|\partial_{x}\nabla Q\|_{X_{b}}^{2}.
\end{equation*}

Recall that the nonlinear terms $ h_{j} (j=2,3) $ in (\ref{hj}) satisfy
\begin{equation*}
h_{j}=\left(u_{2,0}\partial_{y}+u_{3,0}\partial_{z} \right)u_{j}+u_{\neq}\cdot\nabla u_{j}+\partial_{j}\left(P^{0}+P^{2}+P^{3}+P^{4} \right).
\end{equation*} 
For simplicity, we further write $ h_{j}=\sum_{k=1}^{7}h_{j,k}, $ where
$$ h_{j,1}=\left(u_{2,0}\partial_{y}+u_{3,0}\partial_{z} \right)u_{j},\quad h_{j,2}=u_{\neq}\cdot\nabla u_{j,0},\quad h_{j,3}=u_{\neq}\cdot\nabla u_{j,\neq}, $$
$$h_{j,k+2}=\partial_{j}P^{k}\quad {\rm for}\quad k\in\{2,3,4\},\quad h_{j,7}=\partial_{j}P^{0}.$$
As $ P^{2}_{\neq}=0, $ there holds $ (h_{j,4})_{\neq}=\partial_{j}P^{2}_{\neq}=0. $
Besides, we let $ H_{2}=(h_{2}+\kappa h_{3})_{\neq} $ and denote
\begin{equation*}
H_{2}=\sum_{j=1}^{7}H_{2,j},\quad H_{2,j}=(h_{2,j}+\kappa h_{3,j})_{\neq}.
\end{equation*}
For $ Q=u_{2,\neq}+\kappa u_{3,\neq}, $ direct calculation shows that
\begin{equation*}
H_{2,1}=\left(u_{2,0}\partial_{y}+u_{3,0}\partial_{z} \right)Q-u_{3,\neq}\left(u_{2,0}\partial_{y}+u_{3,0}\partial_{z} \right)\kappa.
\end{equation*}

\subsection{Interaction between zero modes}

\begin{lemma}\label{lem u20 u30}
It holds that
\begin{equation*}
	\|u_{2,0}\partial_{y}\widehat{u_{1,0}}\|_{Y_{0}^{2}}+\|u_{3,0}\partial_{z}\widehat{u_{1,0}}\|_{Y_{0}^{2}}\leq CE_{1,1}E_{2}.
\end{equation*}
\end{lemma}
\begin{proof}
As $ \|f_{1}f_{2}\|_{H^{2}}\leq C\|f_{1}\|_{H^{2}}\|f_{2}\|_{H^{2}}$, we get
\begin{equation*}\label{y0}
	\begin{aligned}
		\|f_{1}f_{2}\|_{Y_{0}^{2}}^{2}=&\|f_{1}f_{2}\|_{L^{\infty}H^{2}}^{2}+\nu\|\nabla(f_{1}f_{2})\|_{L^{2}H^{2}}^{2}\leq C\|f_{1}\|_{Y_{0}^{2}}^{2}\|f_{2}\|_{Y_{0}^{2}}^{2},
	\end{aligned}
\end{equation*}
thus $ \|u_{2,0}\partial_{y}\widehat{u_{1,0}}\|_{Y_{0}^{2}}^{2}\leq C\|u_{2,0}\|_{Y_{0}^{2}}^{2}\|\partial_{y}\widehat{u_{1,0}}\|_{Y_{0}^{2}}^{2}. $
Moreover, due to $ \frac{1}{|\mathbb{T}|}\int_{\mathbb{T}}u_{2,0}dz=0, $ there holds
\begin{equation*}\label{U20 U0'}
	\|u_{2,0}\partial_{y}\widehat{u_{1,0}}\|_{Y_{0}^{2}}\leq C\|\triangle u_{2,0}\|_{Y_{0}}\|\widehat{u_{1,0}}\|_{Y_{0}^{4}}\leq CE_{2}E_{1,1}.
\end{equation*}
From (\ref{widehat u10 H2}), due to
\begin{equation}\label{partial_z U0}
	\|\partial_{z}\widehat{u_{1,0}}\|_{H^{2}}\leq \|\widehat{u_{1,0}}\|_{H^{3}}\leq \|\widehat{u_{1,0}}\|_{H^{2}}^{\frac12}\|\widehat{u_{1,0}}\|_{H^{4}}^{\frac12}\leq CE_{1,1}\min(\nu t, 1)^{\frac12},
\end{equation}
we have
\begin{equation}\label{U30 U0‘}
	\begin{aligned}
		&\|u_{3,0}\partial_{z}\widehat{u_{1,0}}\|_{H^{2}}\leq C\left(\|u_{3,0}\|_{L^{2}}+\|\triangle u_{3,0}\|_{L^{2}} \right)E_{1,1}\min(\nu t, 1)^{\frac12}\\\leq&C\big(\|u_{3,0}\|_{L^{\infty}L^{2}}+\|\min(\nu^{\frac23}+\nu t, 1)^{\frac12}\triangle u_{3,0}\|_{Y_{0}} \big)E_{1,1}\leq CE_{2}E_{1,1}.
	\end{aligned}
\end{equation}
Using Lemma \ref{sob_04} and (\ref{partial_z U0}),  direct calculations show that
\begin{equation*}
	\begin{aligned}
		&\|\nabla(u_{3,0}\partial_{z}\widehat{u_{1,0}})\|_{H^{2}}\\\leq&C\left(\|u_{3,0}\|_{L^{2}}\|\partial_{z}\widehat{u_{1,0}}\|_{H^{3}}+\|\triangle u_{3,0}\|_{L^{2}}\|\widehat{u_{1,0}}\|_{H^{4}}+\|\nabla\triangle u_{3,0}\|_{L^{2}}\|\partial_{z}\widehat{u_{1,0}}\|_{H^{2}} \right)\\\leq&C\big(E_{2}\|\nabla \widehat{u_{1,0}}\|_{H^{4}}+\|\triangle u_{3,0}\|_{L^{2}}E_{1,1}+\|\nabla\triangle u_{3,0}\|_{L^{2}}E_{1,1}\min(\nu t, 1)^{\frac12} \big),
	\end{aligned}
\end{equation*}
which indicates that
\begin{equation}\label{nable U30 U0'}
	\begin{aligned}
		\nu\|\nabla(u_{3,0}\partial_{z}\widehat{u_{1,0}})\|_{L^{2}H^{2}}^{2}\leq CE_{2}^{2}E_{1,1}^{2}.
	\end{aligned}
\end{equation}
It follows from (\ref{U30 U0‘}) and (\ref{nable U30 U0'}) that
\begin{equation*}\label{U30 U0' Y0}
	\|u_{3,0}\partial_{z}\widehat{u_{1,0}}\|_{Y_{0}^{2}}\leq CE_{2}E_{1,1}.
\end{equation*}
\end{proof}

\subsection{Interaction between non-zero modes}
\begin{lemma}\label{lem u theta}
It holds that
\begin{equation*}
	\|e^{2a\nu^{\frac13}t}\partial_{x}^{2}\left(u_{\neq}\Theta_{\neq} \right)\|_{L^{2}L^{2}}^{2}
	+\|e^{2a\nu^{\frac13}t}\partial_{z}(u_{\neq}\cdot\nabla\Theta_{\neq})\|_{L^{2}L^{2}}^{2}
	+\|e^{2a\nu^{\frac13}t}u_{\neq}\cdot\nabla\Theta_{\neq}\|_{L^{2}L^{2}}^{2}
	\leq C\nu^{-1}E_{4}^{2}E_{5}^{2}.
\end{equation*}
\end{lemma}
\begin{proof}
Due to Lemma \ref{lem omega} and Lemma \ref{lem nonzreo}, there holds
\begin{equation*}
	\begin{aligned}
		&\|\partial_{x}u_{\neq}\partial_{x}\Theta_{\neq}\|_{L^{2}}\leq\|\partial_{x}u_{\neq}\|_{L^{\infty}_{z}L^{2}_{x,y}}\|\partial_{x}\Theta_{\neq}\|_{L^{\infty}_{x,y}L^{2}_{z}}
		\\\leq&C\left(\|(\partial_{x}^{2}+\partial_{z}^{2})u_{3,\neq}\|_{L^{2}}+\|\triangle u_{2,\neq}\|_{L^{2}} \right)\|\nabla\partial_{x}^{2}\Theta_{\neq}\|_{L^{2}},
	\end{aligned}
\end{equation*}
which implies that 
\begin{equation}\label{0}
	\begin{aligned}
		\|e^{2a\nu^{\frac13}t}\partial_{x}u_{\neq}\partial_{x}\Theta_{\neq}\|_{L^{2}L^{2}}^{2}\leq C\nu^{-1}E_{4}^{2}E_{5}^{2}.	
	\end{aligned}
\end{equation}
Similarly,	by Lemma \ref{lem nonzreo} and Lemma \ref{lem omega} again, we have
\begin{equation*}
	\begin{aligned}
		&\|e^{2a\nu^{\frac13}t}\Theta_{\neq}\partial_{x}^{2}u_{\neq}\|_{L^{2}L^{2}}^{2}
		\leq\|e^{a\nu^{\frac13}t}\Theta_{\neq}\|_{L^{2}L^{\infty}}^{2}
		\|e^{a\nu^{\frac13}t}\partial_{x}^{2}u_{\neq}\|_{L^{\infty}L^{2}}^{2}
		\\\leq&C\|e^{a\nu^{\frac13}t}(\partial_{x}^{2},\partial_{z}^{2})\Theta_{\neq}\|_{L^{2}H^1}^{2}
		\big(\|e^{a\nu^{\frac13}t}(\partial_{x}^{2},\partial_{z}^{2})u_{3,\neq}\|_{L^{\infty}L^{2}}^{2}
		+\|e^{a\nu^{\frac13}t}\triangle u_{2,\neq}\|_{L^{\infty}L^{2}}^{2} \big)\leq C\nu^{-1}E_{4}^{2}E_{5}^{2}
	\end{aligned}
\end{equation*}
and
\begin{equation*}
	\begin{aligned}
		&\|e^{2a\nu^{\frac13}t}u_{\neq}\partial_{x}^{2}\Theta_{\neq}\|_{L^{2}L^{2}}^{2}
		\leq\|e^{a\nu^{\frac13}t}u_{\neq}\|_{L^{2}L^{\infty}}^{2}
		\|e^{a\nu^{\frac13}t}\partial_{x}^{2}\Theta_{\neq}\|_{L^{\infty}L^{2}}^{2}
		\\\leq&C\big(\|e^{a\nu^{\frac13}t}\nabla(\partial_{x}^{2}+\partial_{z}^{2})u_{3,\neq}\|_{L^{2}L^{2}}^{2}
		+\|e^{a\nu^{\frac13}t}\nabla\triangle u_{2,\neq}\|_{L^{2}L^{2}}^{2} \big)\|e^{a\nu^{\frac13}t}\partial_{x}^{2}\Theta_{\neq}\|_{L^{\infty}L^{2}}^{2}
		\leq C\nu^{-1}E_{4}^{2}E_{5}^{2},
	\end{aligned}
\end{equation*}
which along with (\ref{0}) imply that 
\begin{equation*}\label{xx u theta neq}
	\begin{aligned}
		\|e^{2a\nu^{\frac13}t}\partial_{x}^{2}(u_{\neq}\Theta_{\neq})\|_{L^{2}L^{2}}^{2}\leq C\nu^{-1}E_{4}^{2}E_{5}^{2}.
	\end{aligned}	
\end{equation*}

By Lemma \ref{lem nonzreo} and Lemma \ref{lem omega}, we have
\begin{equation*}\label{u infty}
	\|u_{\neq}\|_{L^{\infty}}\leq C\|\nabla(\partial_{x},\partial_{z})\partial_{x}u_{\neq}\|_{L^{2}}\leq C\left(\|\nabla(\partial_{x}^{2}+\partial_{z}^{2})u_{3,\neq}\|_{L^{2}}+\|\nabla\triangle u_{2,\neq}\|_{L^{2}} \right),
\end{equation*}
and for $ k\in\{1,3\}, $ there holds
\begin{equation*}
	\begin{aligned}
		\|\partial_{z}u_{k,\neq}\partial_{k}\Theta_{\neq}\|_{L^{2}}\leq C\left(\|(\partial_{x}^{2}+\partial_{z}^{2})u_{3,\neq}\|_{L^{2}}+\|\triangle u_{2,\neq}\|_{L^{2}} \right)\|\nabla(\partial_{x}^{2}+\partial_{z}^{2})\Theta_{\neq}\|_{L^{2}}.
	\end{aligned}
\end{equation*}
Moreover, Lemma \ref{sob_15} shows that
\begin{equation*}
	\|\partial_{z}u_{2,\neq}\partial_{y}\Theta_{\neq}\|_{L^{2}}\leq C\left(\|\partial_{y}\partial_{z}\Theta_{\neq}\|_{L^{2}}+\|\partial_{y}\Theta_{\neq}\|_{L^{2}} \right)\|\triangle u_{2,\neq}\|_{L^{2}}.
\end{equation*}
Using the above estimations, we obtain 
\begin{equation*}
	\begin{aligned}
		\|\partial_{z}(u_{\neq}\cdot\nabla\Theta_{\neq})\|_{L^{2}}^{2}
		\leq& C\left(\|(\partial_{x}^{2}+\partial_{z}^{2})u_{3,\neq}\|_{L^{2}}^{2}+\|\triangle u_{2,\neq}\|_{L^{2}}^{2} \right)\|\nabla(\partial_{x}^{2}+\partial_{z}^{2})\Theta_{\neq}\|_{L^{2}}^{2}\\	&{+C\left(\|\nabla(\partial_{x}^{2}+\partial_{z}^{2})u_{3,\neq}\|_{L^{2}}^{2}+\|\nabla\triangle u_{2,\neq}\|_{L^{2}}^{2} \right)\|(\partial_{x},\partial_{z})\partial_{z}\Theta_{\neq}\|_{L^{2}}^{2},
		}
	\end{aligned}
\end{equation*}
which implies that
\begin{equation*}\label{2a u theta'}
	\begin{aligned}
		\|e^{2a\nu^{\frac13}t}\partial_{z}(u_{\neq}\cdot\nabla\Theta_{\neq})\|_{L^{2}L^{2}}^{2}\leq C\nu^{-1}E_{4}^{2}E_{5}^{2}.
	\end{aligned}
\end{equation*}
Similarly, there hold
\begin{equation*}
	\begin{aligned}
	&\|u_{k,\neq}\partial_{k}\Theta_{\neq}\|_{L^{2}}\leq C\left(\|(\partial_{x}^{2}+\partial_{z}^{2})u_{3,\neq}\|_{L^{2}}+\|\triangle u_{2,\neq}\|_{L^{2}} \right)\|\nabla(\partial_{x}^{2}+\partial_{z}^{2})\Theta_{\neq}\|_{L^{2}},\\
	&\|u_{2,\neq}\partial_{y}\Theta_{\neq}\|_{L^{2}}\leq C\left(\|\partial_{y}\partial_{z}\Theta_{\neq}\|_{L^{2}}+\|\partial_{y}\Theta_{\neq}\|_{L^{2}} \right)\|\triangle u_{2,\neq}\|_{L^{2}},		
	\end{aligned}
\end{equation*}
which implies that 
$$\|e^{2a\nu^{\frac13}t}u_{\neq}\cdot\nabla\Theta_{\neq}\|_{L^{2}L^{2}}^{2}
\leq C\nu^{-1}E_{4}^{2}E_{5}^{2}.$$

We finish the proof.
\end{proof}
\begin{lemma}\label{lem 2}
It holds that
\begin{itemize}
	\item[(i)] 
	$\|e^{2a\nu^{\frac13}t}\nabla h_{2,3}\|_{L^{2}L^{2}}^{2}+\|e^{2a\nu^{\frac13}t}\triangle P^{4}\|_{L^{2}L^{2}}^{2}+\|e^{2a\nu^{\frac13}t}\nabla h_{2,6}\|_{L^{2}L^{2}}^{2}+\|e^{2a\nu^{\frac13}t}\nabla h_{3,6}\|_{L^{2}L^{2}}^{2}\\+\|e^{2a\nu^{\frac13}t}|u_{\neq}|^{2}\|_{L^{2}L^{2}}^{2}+\|e^{2a\nu^{\frac13}t}u_{\neq}\cdot\nabla u_{\neq}\|_{L^{2}L^{2}}^{2}+\|e^{2a\nu^{\frac13}t}(\partial_{x},\partial_{z})h_{3,3}\|_{L^{2}L^{2}}^{2}\leq C\nu^{-1}E_{4}^{4},$
	\item[(ii)]  
	$\|e^{2a\nu^{\frac13}t}\nabla h_{3,3}\|_{L^{2}L^{2}}^{2}+\|e^{2a\nu^{\frac13}t}\nabla(u_{\neq}\cdot\nabla u_{\neq})\|_{L^{2}L^{2}}^{2}
	\leq C\nu^{-\frac53}E_{4}^{4}.$
\end{itemize}
\end{lemma}
\begin{proof} 
For convenience, we denote 
\begin{equation*}
	\begin{aligned}
		&\Gamma_{1}=\|(\partial_{x}^{2}+\partial_{z}^{2})u_{3,\neq}\|_{L^{2}}+\|\triangle u_{2,\neq}\|_{L^{2}},\\
		&\Gamma_{2}=\|\nabla(\partial_{x}^{2}+\partial_{z}^{2})u_{3,\neq}\|_{L^{2}}+\|\nabla\triangle u_{2,\neq}\|_{L^{2}}. 
	\end{aligned}
\end{equation*}
%	Direct calculations indicate that
%	\begin{equation}\label{gamma 1}
	%		\begin{aligned}
		%			\|e^{a\nu^{\frac13}t}\Gamma_{1}\|_{L^{\infty}(0,T)}\leq&\|e^{a\nu^{\frac13}t}(\partial_{x}^{2}+\partial_{z}^{2})u_{3,\neq}\|_{L^{\infty}L^{2}}+\|e^{a\nu^{\frac13}t}\triangle u_{2,\neq}\|_{L^{\infty}L^{2}}\\\leq&\|(\partial_{x}^{2}+\partial_{z}^{2})u_{3,\neq}\|_{X_{a}}+\|\triangle u_{2,\neq}\|_{X_{a}}\leq E_{4},
		%		\end{aligned}
	%	\end{equation}
%	and
%	\begin{equation}\label{gamma 2}
	%		\begin{aligned}
		%			\|e^{a\nu^{\frac13}t}\Gamma_{2}\|_{L^{2}(0,T)}\leq&\|e^{a\nu^{\frac13}t}\nabla(\partial_{x}^{2}+\partial_{z}^{2})u_{3,\neq}\|_{L^{2}L^{2}}+\|\nabla\triangle u_{2,\neq}\|_{L^{2}L^{2}}\\\leq&\nu^{-\frac12}\|(\partial_{x}^{2}+\partial_{z}^{2})u_{3,\neq}\|_{X_{a}}+\nu^{-\frac12}\|\triangle u_{2,\neq}\|_{X_{a}}\leq\nu^{-\frac12}E_{4}.
		%		\end{aligned}
	%	\end{equation}

\textbf{Proof of (i).}	{\underline{\textbf{Estimate $ |u_{\neq}|^{2},$ $ u_{\neq}\cdot\nabla u_{\neq}  $ and $ (\partial_{x},\partial_{z})h_{3,3}. $ }}} Using Lemma \ref{sob_04}, we get
\begin{equation*}
	\begin{aligned}
		\||u_{\neq}|^{2}\|_{L^{2}}\leq&C\left(\|\partial_{x}u_{\neq}\|_{H^{1}}+\|u_{\neq}\|_{H^{1}} \right)\left(\|\partial_{x}u_{\neq}\|_{L^{2}}+\|u_{\neq}\|_{L^{2}} \right)\\\leq&C\|\nabla\partial_{x}^{2}u_{\neq}\|_{L^{2}}\|\partial_{x}^{2}u_{\neq}\|_{L^{2}},
	\end{aligned}
\end{equation*}
and for $ s\in\{1,3\}, $ there hold
\begin{equation}\label{U1}
	\begin{aligned}
		\|u_{s,\neq}\partial_{s}u_{\neq}\|_{L^{2}}+\|\partial_{x}(u_{s,\neq}\partial_{s}u_{\neq})\|_{L^{2}}\leq&C\|(\partial_{x}u_{s,\neq},u_{s,\neq})\|_{H^{1}}\|(\partial_{x}\partial_{s}u_{\neq}, \partial_{s}u_{\neq})\|_{L^{2}}\\&+C\|(\partial_{x}u_{s,\neq},u_{s,\neq})\|_{L^{2}}\|(\partial_{x}\partial_{s}u_{\neq},\partial_{s}u_{\neq})\|_{H^{1}}\\\leq&C\|\nabla(\partial_{x},\partial_{z})\partial_{x}u_{\neq}\|_{L^{2}}\|(\partial_{x},\partial_{z})\partial_{x}u_{\neq}\|_{L^{2}}
	\end{aligned}
\end{equation}
and
\begin{equation*}
	\begin{aligned}
		\|\partial_{z}\left(u_{s,\neq}\partial_{s}u_{3,\neq} \right)\|_{L^{2}}\leq&C\|\nabla(\partial_{x},\partial_{z})\partial_{x}u_{\neq}\|_{L^{2}}\|(\partial_{x}^{2}+\partial_{z}^{2})u_{3,\neq}\|_{L^{2}}\\&+C\|(\partial_{x},\partial_{z})\partial_{x}u_{\neq}\|_{L^{2}}\|\nabla(\partial_{x}^{2}+\partial_{z}^{2})u_{3,\neq}\|_{L^{2}}.
	\end{aligned}
\end{equation*}
For $ s=2, $ by Lemma \ref{sob_15}, we obtain
\begin{equation}\label{U1 1}
	\begin{aligned}
		\|u_{2,\neq}\partial_{y}u_{\neq}\|_{L^{2}}+\|(\partial_{x},\partial_{z})(u_{2,\neq}\partial_{y}u_{\neq})\|_{L^{2}}
		\leq C\|\nabla(\partial_{x},\partial_{z})\partial_{x}u_{\neq}\|_{L^{2}}\|\triangle u_{2,\neq}\|_{L^{2}}.
	\end{aligned}
\end{equation}

Due to ${\rm div}~u_{\neq}=0,$ we infer from above estimates that
\begin{equation}\label{U h33}
	\begin{aligned}
		&\||u_{\neq}|^{2}\|_{L^{2}}+\|u_{\neq}\cdot\nabla u_{\neq}\|_{L^{2}}+\|(\partial_{x},\partial_{z})h_{3,3}\|_{L^{2}}\\\leq&C\|\nabla(\partial_{x},\partial_{z})\partial_{x}u_{\neq}\|_{L^{2}}\big( \|(\partial_{x}^{2}+\partial_{z}^{2})u_{3,\neq}\|_{L^{2}}+\|\triangle u_{2,\neq}\|_{L^{2}}	\big)\\&+C\|(\partial_{x},\partial_{z})\partial_{x}u_{\neq}\|_{L^{2}}\|\nabla(\partial_{x}^{2}+\partial_{z}^{2})u_{3,\neq}\|_{L^{2}},
	\end{aligned}
\end{equation}
which along with Lemma \ref{lem omega} and  (\ref{U h33}) imply that 
\begin{equation}\label{result U h33}
	\begin{aligned}
		&\|e^{2a\nu^{\frac13}t}|u_{\neq}|^{2}\|_{L^{2}L^{2}}^{2}+\|e^{2a\nu^{\frac13}t}u_{\neq}\cdot\nabla u_{\neq}\|_{L^{2}L^{2}}^{2}+\|e^{2a\nu^{\frac13}t}(\partial_{x},\partial_{z})h_{3,3}\|_{L^{2}L^{2}}^{2}\leq C\nu^{-1}E_{4}^{4}.
	\end{aligned}
\end{equation}

{\underline{\textbf{Estimate $ \nabla h_{2,3}. $}}} For $ s\in\{1,3\}, $ using Lemma \ref{sob_15}, we have
\begin{equation*}
	\begin{aligned}
		\|\nabla\left(u_{s,\neq}\partial_{s}u_{2,\neq} \right)\|_{L^{2}}\leq&\|\nabla u_{s,\neq}\partial_{s}u_{2,\neq}\|_{L^{2}}+\|u_{s,\neq}\nabla\partial_{s}u_{2,\neq}\|_{L^{2}}\\\leq&C\left(\|\nabla\partial_{s}u_{s,\neq}\|_{L^{2}}+\|\nabla u_{s,\neq}\|_{L^{2}} \right)\|\triangle u_{2,\neq}\|_{L^{2}}\\&+C\left(\|\partial_{s}u_{s,\neq}\|_{L^{2}}+\|u_{s,\neq}\|_{L^{2}} \right)\|\nabla\triangle u_{2,\neq}\|_{L^{2}}.
	\end{aligned}
\end{equation*}
For $ s=2, $ it follows that
\begin{equation*}
	\begin{aligned}
		\|\nabla\left(u_{2,\neq}\partial_{y}u_{2,\neq} \right)\|_{L^{2}}\leq C\|u_{2,\neq}\|_{H^{2}}\|\nabla u_{2,\neq}\|_{H^{2}}\leq C\|\triangle u_{2,\neq}\|_{L^{2}}\|\nabla\triangle u_{2,\neq}\|_{L^{2}}.
	\end{aligned}
\end{equation*}
Hence for $ h_{2,3}=u_{\neq}\cdot\nabla u_{2,\neq}, $ by using  Lemma \ref{lem omega}, we obtain
\begin{equation*}
	\begin{aligned}
		\|\nabla h_{2,3}\|_{L^{2}}\leq&C\left(\|(\partial_{x}^{2}+\partial_{z}^{2})u_{3,\neq}\|_{L^{2}}+\|\triangle u_{2,\neq}\|_{L^{2}} \right)\\&\cdot\left(\|\nabla(\partial_{x}^{2}+\partial_{z}^{2})u_{3,\neq}\|_{L^{2}}+\|\nabla\triangle u_{2,\neq}\|_{L^{2}} \right)
	\end{aligned}
\end{equation*}	
and
\begin{equation}\label{h23}
	\begin{aligned}
		\|e^{2a\nu^{\frac13}t}\nabla h_{2,3}\|_{L^{2}L^{2}}^{2}\leq C\|e^{a\nu^{\frac13}t}\Gamma_{1}\|_{L^{\infty}(0,T)}^{2}\|e^{a\nu^{\frac13}t}\Gamma_{2}\|_{L^{2}(0,T)}^{2}\leq C\nu^{-1}E_{4}^{4}.
	\end{aligned}
\end{equation}

{\underline{\textbf{Estimate $ \triangle P^{4}, \nabla h_{2,6} $ and $ \nabla h_{3,6}. $}}} Due to $ {\rm{div}}~u_{\neq}=0, $ we notice that
\begin{equation*}
	\triangle P^{4}=-\partial_{i}u_{j,\neq}\partial_{j}u_{i,\neq}=-\partial_{j}\left(\partial_{i}u_{j,\neq}u_{i,\neq} \right)=-\partial_{j}\left(u_{\neq}\cdot\nabla u_{j,\neq} \right),
\end{equation*}
which implies that 
\begin{equation}\label{first P4}
	\begin{aligned}
		\|e^{2a\nu^{\frac13}t}\triangle P^{4}\|_{L^{2}L^{2}}^{2}\leq&\|e^{2a\nu^{\frac13}t}\partial_{x}(u_{\neq}\cdot\nabla u_{1,\neq})\|_{L^{2}L^{2}}^{2}+\|e^{2a\nu^{\frac13}t}\partial_{y}(u_{\neq}\cdot\nabla u_{2,\neq})\|_{L^{2}L^{2}}^{2}\\&+\|e^{2a\nu^{\frac13}t}\partial_{z}(u_{\neq}\cdot\nabla u_{3,\neq})\|_{L^{2}L^{2}}^{2}:=J_{1}+J_{2}+J_{3}.
	\end{aligned}
\end{equation}
Using (\ref{U1}), (\ref{U1 1}), (\ref{result U h33}), (\ref{h23}) and Lemma \ref{lem omega}, we obtain 
\begin{equation*}
	\begin{aligned}
		&J_{1}\leq C\|e^{a\nu^{\frac13}t}\Gamma_{1}\|_{L^{\infty}(0,T)}^{2}\|e^{a\nu^{\frac13}t}\Gamma_{2}\|_{L^{2}(0,T)}^{2}\leq C\nu^{-1}E_{4}^{4},\\
		&J_{2}\leq\|e^{2a\nu^{\frac13}t}\nabla h_{2,3}\|_{L^{2}L^{2}}^{2}\leq C\nu^{-1}E_{4}^{4},\\
		&J_{3}\leq\|e^{2a\nu^{\frac13}t}\partial_{z}h_{3,3}\|_{L^{2}L^{2}}^{2}\leq C\nu^{-1}E_{4}^{4}.
	\end{aligned}
\end{equation*}
Combining the estimates of $ J_{1}-J_{3}, $ (\ref{first P4}) yields that
\begin{equation}\label{P4}
	\|e^{2a\nu^{\frac13}t}\triangle P^{4}\|_{L^{2}L^{2}}^{2}\leq C\nu^{-1}E_{4}^{4}.
\end{equation}

Recall $ h_{s,6}=\partial_{s}P^{4} $ for $ s\in\{2,3\}, $ then there holds
\begin{equation}\label{h26 h36}
	\|e^{2a\nu^{\frac13}t}\nabla h_{s,6}\|_{L^{2}L^{2}}^{2}\leq \|e^{2a\nu^{\frac13}t}\triangle P^{4}\|_{L^{2}L^{2}}^{2}\leq C\nu^{-1}E_{4}^{4}.
\end{equation}
Collecting (\ref{result U h33}), (\ref{h23}), (\ref{P4}) and (\ref{h26 h36}), the first inequality holds.

{\textbf{Proof of (ii).}} 
Using $(\ref{f_neq}),$ one deduces
\begin{equation*}
	\begin{aligned}
		\|\nabla(f_{1}f_{2})\|_{L^{2}}\leq&\|f_{2}\|_{L^{\infty}_{x}L^{2}_{y,z}}\|\nabla f_{1}\|_{L^{\infty}_{y,z}L^{2}_{x}}+\|f_{1}\|_{L^{\infty}_{x,z}L^{2}_{y}}\|\nabla f_{2}\|_{L^{\infty}_{y}L^{2}_{x,z}}\\\leq&C\left(\|(\partial_{z}f_{1}, f_{1})\|_{H^{2}}\|(\partial_{x}f_{2}, f_{2})\|_{L^{2}}+\|(\partial_{x}\partial_{z}f_{1},\partial_{x}f_{1},\partial_{z}f_{1}, f_{1})\|_{L^{2}}\|f_{2}\|_{H^{2}} \right).
	\end{aligned}
\end{equation*}
For $ s\in\{1, 3\},$ by Lemma \ref{lem omega}, there holds
\begin{equation*}
	\begin{aligned}
		\|\nabla(u_{s,\neq}\partial_{s}u_{\neq})\|_{L^{2}}^{2}
		\leq C\left(\|\nabla\triangle u_{2,\neq}\|_{L^{2}}^{2}+\|\triangle\omega_{2,\neq}\|_{L^{2}}^{2} \right)
		\left(\|(\partial_{x}^{2}+\partial_{z}^{2})u_{3,\neq}\|_{L^{2}}^{2}+\|\triangle u_{2,\neq}\|_{L^{2}}^{2} \right),
	\end{aligned}
\end{equation*}
and for $ s=2, $ we have
\begin{equation*}
	\begin{aligned}
		\|\nabla(u_{2,\neq}\partial_{y}u_{\neq})\|_{L^{2}}^{2}\leq&C\|(\partial_{z}u_{2,\neq},u_{2,\neq})\|_{H^{1}}\|(\partial_{x}\partial_{y}u_{\neq},\partial_{y}u_{\neq})\|_{H^{1}}\\\leq&C\|\triangle u_{2,\neq}\|_{L^{2}}^{2}\left(\|\nabla\triangle u_{2,\neq}\|_{L^{2}}^{2}+\|\triangle\omega_{2,\neq}\|_{L^{2}}^{2} \right).
	\end{aligned}
\end{equation*}
Therefore, we conclude that
\begin{equation*}\label{u cdot u'}
	\begin{aligned}
		\|e^{2a\nu^{\frac13}t}\nabla(u_{\neq}\cdot\nabla u_{\neq})\|_{L^{2}L^{2}}^{2}
		\leq C\left(\nu^{-1}E_{4,1}^{2}+\nu^{-\frac53}E_{3,3}^{2} \right)E_{4,1}^{2}\leq C\nu^{-\frac53}E_{4}^{4},
	\end{aligned}
\end{equation*}
which along with $\|\nabla h_{3,3}\|_{L^{2}}\leq \|\nabla(u_{\neq}\cdot\nabla u_{\neq})\|_{L^{2}} $ give the second result.
\end{proof}

\subsection{Interaction between zero mode and non-zero mode}
\begin{lemma}\label{lem non zero mode 1}
For $ j\in\{2,3\}, $ it holds that
\begin{equation}\label{uj0 theta}
	\begin{aligned}
		\|e^{a\nu^{\frac13}t}\partial_{z}(u_{j,0}\partial_{j}\Theta_{\neq})\|_{L^{2}L^{2}}^{2}&\leq C\nu^{-1}E_{2}^{2}E_{5}^{2},\\
		\|e^{a\nu^{\frac13}t}\partial_{z}(u_{j,\neq}\partial_{j}\Theta_{0})\|_{L^{2}L^{2}}^{2}&\leq C\nu^{-1}E_{3}^{2}E_{4}^{2},\\
		\nu^{-\frac13}\|e^{a\nu^{\frac13}t}\partial_{z}\widehat{u_{1,0}}\partial_{x}\partial_{z}\Theta_{\neq}\|_{L^{2}L^{2}}^{2}+\nu^{-1}\|e^{a\nu^{\frac13}t}&\partial_{z}\widehat{u_{1,0}}\partial_{x}\Theta_{\neq}\|_{L^{2}L^{2}}^{2}\leq CE_{1}^{2}E_{5}^{2}.
	\end{aligned}
\end{equation}
\end{lemma}
\begin{proof}
{\bf\underline{Estimate $ (\ref{uj0 theta})_{1}. $}} For $ j\in\{2,3\}, $ using Lemma \ref{sob_15}, we get
\begin{equation*}
	\begin{aligned}
		&\|\partial_{z}(u_{j,0}\partial_{j}\Theta_{\neq})\|_{L^{2}}\leq \|\partial_{z}u_{j,0}\partial_{j}\Theta_{\neq}\|_{L^{2}}+\|u_{j,0}\partial_{z}\partial_{j}\Theta_{\neq}\|_{L^{2}}
		\\\leq&C\left(\|\partial_{z}\partial_{j}\Theta_{\neq}\|_{L^{2}}+\|\partial_{j}\Theta_{\neq}\|_{L^{2}} \right)\|\triangle u_{j,0}\|_{L^{2}}+\|u_{j,0}\|_{L^{\infty}}\|\partial_{z}\partial_{j}\Theta_{\neq}\|_{L^{2}},
	\end{aligned}
\end{equation*}
which along with  Lemma \ref{zero mode} indicate that 
\begin{equation*}\label{02}
	\begin{aligned}
		&\|e^{a\nu^{\frac13}t}\partial_{z}(u_{j,0}\partial_{j}\Theta_{\neq})\|_{L^{2}L^{2}}^{2}\\\leq&C\left(\|e^{a\nu^{\frac13}t}\partial_{z}\partial_{j}\Theta_{\neq}\|_{L^{2}L^{2}}^{2}+\|e^{a\nu^{\frac13}t}\partial_{j}\Theta_{\neq}\|_{L^{2}L^{2}}^{2} \right)\|\triangle u_{j,0}\|_{L^{\infty}L^{2}}^{2}\leq C\nu^{-1}E_{2}^{2}E_{5}^{2}.
	\end{aligned}
\end{equation*}

{\bf\underline{Estimate $ (\ref{uj0 theta})_{2}. $}} 
Direct calculations show that 
\begin{equation*}\label{04}
	\begin{aligned}
		&\|e^{a\nu^{\frac13}t}\partial_{z}(u_{2,\neq}\partial_{y}\Theta_{0})\|_{L^{2}L^{2}}^{2}\leq C\|\partial_{y}\Theta_{0}\|_{L^{2}H^{1}}^{2}\|e^{a\nu^{\frac13}t}\triangle u_{2,\neq}\|_{L^{\infty}L^{2}}^{2}\\\leq&C\nu^{-1}\left(\|\Theta_{0}\|_{Y_{0}}^{2}+\|\partial_{z}\nabla\Theta_{0}\|_{Y_{0}}^{2} \right)\|\triangle u_{2,\neq}\|_{X_{a}}^{2}\leq C\nu^{-1}E_{3}^{2}E_{4}^{2}.
	\end{aligned}
\end{equation*}
Using Lemma \ref{lem zero nonzero}, Lemma \ref{lem nonzreo} and Lemma \ref{lem omega}, one deduces
\begin{equation*}
	\begin{aligned}
		\|\partial_{z}(u_{3,\neq}\partial_{z}\Theta_{0})\|_{L^{2}}\leq&\|\partial_{z}u_{3,\neq}\|_{L^{\infty}_{z}L^{2}_{y,x}}\|\partial_{z}\Theta_{0}\|_{L^{\infty}_{y}L^{2}_{z}}
		+\|u_{3,\neq}\|_{L^{\infty}}\|\partial_{z}^{2}\Theta_{0}\|_{L^{2}}\\\leq& C\left(\|\nabla(\partial_{x}^{2}+\partial_{z}^{2})u_{3,\neq}\|_{L^{2}}+\|\nabla\triangle u_{2,\neq}\|_{L^{2}}^{2} \right)\|\nabla\partial_{z}\Theta_{0}\|_{L^{2}},
	\end{aligned}
\end{equation*}
which follows that
\begin{equation*}
	\begin{aligned}
		&\|e^{a\nu^{\frac13}t}\partial_{z}(u_{3,\neq}\partial_{z}\Theta_{0})\|^2_{L^{2}L^2}\\\leq&C\nu^{-1}\left(\|(\partial_{x}^{2}+\partial_{z}^{2})u_{3,\neq}\|_{X_{a}}^{2}+\|\triangle u_{2,\neq}\|_{X_{a}}^{2} \right)\|\nabla\partial_{z}\Theta_{0}\|_{Y_{0}}^{2}\leq C\nu^{-1}E_{3}^{2}E_{4}^{2}.
	\end{aligned}
\end{equation*}

{\bf\underline{Estimate $ (\ref{uj0 theta})_{3}. $}} 	Due to Lemma \ref{sob_15} and Lemma \ref{U10}, there holds
\begin{equation*}
	\begin{aligned}
		\|\partial_{z}\widehat{u_{1,0}}\partial_{x}\partial_{z}\Theta_{\neq}\|_{L^{2}}\leq& C\left(\|\partial_{z}^{2}\partial_{x}\Theta_{\neq}\|_{L^{2}}+\|\partial_{x}\partial_{z}\Theta_{\neq}\|_{L^{2}} \right)\|\triangle \widehat{u_{1,0}}\|_{L^{2}}\\\leq&CE_{1}\nu t \|\nabla\partial_{x}^{2}\Theta_{\neq}\|_{L^{2}}^{\frac12}\|\nabla\partial_{z}^{2}\Theta_{\neq}\|_{L^{2}}^{\frac12}
	\end{aligned}
\end{equation*}
and 
\begin{equation*}
	\begin{aligned}
		\|\partial_{z}\widehat{u_{1,0}}\partial_{x}\Theta_{\neq}\|_{L^{2}}\leq& C\left(\|\partial_{x}\partial_{z}\Theta_{\neq}\|_{L^{2}}+\|\partial_{x}\Theta_{\neq}\|_{L^{2}} \right)\|\triangle \widehat{u_{1,0}}\|_{L^{2}}\\\leq&CE_{1}\nu t\|\partial_{x}^{2}\Theta_{\neq}\|_{L^{2}}^{\frac12}\|\partial_{z}^{2}\Theta_{\neq}\|_{L^{2}}^{\frac12}.
	\end{aligned}
\end{equation*}
We conclude that
\begin{equation*}\label{01}
	\begin{aligned}
		&\nu^{-\frac13}\|e^{a\nu^{\frac13}t}\partial_{z}\widehat{u_{1,0}}\partial_{x}\partial_{z}\Theta_{\neq}\|_{L^{2}L^{2}}^{2}+\nu^{-1}\|e^{a\nu^{\frac13}t}\partial_{z}\widehat{u_{1,0}}\partial_{x}\Theta_{\neq}\|_{L^{2}L^{2}}^{2}\\\leq&C\nu^{-\frac13}E_{1}^{2}\|(\nu t)^{2}e^{a\nu^{\frac13}t}\nabla\partial_{x}^{2}\Theta_{\neq}\|_{L^{2}L^{2}}\|e^{a\nu^{\frac13}t}\nabla\partial_{z}^{2}\Theta_{\neq}\|_{L^{2}L^{2}}\\&+C\nu^{-1}E_{1}^{2}\|(\nu t)^{2}e^{a\nu^{\frac13}t}\partial_{x}^{2}\Theta_{\neq}\|_{L^{2}L^{2}}\|e^{a\nu^{\frac13}t}\partial_{z}^{2}\Theta_{\neq}\|_{L^{2}L^{2}}
		\leq CE_{1}^{2}E_{5}^{2}.
	\end{aligned}
\end{equation*}
\end{proof}

\begin{lemma}\label{lem 1}
It holds that
\begin{itemize}
	\item[(i)]
	$\|e^{a\nu^{\frac13}t}\nabla(h_{2,1})_{\neq}\|_{L^{2}L^{2}}^{2}+\|e^{a\nu^{\frac13}t}\partial_{z}(h_{3,1})_{\neq}\|_{L^{2}L^{2}}^{2}\leq C\nu^{-1}E_{2}^{2}E_{4}^{2},$
	\item[(ii)]
	$\|e^{a\nu^{\frac13}t}\nabla(u_{1,0}\partial_{x}u_{2,\neq})\|_{L^{2}L^{2}}^{2}+\|e^{a\nu^{\frac13}t}\triangle P^{1}\|_{L^{2}L^{2}}^{2}+ \|e^{a\nu^{\frac13}t}(\partial_{x},\partial_{z})(u_{1,0}\partial_{x}u_{3,\neq})\|_{L^{2}L^{2}}^{2} \leq C\nu E_{1}^{2}E_{4}E_{6}.$
\end{itemize}	
\end{lemma}
\begin{proof}
\underline{\textbf{Estimate (i).}}
Recall that $ (h_{2,1})_{\neq}=\left(u_{2,0}\partial_{y}+u_{3,0}\partial_{z} \right)u_{2,\neq}, $ by  Lemma \ref{zero mode}, and we have
\begin{equation}\label{h2,1}
	\begin{aligned}
		\|e^{a\nu^{\frac13}t}\nabla(h_{2,1})_{\neq}\|_{L^{2}L^{2}}^{2}\leq& C\left(\|u_{2,0}\|_{L^{\infty}H^2}^{2}+\|u_{3,0}\|_{L^{\infty}H^1}^{2} \right)\|e^{a\nu^{\frac13}t}\triangle u_{2,\neq}\|_{L^2L^2}^{2}\\\leq& C\nu^{-\frac13}E_2^2E_4^2.
	\end{aligned}
\end{equation}
Due to $ (h_{3,1})_{\neq}=\left(u_{2,0}\partial_{y}+u_{3,0}\partial_{z} \right)u_{3,\neq}, $ one obtains that
\begin{equation*}
	\partial_{z}(h_{3,1})_{\neq}=\left(u_{2,0}\partial_{y}+u_{3,0}\partial_{z} \right)\partial_{z}u_{3,\neq}+\left(\partial_{z}u_{2,0}\partial_{y}+\partial_{z}u_{3,0}\partial_{z} \right)u_{3,\neq}.
\end{equation*}
Then using Lemma \ref{lem zero nonzero}, Lemma \ref{zero mode} and ${\rm div}~u_0=0$, we get
\begin{equation*}
	\begin{aligned}
		\|\partial_{z}(h_{3,1})_{\neq}\|_{L^{2}}\leq&\left(\|u_{2,0}\|_{H^{2}}+\|u_{3,0}\|_{H^{1}} \right)\|\nabla\partial_{z}u_{3,\neq}\|_{L^{2}},
	\end{aligned}
\end{equation*}
which indicates that 
\begin{equation}\label{partial_z h31}
	\begin{aligned}
		\|e^{a\nu^{\frac13}t}\partial_{z}(h_{3,1})_{\neq}\|_{L^{2}L^{2}}^{2}\leq&CE_{2}^{2}\nu^{-1}\|(\partial_{x}^{2}+\partial_{z}^{2})u_{3,\neq}\|_{X_{a}}^{2}\leq C\nu^{-1}E_{2}^{2}E_{4}^{2}.
	\end{aligned}
\end{equation}  

Combining (\ref{h2,1}) with (\ref{partial_z h31}), the first inequality holds.

\underline{\textbf{Estimate (ii).}}
It follows from Lemma \ref{U10} that
\begin{equation*}
	\begin{aligned}
		\|\nabla(u_{1,0}\partial_{x}u_{2,\neq})\|_{L^{2}}\leq&\|u_{1,0}\partial_{x}u_{2,\neq}\|_{H^{1}}\leq C\|u_{1,0}\|_{H^{2}}\|\partial_{x}\nabla u_{2,\neq}\|_{L^2}
		\\\leq& C\nu^{\frac23} E_{1}(1+\nu^{\frac13}t)\|\partial_{x}^{2}u_{2,\neq}\|_{L^{2}}^{\frac12}\|\triangle u_{2,\neq}\|_{L^{2}}^{\frac12}.
	\end{aligned}
\end{equation*}
By using the fact
\begin{equation}\label{a b}
	(1+\nu^{\frac13}t)^{2}\leq Ce^{(b-a)\nu^{\frac13}t},
\end{equation}
we get
\begin{equation}\label{1}
	\begin{aligned}
		&	\|e^{a\nu^{\frac13}t}\nabla\left(u_{1,0}\partial_{x}u_{2,\neq} \right)\|_{L^{2}L^{2}}^{2}\\\leq&C\nu^{\frac43}E_{1}^{2}\|(1+\nu^{\frac13}t)^{2}e^{a\nu^{\frac13}t}\partial_{x}^{2}u_{2,\neq}\|_{L^{2}L^{2}}\|e^{a\nu^{\frac13}t}\triangle u_{2,\neq}\|_{L^{2}L^{2}}\\\leq&C\nu^{\frac43}E_{1}^{2}\|e^{b\nu^{\frac13}t}\partial_{x}^{2}u_{2,\neq}\|_{L^{2}L^{2}}\|e^{a\nu^{\frac13}t}\triangle u_{2,\neq}\|_{L^{2}L^{2}}\leq  C\nu E_{1}^{2}E_{4}E_{6}.
	\end{aligned}
\end{equation}
Recall that $ \triangle P^{1}=-2\left(\partial_{y}\widehat{u_{1,0}}\partial_{x}u_{2,\neq}+\partial_{z}\widehat{u_{1,0}}\partial_{x}u_{3,\neq} \right), $ by  Lemma \ref{U10} and $ (\ref{a b}) $, and we get
\begin{equation*}
	\begin{aligned}
		\|\triangle P^{1}\|_{L^{2}}\leq&C\|u_{1,0}\|_{H^{2}}\left(\|\nabla\partial_{x}u_{2,\neq}\|_{L^{2}}+\|(\partial_{x},\partial_{z})\partial_{x}u_{3,\neq}\|_{L^{2}} \right)\\\leq&C\nu^{\frac23}E_{1}(1+\nu^{\frac13}t)\big(\|\partial_{x}^{2}u_{2,\neq}\|_{L^{2}}^{\frac12}\|\triangle u_{2,\neq}\|_{L^{2}}^{\frac12}+\|\partial_{x}^{2}u_{3,\neq}\|_{L^{2}}^{\frac12}\|(\partial_{x}^{2}+\partial_{z}^{2})u_{3,\neq}\|_{L^{2}}^{\frac12}\big),
	\end{aligned}
\end{equation*}
which implies that 
\begin{equation}\label{2}
	\begin{aligned}
		\|e^{a\nu^{\frac13}t}\triangle P^{1}\|_{L^{2}L^{2}}^{2}\leq&C\nu E_{1}^{2}E_{4}E_{6}.
	\end{aligned}
\end{equation}
Thanks to Lemma \ref{sob_15}, Lemma \ref{U10} and (\ref{a b}), there holds
\begin{equation*}
	\begin{aligned}
		\|(\partial_{x},\partial_{z})(u_{1,0}\partial_{x}u_{3,\neq})\|_{L^{2}}\leq&C\|u_{1,0}\|_{H^{2}}\|(\partial_{x},\partial_{z})\partial_{x}u_{3,\neq}\|_{L^{2}}
		\\\leq&C\nu^{\frac23} E_{1}(1+\nu^{\frac13}t)\|\partial_{x}^{2}u_{3,\neq}\|_{L^{2}}^{\frac12}\|(\partial_{x}^{2}+\partial_{z}^{2})u_{3,\neq}\|_{L^{2}}^{\frac12},
	\end{aligned}
\end{equation*}
which gives that
\begin{equation}\label{3}
	\begin{aligned}
		&\|e^{a\nu^{\frac13}t}(\partial_{x},\partial_{z})(u_{1,0}\partial_{x}u_{3,\neq})\|_{L^{2}L^{2}}^{2}\leq C\nu E_{1}^{2}E_{4}E_{6}.
	\end{aligned}
\end{equation}

Thus (\ref{1}), (\ref{2}) and (\ref{3}) imply the second result.

\end{proof}

\begin{lemma}\label{lem: zero and non-zero mode}
It holds that
\begin{itemize}
	\item[(i)] 
	$\|e^{b\nu^{\frac13}t}\nabla H_{2,1}\|_{L^{2}L^{2}}^{2}+\|e^{b\nu^{\frac13}t}\partial_{x}h_{3,1}\|_{L^{2}L^{2}}^{2}
	+\|e^{b\nu^{\frac13}t}\nabla h_{2,2}\|_{L^{2}L^{2}}^{2}+\|e^{b\nu^{\frac13}t}\nabla h_{3,2}\|_{L^{2}L^{2}}^{2}\\
	+\|e^{b\nu^{\frac13}t}\triangle P^{3}\|_{L^{2}L^{2}}^{2}+\|e^{b\nu^{\frac13}t}\nabla h_{2,5}\|_{L^{2}L^{2}}^{2}+\|e^{b\nu^{\frac13}t}\nabla h_{3,5}\|_{L^{2}L^{2}}^{2}\leq C\nu^{-1}E_{2}^{2}E_{7},$
	\item[(ii)] 
	$\|e^{b\nu^{\frac13}t}\nabla(h_{3,1})_{\neq}\|_{L^{2}L^{2}}^{2}\leq C\nu^{-\frac53}E_{2}^{2}\left(E_{7}+\nu^{\frac43}\|\triangle u_{3,\neq}\|_{X_{b}}^{2} \right),$
	\item[(iii)] 
	$\|e^{b\nu^{\frac13}t}\nabla(h_{2,7})_{\neq}\|_{L^{2}L^{2}}^{2}
	+\|e^{b\nu^{\frac13}t}\nabla(h_{3,7})_{\neq}\|_{L^{2}L^{2}}^{2}\leq C\nu E_{1}^{2}E_{7}.$
\end{itemize}
\end{lemma}
\begin{proof}
{\underline{\textbf{Estimate $ \nabla H_{2,1}. $}}} Recall $$ H_{2,1}=\left(u_{2,0}\partial_{y}+u_{3,0}\partial_{z} \right)Q-u_{3,\neq}\left(u_{2,0}\partial_{y}+u_{3,0}\partial_{z} \right)\kappa, $$ 
where $ Q=u_{2,\neq}+\kappa u_{3,\neq}, $ then using  Lemma \ref{zero mode}, we get
\begin{equation}\label{H_21}
	\begin{aligned}
		&\|e^{b\nu^{\frac13}t}\nabla H_{2,1}\|_{L^{2}L^{2}}^{2}\leq C\nu^{-1}\|e^{b\nu^{\frac13}t}H_{2,1}\|_{Y_{0}}^{2}\\\leq&C\nu^{-1}\left(\|e^{b\nu^{\frac13}t}(u_{2,0}\partial_{y}+u_{3,0}\partial_{z})Q\|_{Y_{0}}^{2}+\|e^{b\nu^{\frac13}t}u_{3,\neq}(u_{2,0}\partial_{y}+u_{3,0}\partial_{z})\kappa\|_{Y_{0}}^{2} \right)\\\leq&C\nu^{-1}E_{2}^{2}\left(\|e^{b\nu^{\frac13}t}\nabla Q\|_{Y_{0}}^{2}+\|e^{b\nu^{\frac13}t}u_{3,\neq}\nabla\kappa\|_{Y_{0}}^{2} \right).
	\end{aligned}
\end{equation}
According to Lemma \ref{kappa}, due to $ \|\nabla\kappa\|_{L^{\infty}}\leq C\|\kappa\|_{H^{3}}\leq C\|\widehat{u_{1,0}}\|_{H^{4}}\leq C, $ there holds
\begin{equation*}
	\begin{aligned}
		&\|u_{3,\neq}\nabla\kappa\|_{L^{2}}\leq \|u_{3,\neq}\|_{L^{2}}\|\nabla\kappa\|_{L^{\infty}}\leq C\|\partial_{x}^{2}u_{3,\neq}\|_{L^{2}},\\
		&\|\nabla(u_{3,\neq}\nabla\kappa)\|_{L^{2}}\leq C\|\kappa\|_{H^3}^2\|\nabla\partial_{x}^{2}u_{3,\neq}\|_{L^{2}}\leq C\|\nabla\partial_{x}^{2}u_{3,\neq}\|_{L^{2}},
	\end{aligned}
\end{equation*}
which implies that 
\begin{equation}\label{U3 k}
	\begin{aligned}
		\|e^{b\nu^{\frac13}t}u_{3,\neq}\nabla\kappa\|_{Y_{0}}^{2}\leq&C\|\partial_{x}^{2}u_{3,\neq}\|_{X_{b}}^{2}\leq CE_7.
	\end{aligned}
\end{equation}
On the other hand, we have
\begin{equation}\label{Q'}
	\|e^{b\nu^{\frac13}t}\nabla Q\|_{Y_{0}}^{2}\leq\|e^{b\nu^{\frac13}t}\partial_{x}\nabla Q\|_{Y_{0}}^{2}\leq \|\partial_{x}\nabla Q\|_{X_{b}}^{2}\leq E_{7}.
\end{equation}

It follows from (\ref{H_21}), (\ref{U3 k}) and (\ref{Q'}) that
\begin{equation}\label{result H21}
	\|e^{b\nu^{\frac13}t}\nabla H_{2,1}\|_{L^{2}L^{2}}^{2}\leq C\nu^{-1}E_{2}^{2}E_{7}.
\end{equation}

{\underline{\textbf{Estimate $ \partial_{x}h_{3,1} $ and $ \nabla(h_{3,1})_{\neq}. $}}}	Due to $ \partial_{x}h_{3,1}=\left(u_{2,0}\partial_{y}+u_{3,0}\partial_{z} \right)\partial_{x}u_{3}, $ using  Lemma \ref{zero mode}, there holds
\begin{equation}\label{result h31}
	\begin{aligned}
		\|e^{b\nu^{\frac13}t}\partial_{x}h_{3,1}\|_{L^{2}L^{2}}^{2}
		&\leq C\|(u_{2,0},u_{3,0})\|_{L^{\infty}L^{\infty}}^{2}\|e^{b\nu^{\frac13}t}\partial_x\nabla u_{3,\neq}\|_{L^2L^2}^{2}\\
		&\leq C\nu^{-1}E_{2}^{2}\|\partial_{x}^{2}u_{3,\neq}\|_{X_{b}}^{2}\leq C\nu^{-1}E_{2}^{2}E_{7}.
	\end{aligned}
\end{equation}
By  Lemma \ref{zero mode} and 
$$\|\nabla u_{3,\neq}\|_{L^{2}}^{2}\leq\|u_{3,\neq}\|_{L^{2}}\|\triangle u_{3,\neq}\|_{L^{2}}\leq 
\nu^{-\frac23}\|\partial_{x}^{2}u_{3,\neq}\|_{L^{2}}^{2}+\nu^{\frac23}\|\triangle u_{3,\neq}\|_{L^{2}}^{2},$$
we get
\begin{equation}\label{h31}
	\begin{aligned}
		&\|e^{b\nu^{\frac13}t}\nabla(h_{3,1})_{\neq}\|_{L^{2}L^{2}}^{2}
		=\|e^{b\nu^{\frac13}t}\nabla\left( u_{2,0}\partial_{y}u_{3,\neq}+u_{3,0}\partial_{z}u_{3,\neq} \right)\|_{L^{2}L^{2}}^{2}
		\\\leq&\|\nabla (u_{2,0},u_{3,0})\|_{L^{2}L^{\infty}}^{2}\|e^{b\nu^{\frac13}t}\nabla u_{3,\neq}\|_{L^{\infty}L^{2}}^{2}+\|(u_{2,0},u_{3,0})\|_{L^{\infty}L^{\infty}}^{2}\|e^{b\nu^{\frac13}t}\triangle u_{3,\neq}\|_{L^{2}L^{2}}^{2}\\\leq&C\nu^{-1}E_{2}^{2}\|\nabla u_{3,\neq}\|_{X_{b}}^{2}\leq C\nu^{-\frac53}E_{2}^{2}\big(\|\partial_{x}^{2}u_{3,\neq}\|_{X_{b}}^{2}+\nu^{\frac43}\|\triangle u_{3,\neq}\|_{X_{b}}^{2} \big).
	\end{aligned}
\end{equation}

{\underline{\textbf{Estimate $ \nabla h_{2,2}. $}}} 
For $ h_{2,2}=u_{\neq}\cdot\nabla u_{2,0} $ and $ s\in\{2,3\}, $ it follows from Lemma \ref{sob_14} and  Lemma \ref{zero mode} that
\begin{equation*}
	\begin{aligned}
		\|\nabla h_{2,2}\|_{L^{2}}=&\|\nabla\left(u_{s,\neq}\partial_{s}u_{2,0} \right)\|_{L^{2}}\leq C\|\partial_{s}u_{2,0}\|_{H^{1}}\left(\|u_{s,\neq}\|_{H^{1}}+\|(\partial_{z}-\kappa\partial_{y}) u_{s,\neq}\|_{H^{1}}\right)\\\leq&C\|\nabla u_{2,0}\|_{H^{1}}\left(\|\nabla\partial_{x}^{2}u_{s,\neq}\|_{L^{2}}+\|\nabla\partial_{x}(\partial_{z}-\kappa\partial_{y})u_{s,\neq}\|_{L^{2}} \right)
	\end{aligned}
\end{equation*}	
and
\begin{equation}\label{h22}
	\begin{aligned}
		\|e^{b\nu^{\frac13}t}\nabla h_{2,2}\|_{L^{2}L^{2}}^{2}\leq&CE_{2}^{2}\left(\|e^{b\nu^{\frac13}t}\nabla\partial_{x}^{2}u_{s,\neq}\|_{L^{2}L^{2}}^{2}+\|e^{b\nu^{\frac13}t}\nabla\partial_{x}(\partial_{z}-\kappa\partial_{y})u_{s,\neq}\|_{L^{2}L^{2}}^{2} \right)\\\leq&C\nu^{-1}E_{2}^{2}\left(\|\partial_{x}^{2}u_{s,\neq}\|_{X_{b}}^{2}+\|\partial_{x}(\partial_{z}-\kappa\partial_{y})u_{s,\neq}\|_{X_{b}}^{2} \right)\leq C\nu^{-1}E_{2}^{2}E_{7}.
	\end{aligned}
\end{equation}

{\underline{\textbf{Estimate $ \nabla h_{3,2}. $}}} For $ h_{3,2}=u_{\neq}\cdot\nabla u_{3,0}, $ we rewrite it into
\begin{equation*}
	h_{3,2}=\left(u_{2,\neq}\partial_{y}+u_{3,\neq}\partial_{z} \right)u_{3,0}=Q\partial_{y}u_{3,0}+u_{3,\neq}\left(\partial_{z}u_{3,0}-\kappa\partial_{y}u_{3,0} \right),
\end{equation*}
where $ Q=u_{2,\neq}+\kappa u_{3,\neq}. $ Then
\begin{equation}\label{nabla h32}
	\begin{aligned}
		\|e^{b\nu^{\frac13}t}\nabla h_{3,2}\|_{L^{2}L^{2}}^{2}\leq&\|e^{b\nu^{\frac13}t}\nabla(Q\partial_{y}u_{3,0})\|_{L^{2}L^{2}}^{2}+\|e^{b\nu^{\frac13}t}\nabla\left(u_{3,\neq}(\partial_{z}u_{3,0}-\kappa\partial_{y}u_{3,0}) \right)\|_{L^{2}L^{2}}^{2}\\:=&I_{1}+I_{2}.
	\end{aligned}
\end{equation}
By Lemma \ref{zero mode}, we get
\begin{equation*}
	\begin{aligned}
		I_{1}\leq C\nu^{-1}E_{2}^{2}\|e^{b\nu^{\frac13}t}\nabla\partial_{x}Q\|_{L^{\infty}L^{2}}^{2}\leq C\nu^{-1}E_{2}^{2}E_{7}.
	\end{aligned}
\end{equation*}
For $ I_{2}, $ it follows from Lemma \ref{sob_14},  Lemma \ref{zero mode} and Lemma \ref{U10} that
\begin{equation*}
	\begin{aligned}
		&\|\nabla\left(u_{3,\neq}(\partial_{z}u_{3,0}-\kappa\partial_{y}u_{3,0}) \right)\|_{L^{2}}
		\\\leq&C\left(\|\partial_{z}u_{3,0}\|_{H^{1}}+\|\kappa\nabla u_{3,0}\|_{H^{1}} \right)\left(\|u_{3,\neq}\|_{H^{1}}+\|(\partial_{z}-\kappa\partial_{y})u_{3,\neq}\|_{H^{1}} \right)\\\leq& C\left(E_{2}+E_{1}E_{2}\right)\left(\|\nabla\partial_{x}^{2}u_{3,\neq}\|_{L^{2}}+\|\nabla\partial_{x}(\partial_{z}-\kappa\partial_{y})u_{3,\neq}\|_{L^{2}} \right)
	\end{aligned}
\end{equation*}
and
\begin{equation*}
	\begin{aligned}
		I_{2}\leq &CE_{2}^{2}\left(\|e^{b\nu^{\frac13}t}\nabla\partial_{x}^{2}u_{3,\neq}\|_{L^{2}L^{2}}^{2}+\|e^{b\nu^{\frac13}t}\nabla\partial_{x}(\partial_{z}-\kappa\partial_{y})u_{3,\neq}\|_{L^{2}L^{2}}^{2} \right)\\\leq& CE_{2}^{2}\nu^{-1}\left(\|\partial_{x}^{2}u_{3,\neq}\|_{X_{b}}^{2}+\|\partial_{x}(\partial_{z}-\kappa\partial_{y})u_{3,\neq}\|_{X_{b}}^{2} \right)\leq C\nu^{-1}E_{2}^{2}E_{7},
	\end{aligned}
\end{equation*}
where we use $E_1\leq C.$

Combining the estimates of $ I_{1} $ with $ I_{2}, $ (\ref{nabla h32}) shows that
\begin{equation}\label{h32}
	\|e^{b\nu^{\frac13}t}\nabla h_{3,2}\|_{L^{2}L^{2}}^{2}\leq C\nu^{-1}E_{2}^{2}E_{7}.
\end{equation}

{\underline{\textbf{Estimate $ \triangle P^{3}, \nabla h_{2,5} $ and $ \nabla h_{3,5}. $ }}}
Due to $ \partial_{y}u_{2,0}+\partial_{z}u_{3,0}=0, $ for $ \alpha, \beta\in\{2,3\}, $ there holds
\begin{equation*}
	\begin{aligned}
		\triangle P^{3}=-2\partial_{\alpha}u_{\beta,0}\partial_{\beta}u_{\alpha,\neq}=-2\partial_{\beta}\left(\partial_{\alpha}u_{\beta,0}u_{\alpha,\neq}\right)= -2\left(\partial_{y}h_{2,2}+\partial_{z}h_{3,2} \right),
	\end{aligned}
\end{equation*}
which follows that
\begin{equation}\label{p3}
	\begin{aligned}
		\|e^{b\nu^{\frac13}t}\triangle P^{3}\|_{L^{2}L^{2}}^{2}\leq C\left(\|e^{b\nu^{\frac13}t}\nabla h_{2,2}\|_{L^{2}L^{2}}^{2}+\|e^{b\nu^{\frac13}t}\nabla h_{3,2}\|_{L^{2}L^{2}}^{2} \right)\leq C\nu^{-1}E_{2}^{2}E_{7},
	\end{aligned}
\end{equation}
where we use (\ref{h22}) and (\ref{h32}).

Recall $ h_{s,5}=\partial_{s}P^{3} $ for $ s\in\{2,3\}, $ then there holds
\begin{equation}\label{h25 h35}
	\|e^{b\nu^{\frac13}t}\nabla h_{s,5}\|_{L^{2}L^{2}}^{2}\leq C\|e^{b\nu^{\frac13}t}\triangle P^{3}\|_{L^{2}L^{2}}^{2}\leq C\nu^{-1}E_{2}^{2}E_{7}.
\end{equation}

Collecting (\ref{result H21}), (\ref{result h31}), (\ref{h22}), (\ref{h32}), (\ref{p3}) and (\ref{h25 h35}), the first inequality holds.	

\underline{\textbf{Estimate $ \nabla (h_{2,7})_{\neq} $ and 
		$\nabla (h_{3,7})_{\neq}$.}} Recall $ h_{j,7}=\partial_{j}P^{0} $ for $ j=2,3, $ and
$$ \triangle P^{0}=-2\left(\partial_{y}\widetilde{u_{1,0}}\partial_{x}u_{2,\neq}+\partial_{z}\widetilde{u_{1,0}}\partial_{x}u_{3,\neq} \right), $$
one deduces
\begin{equation}\label{P0}
	\begin{aligned}
		&\|e^{b\nu^{\frac13}t}\nabla(h_{2,7})_{\neq}\|_{L^{2}L^{2}}^{2}+\|e^{b\nu^{\frac13}t}\nabla (h_{3,7})_{\neq}\|_{L^{2}L^{2}}^{2}\leq\|e^{b\nu^{\frac13}t}\triangle P^{0}_{\neq}\|_{L^{2}L^{2}}^{2}\\\leq&C\big(\|e^{b\nu^{\frac13}t}\partial_{y}\widetilde{u_{1,0}}\partial_{x}u_{2,\neq}\|_{L^{2}L^{2}}^{2}+\|e^{b\nu^{\frac13}t}\partial_{z}\widetilde{u_{1,0}}\partial_{x}u_{3,\neq}\|_{L^{2}L^{2}}^{2} \big).
	\end{aligned}
\end{equation}
Using Lemma \ref{sob_14}, for $j=2,3,$ we arrive
\begin{equation*}
	\|\partial_{j}\widetilde{u_{1,0}}\partial_{x}u_{j,\neq}\|_{L^{2}}\leq C\|\partial_{j}\widetilde{u_{1,0}}\|_{H^{1}}\left(\|\partial_{x}u_{j,\neq}\|_{L^{2}}+\|(\partial_{z}-\kappa\partial_{j})\partial_{x}u_{j,\neq}\|_{L^{2}} \right),
\end{equation*}
which along with (\ref{P0}), there holds
\begin{equation*}
	\begin{aligned}
		\|e^{b\nu^{\frac13}t}\nabla(h_{2,7})_{\neq}\|_{L^{2}L^{2}}^{2}+\|e^{b\nu^{\frac13}t}\nabla (h_{3,7})_{\neq}\|_{L^{2}L^{2}}^{2}
		\leq C\nu^{-\frac13}E_{1,2}^{2}E_{7}\leq C\nu E_{1}^{2}E_{7}.
	\end{aligned}
\end{equation*}
\end{proof}

\begin{lemma}\label{lem: u0 u_neq}
It holds that
\begin{equation}\label{u0 neq}
	\begin{aligned}
		&\|e^{a\nu^{\frac13}t}\nabla(u_{0}\cdot\nabla u_{\neq})\|_{L^{2}L^{2}}^{2}\leq C\nu^{\frac13}E_{1}^{2}E_{4}E_{6}+C\nu^{-\frac53}E_{2}^{2}E_{4}^{2},\\
		&\|e^{a\nu^{\frac13}t}\nabla(u_{\neq}\cdot\nabla u_{1,0})\|_{L^{2}L^{2}}^{2}\leq C\nu^{\frac13}E_{1}^{2}E_{7},\\
		&\|e^{a\nu^{\frac13}t}\partial_{x}(u_{\neq}\cdot\nabla u_{3,0})\|_{L^{2}L^{2}}^{2}\leq C\nu^{-1}E_{2}^{2}E_{4}^{2}.
	\end{aligned}
\end{equation}
\end{lemma}
\begin{proof}
{\underline{\bf{Estimate $ (\ref{u0 neq})_{1}. $}}} As $ u_{0}\cdot\nabla u_{\neq}=u_{1,0}\partial_{x}u_{\neq}+u_{2,0}\partial_{y}u_{\neq}+u_{3,0}\partial_{z}u_{\neq} $ and $ \int_{\mathbb{T}}u_{2,0}dz=0,$ using Lemma \ref{U10}, we get
\begin{equation*}
	\begin{aligned}
		&\|\nabla(u_{0}\cdot\nabla u_{\neq})\|_{L^{2}} \\\leq&C\|u_{1,0}\|_{H^{2}}\|\nabla\partial_{x}u_{\neq}\|_{L^{2}}+C\|u_{2,0}\|_{H^{2}}\|\partial_{y}u_{\neq}\|_{H^{1}}+C\|u_{3,0}\|_{H^{1}}\|\partial_{z}u_{\neq}\|_{H^{2}}\\\leq&CE_{1}\nu^{\frac23}(1+\nu^{\frac13}t)\|\nabla\partial_{x}u_{\neq}\|_{L^{2}}+CE_2\|\triangle(\partial_{x},\partial_{z})u_{\neq}\|_{L^{2}}.
	\end{aligned}
\end{equation*}
Combining it with Lemma \ref{lem omega} and
\begin{equation}\label{1-nu t}
	(1+\nu^{\frac13}t)^{2}\leq Ce^{\frac{b-a}{2}\nu^{\frac13}t},
\end{equation}
there holds
\begin{equation}\label{u0 u_neq}
	\begin{aligned}
		\|e^{a\nu^{\frac13}t}\nabla(u_{0}\cdot\nabla u_{\neq})\|_{L^{2}L^{2}}^{2}
		\leq C\nu^{\frac13}E_{1}^{2}E_{4}E_{6}+C\nu^{-\frac53}E_{2}^{2}E_{4}^{2}.
	\end{aligned}
\end{equation}

{\underline{\bf{Estimate $ (\ref{u0 neq})_{2}. $}}}	Thanks to Lemma \ref{sob_14} and Lemma \ref{U10}, for $ k\in\{2,3\}, $ we have
\begin{equation*}
	\begin{aligned}
		\|\nabla(u_{k,\neq}\partial_{k}u_{1,0})\|_{L^{2}}\leq&C\|u_{1,0}\|_{H^{2}}\left(\|\nabla u_{k,\neq}\|_{L^{2}}+\|\nabla(\partial_{z}-\kappa\partial_{y})u_{k,\neq}\|_{L^{2}} \right)\\\leq&CE_{1}\nu^{\frac23}(1+\nu^{\frac13}t)\left(\|\nabla u_{k,\neq}\|_{L^{2}}+\|\nabla(\partial_{z}-\kappa\partial_{y})u_{k,\neq}\|_{L^{2}} \right).
	\end{aligned}
\end{equation*}
Then by using (\ref{1-nu t}), one deduces
\begin{equation*}
	\begin{aligned}
		&\|e^{a\nu^{\frac13}t}\nabla(u_{\neq}\cdot\nabla u_{1,0})\|_{L^{2}L^{2}}^{2}\leq\sum_{k=2}^{3}\|e^{a\nu^{\frac13}t}\nabla(u_{k,\neq}\partial_{k}u_{1,0})\|_{L^{2}L^{2}}^{2}\\\leq&CE_{1}^{2}\nu^{\frac43}\sum_{k=2}^{3}\left(\|e^{b\nu^{\frac13}t}\nabla\partial_{x}^{2}u_{k,\neq}\|_{L^{2}L^{2}}^{2}+\|e^{b\nu^{\frac13}t}\nabla\partial_{x}(\partial_{z}-\kappa\partial_{y})u_{k,\neq}\|_{L^{2}L^{2}}^{2} \right)\leq C\nu^{\frac13}E_{1}^{2}E_{7}.
	\end{aligned}
\end{equation*}

{\underline{\bf{Estimate $ (\ref{u0 neq})_{3}. $}}} For $ k\in\{2,3\}, $ it follows from Lemma \ref{lem zero nonzero} that
\begin{equation*}
	\|\partial_{x}u_{k,\neq}\partial_{k}u_{3,0}\|_{L^{2}}\leq C\|\partial_{k}u_{3,0}\|_{H^{1}}\left(\|\partial_{x}u_{k,\neq}\|_{L^{2}}+\|\partial_{z}\partial_{x}u_{k,\neq}\|_{L^{2}} \right),
\end{equation*}
which implies that 
\begin{equation*}
	\begin{aligned}
		&\|e^{a\nu^{\frac13}t}\partial_{x}(u_{\neq}\cdot\nabla u_{3,0})\|_{L^{2}L^{2}}^{2}
		\leq \|e^{a\nu^{\frac13}t}\partial_{x}u_{\neq}\cdot\nabla u_{3,0}\|_{L^{2}L^{2}}^{2}
		\\\leq&C\nu^{-1}\left(\|u_{3,0}\|_{Y_{0}}^{2}+\|\nabla u_{3,0}\|_{Y_{0}}^{2} \right)\left(\|\triangle u_{2,\neq}\|_{X_{a}}^{2}+\|(\partial_{x}^{2}+\partial_{z}^{2})u_{3,\neq}\|_{X_{a}}^{2} \right)
		\leq C\nu^{-1}E_{2}^{2}E_{4}^{2}.
	\end{aligned}
\end{equation*}
\end{proof}

%%%%%%%%%%
\section{Energy estimates for zero modes}\label{7}
%We assume that
%\begin{equation}\label{ass E1 E21 E3}
%	E_{1}\leq\varepsilon_{0},\quad E_{2}\leq\varepsilon_{0}\nu,\quad E_{4}\leq\varepsilon_{0}\nu,
%\end{equation}
%where $\varepsilon_{0}\in(0,\delta_{4})$ is a sufficiently small constant independent of $ \nu $ and $ T. $
\subsection{Estimate $E_{1}$}
\begin{proposition}\label{prop:E1}
It holds that
\begin{equation*}
	E_{1}\leq C\left(\nu^{-\frac23}\|u_{\rm in}\|_{H^{2}}+\nu^{-1}E_{2}+\nu^{-2}E_{4}^{2} \right).
\end{equation*}
%	\begin{equation*}
	%		E_{1,1}\leq C\left(\|u_{0}(0)\|_{H^{4}}+\nu^{-1}E_{2} \right),
	%	\end{equation*}
%	\begin{equation*}
	%		E_{1,2}\leq C\left(\|u_{0}(0)\|_{H^{2}}+\nu^{-\frac43}E_{4}^{2} \right).
	%	\end{equation*}
\end{proposition}
\begin{proof}
{\underline{\textbf{Step I. Estimate $ E_{1,1}. $}}}
Multiplying $ (\ref{decom u10})_{1} $ by $ 2\widehat{u_{1,0}} $ and integrating with $ (y,z) $ over $ \mathbb{R}\times\mathbb{T}, $ and noting that $ \partial_{y}u_{2,0}+\partial_{z}u_{3,0}=0 $ and $ \frac{1}{|\mathbb{T}|}\int_{\mathbb{T}}u_{2,0}dz=0, $ we obtain
\begin{equation*}
	\begin{aligned}
		&\partial_{t}\|\widehat{u_{1,0}}\|_{L^{2}}^{2}+2\nu\|\nabla\widehat{u_{1,0}}\|_{L^{2}}^{2}\\=&-2\int_{\mathbb{R}\times\mathbb{T}}u_{2,0}\widehat{u_{1,0}}dydz\leq C\|\partial_{z}u_{2,0}\|_{L^{2}}\|\partial_{z}\widehat{u_{1,0}}\|_{L^{2}}\leq C\|\nabla\triangle u_{2,0}\|_{L^{2}}\|\nabla\widehat{u_{1,0}}\|_{L^{2}},
	\end{aligned}
\end{equation*}
which follows that
\begin{equation}\label{widehat u10}
	\|\widehat{u_{1,0}}\|_{Y_{0}}^{2}\leq C\nu^{-1}\|\nabla\triangle u_{2,0}\|_{L^{2}L^{2}}^{2}\leq C\nu^{-2}E_{2}^{2}. 
\end{equation}
Taking $ \triangle^{2} $ for $ (\ref{decom u10})_{1}, $ the energy estimate yields that
\begin{equation*}
	\begin{aligned}
		&\partial_{t}\|\triangle^{2}\widehat{u_{1,0}}\|_{L^{2}}^{2}+2\nu\|\nabla\triangle^{2}\widehat{u_{1,0}}\|_{L^{2}}^{2}\\\leq&C\left(\|\nabla\triangle u_{2,0}\|_{L^{2}}+\|\nabla\triangle\left(u_{2,0}\partial_{y}\widehat{u_{1,0}}+u_{3,0}\partial_{z}\widehat{u_{1,0}} \right)\|_{L^{2}} \right)\|\nabla\triangle^{2}\widehat{u_{1,0}}\|_{L^{2}}.
	\end{aligned}
\end{equation*}
From this, along with Lemma \ref{lem u20 u30}, we get
\begin{equation}\label{widehat u10''''}
	\begin{aligned}
		\|\triangle^{2}\widehat{u_{1,0}}\|_{Y_{0}}^{2}
		%			\\\leq&C\nu^{-1}\left(\|\nabla\triangle u_{2,0}\|_{L^{2}L^{2}}^{2}+\|\nabla\triangle\left(u_{2,0}\partial_{y}\widehat{u_{1,0}}+u_{3,0}\partial_{z}\widehat{u_{1,0}} \right)\|_{L^{2}L^{2}}^{2} \right)
		\leq&C\nu^{-2}\left(\|\triangle u_{2,0}\|_{Y_{0}}^{2}+\|\triangle\left(u_{2,0}\partial_{y}\widehat{u_{1,0}}+u_{3,0}\partial_{z}\widehat{u_{1,0}} \right)\|_{Y_{0}}^{2} \right)
		\\\leq&C\left(\nu^{-2}E_{2}^{2}+\nu^{-2}E_{1,1}^{2}E_{2}^{2} \right).
	\end{aligned}
\end{equation}

Combining (\ref{widehat u10}) with (\ref{widehat u10''''}), there holds
\begin{equation}\label{hat u10}
	\begin{aligned}
		\|\widehat{u_{1,0}}\|_{L^{\infty}H^{4}}
		+\nu^{\frac12}\|\nabla\widehat{u_{1,0}}\|_{L^{2}H^{4}}
		\leq&C\left(\|\widehat{u_{1,0}}\|_{Y_{0}}
		+\|\triangle^{2}\widehat{u_{1,0}}\|_{Y_{0}} \right)
		\leq C\nu^{-1}\left(E_{2}+E_{1,1}E_{2} \right).
	\end{aligned}
\end{equation}
We also need to estimate $ \|\partial_{t}\widehat{u_{1,0}}\|_{L^{\infty}H^{2}}. $ It follows from $ (\ref{decom u10}) $,  Lemma \ref{zero mode} and Lemma \ref{lem u20 u30} that
\begin{equation}\label{partial t u10}
	\begin{aligned}
		\|\partial_{t}\widehat{u_{1,0}}\|_{H^{2}}\leq&\nu\|\triangle\widehat{u_{1,0}}\|_{H^{2}}+\|u_{2,0}\|_{H^{2}}+\|u_{2,0}\partial_{y}\widehat{u_{1,0}}\|_{H^{2}}+\|u_{3,0}\partial_{z}\widehat{u_{1,0}}\|_{H^{2}}\\\leq&\nu\|\widehat{u_{1,0}}\|_{L^{\infty}H^{4}}+CE_{2}+CE_{1,1}E_{2}.
	\end{aligned}
\end{equation}
From (\ref{hat u10}) and (\ref{partial t u10}), we deduce
\begin{equation*}
	\begin{aligned}
		E_{1,1}=&\|\widehat{u_{1,0}}\|_{L^{\infty}H^{4}}+\nu^{\frac12}\|\nabla\widehat{u_{1,0}}\|_{L^{2}H^{4}}+\nu^{-1}\|\partial_{t}\widehat{u_{1,0}}\|_{L^{\infty}H^{2}}\\\leq&C\left(\nu^{-1}E_{2}+\nu^{-1}E_{1,1}E_{2} \right).
	\end{aligned}
\end{equation*}
As $ E_{2}\leq \varepsilon_{0}\nu, $ taking $ \varepsilon_{0} $ small enough such that $ C\varepsilon_{0}<\frac12, $ one obtains
\begin{equation}\label{E11 end}
	E_{1,1}\leq C\nu^{-1}E_{2}.
\end{equation}

{\underline{\textbf{Step II. Estimate $ E_{1,2}. $}}}
As $ \widetilde{u_{1,0}} $ satisfies
\begin{equation}\label{tilde u10}
	\left\{
	\begin{array}{lr}
		\partial_{t}\widetilde{u_{1,0}}-\nu\triangle\widetilde{u_{1,0}}+u_{2,0}\partial_{y}\widetilde{u_{1,0}}+u_{3,0}\partial_{z}\widetilde{u_{1,0}}+(u_{\neq}\cdot\nabla u_{1,\neq})_{0}=0, \\
		\widetilde{u_{1,0}}|_{t=0}=(u_{\rm in})_0, 
	\end{array}
	\right.
\end{equation}
and $ {\rm div}~u=0, $ then $L^{2}$ energy estimate gives
\begin{equation*}
	\begin{aligned}
		\partial_{t}\|\widetilde{u_{1,0}}\|_{L^{2}}^{2}+2\nu\|\nabla\widetilde{u_{1,0}}\|_{L^{2}}^{2}\leq C\||u_{\neq}|^{2}\|_{L^{2}}\|\nabla\widetilde{u_{1,0}}\|_{L^{2}},
	\end{aligned}
\end{equation*}
which along with Lemma \ref{lem 2} indicate that
\begin{equation}\label{widetilde u10}
	\begin{aligned}
		\|\widetilde{u_{1,0}}\|_{Y_{0}}^{2}\leq C\left(\|(u_{\rm in})_0\|_{L^{2}}^{2}+\nu^{-1}\||u_{\neq}|^{2}\|_{L^{2}L^{2}}^{2} \right)\leq C\left(\|(u_{\rm in})_0\|_{L^{2}}^{2}+\nu^{-2}E_{4}^{4} \right).
	\end{aligned}
\end{equation}

Taking $ \triangle $ for $ (\ref{tilde u10})_{1},$  energy estimate indicates
\begin{equation*}
	\begin{aligned}
		&\partial_{t}\|\triangle\widetilde{u_{1,0}}\|_{L^{2}}^{2}+2\nu\|\nabla\triangle\widetilde{u_{1,0}}\|_{L^{2}}^{2}\\=&2\int_{\mathbb{R}\times\mathbb{T}}\nabla\left(u_{2,0}\partial_{y}\widetilde{u_{1,0}}+u_{3,0}\partial_{z}\widetilde{u_{1,0}}+(u_{\neq}\cdot\nabla u_{1,\neq})_{0} \right)\cdot\nabla\triangle\widetilde{u_{1,0}}dydz\\\leq&C\left(\|\nabla\left(u_{2,0}\partial_{y}\widetilde{u_{1,0}}+u_{3,0}\partial_{z}\widetilde{u_{1,0}} \right)\|_{L^{2}}+\|\nabla(u_{\neq}\cdot\nabla u_{1,\neq})_{0}\|_{L^{2}} \right)\|\nabla\triangle\widetilde{u_{1,0}}\|_{L^{2}},
	\end{aligned}
\end{equation*}
then we get
\begin{equation}\label{tilde u10''}
	\begin{aligned}
		\|\triangle\widetilde{u_{1,0}}\|_{Y_{0}}^{2}
		\leq& C\big(\|(u_{\rm in})_0\|_{H^{2}}^{2}+\nu^{-1}\|\nabla(u_{2,0}\partial_{y}\widetilde{u_{1,0}}
		+u_{3,0}\partial_{z}\widetilde{u_{1,0}})\|_{L^{2}L^{2}}^{2}
		\\&+\nu^{-1}\|\nabla(u_{\neq}\cdot\nabla u_{1,\neq})_{0}\|_{L^{2}L^{2}}^{2} \big).
	\end{aligned}
\end{equation}
Due to $ \frac{1}{|\mathbb{T}|}\int_{\mathbb{T}}u_{2,0}dz=0, $ there holds
\begin{equation*}
	\|\nabla(u_{2,0}\partial_{y}\widetilde{u_{1,0}})\|_{L^{2}}\leq C\|u_{2,0}\|_{H^{1}}\|\partial_{y}\widetilde{u_{1,0}}\|_{H^{2}}\leq C\|\triangle u_{2,0}\|_{L^{2}}\|\nabla\widetilde{u_{1,0}}\|_{H^{2}}.
\end{equation*}
Similarly, we have
\begin{equation*}
	\|\nabla(u_{3,0}\partial_{z}\widetilde{u_{1,0}})\|_{L^{2}}\leq C\|u_{3,0}\|_{H^{1}}\|\partial_{z}\widetilde{u_{1,0}}\|_{H^{2}}\leq C\left(\|u_{3,0}\|_{L^{2}}+\|\nabla u_{3,0}\|_{L^{2}} \right)\|\nabla\widetilde{u_{1,0}}\|_{H^{2}}.
\end{equation*}
Using the above estimates and Lemma \ref{lem 2}, we get by (\ref{tilde u10''}) that
\begin{equation}\label{widetilde u10''}
	\begin{aligned}
		\|\triangle\widetilde{u_{1,0}}\|_{Y_{0}}^{2}\leq C\big(\|u_{\rm in}\|_{H^{2}}^{2}
		+\nu^{-2}E_{2}^{2}E_{1,2}^{2}+\nu^{-\frac83}E_{4}^{4} \big).
	\end{aligned}
\end{equation}
It follows from (\ref{widetilde u10}) and (\ref{widetilde u10''}) that
\begin{equation*}
	\begin{aligned}
		E_{1,2}=&\|\widetilde{u_{1,0}}\|_{L^{\infty}H^{2}}
		+\nu^{\frac12}\|\nabla\widetilde{u_{1,0}}\|_{L^{2}H^{2}}
		\leq C\left(\|\widetilde{u_{1,0}}\|_{Y_{0}}
		+\|\triangle\widetilde{u_{1,0}}\|_{Y_{0}} \right)\\\leq&C\left(\|u_{\rm in}\|_{H^{2}}+\nu^{-1}E_{2}E_{1,2}+\nu^{-\frac43}E_{4}^{2} \right).
	\end{aligned}
\end{equation*}
Due to $ E_{2}\leq \varepsilon_{0}\nu, $ taking $ \varepsilon_{0} $ small enough such that $ C\varepsilon_{0}<\frac12, $ there holds
\begin{equation}\label{e12 end}
	E_{1,2}\leq C\left(\|u_{\rm in}\|_{H^{2}}+\nu^{-\frac43}E_{4}^{2} \right).
\end{equation}

Since $ E_{1}=E_{1,1}+\nu^{-\frac23}E_{1,2},$ combining (\ref{E11 end}) with (\ref{e12 end}), the proof is complete.
\end{proof}

\subsection{Estimate $ E_{2} $}
\begin{proposition}\label{prop:E21}
It holds that
\begin{equation*}
	E_{2}\leq C\left(\|u_{\rm in}\|_{H^{2}}+\nu^{-1}E_{4}^{2}+\nu^{-1}E_{3} \right).
\end{equation*}
\end{proposition}
\begin{proof}
Recall (\ref{23})-(\ref{hj}), the zero modes for $ u_{2} $ and $ u_{3} $ satisfy
\begin{equation}\label{eq:u20 u30 theta}
	\left\{
	\begin{array}{lr}
		\partial_{t}u_{2,0}-\nu\triangle u_{2,0}+\partial_{y}P^{5}+h_{2,0}=\Theta_{0}-\partial_{y}P^{N_{3}}_{0},\\
		\partial_{t}u_{3,0}-\nu\triangle u_{3,0}+\partial_{z}P^{5}+h_{3,0}=-\partial_{z}P^{N_{3}}_{0}.
	\end{array}
	\right.
\end{equation}
Thanks to $ \partial_{y}u_{2,0}+\partial_{z}u_{3,0}=0,$  energy estimates yield that
\begin{equation}\label{U20 U30 energy}
	\begin{aligned}
		\partial_{t}\left(\|u_{2,0}\|_{L^{2}}^{2}+\|u_{3,0}\|_{L^{2}}^{2} \right)+2\nu\left(\|\nabla u_{2,0}\|_{L^{2}}^{2}+\|\nabla u_{3,0}\|_{L^{2}}^{2} \right)=2<\Theta_{0},u_{2,0}>.
	\end{aligned}
\end{equation}
By $\int_{\mathbb{T}}u_{2,0}dz=0 $, we get
\begin{equation*}
	\begin{aligned}
		<\Theta_{0},u_{2,0}>\leq \|\partial_{z}u_{2,0}\|_{L^{2}}\|\partial_{z}\Theta_{0}\|_{L^{2}}\leq \|\nabla u_{2,0}\|_{L^{2}}\|\nabla\Theta_{0}\|_{L^{2}}.
	\end{aligned}
\end{equation*}
Collecting the estimates of $ J_{1} $ and $ J_{2}, $  then (\ref{U20 U30 energy}) yields that
\begin{equation*}
	\begin{aligned}
		\partial_{t}\left(\|u_{2,0}\|_{L^{2}}^{2}+\|u_{3,0}\|_{L^{2}}^{2} \right)+\nu\left(\|\nabla u_{2,0}\|_{L^{2}}^{2}+\|\nabla u_{3,0}\|_{L^{2}}^{2} \right)
		\leq C\nu^{-1}\left(\||u_{\neq}|^{2}\|_{L^{2}}^{2}+\|\nabla\Theta_{0}\|_{L^{2}}^{2} \right),
	\end{aligned}
\end{equation*}
which along with {Lemma \ref{lem 2}} imply that
\begin{equation}\label{U20 U30 end}
	\begin{aligned}
		\|u_{2,0}\|_{Y_{0}}^{2}+\|u_{3,0}\|_{Y_{0}}^{2}
		&\leq C\left(\|(u_{\rm in})_0\|_{L^{2}}^{2}+\nu^{-2}E_{4}^{4}+
		\nu^{-2}E_{3}^{2}\right).
	\end{aligned}
\end{equation}
For $ \omega_{1,0}=\partial_{y}u_{3,0}-\partial_{z}u_{2,0}, $ from $ (\ref{eq:u20 u30 theta}), $ it is easy to know that $ \omega_{1,0} $ satisfies
\begin{equation}\label{omega 10}
	\partial_{t}\omega_{1,0}-\nu\triangle\omega_{1,0}+\partial_{y}h_{3,0}-\partial_{z}h_{2,0}=-\partial_{z}\Theta_{0}.
\end{equation}
Recall that $ h_{j}=\sum_{k=1}^{7}h_{j,k}. $ Obviously, $ (h_{j,2})_{0}=(u_{\neq}\cdot\nabla u_{j,0})_{0}=0. $ As $ \partial_{y}u_{2,0}+\partial_{z}u_{3,0}=0 $ and $ (h_{j,1})_{0}=(u_{2,0}\partial_{y}+u_{3,0}\partial_{z})u_{j,0}, $ we get
\begin{equation*}
	\partial_{y}(h_{3,1})_{0}-\partial_{z}(h_{2,1})_{0}=\left(u_{2,0}\partial_{y}+u_{3,0}\partial_{z} \right)\omega_{1,0}:=g_{1}.
\end{equation*}
Thanks to $ h_{j,k+2}=\partial_{j}P^{k} $ for $ k=2,3,4, $  and $ h_{j,7}=\partial_{j}P^{0}, $ we have
\begin{equation*}
	\begin{aligned}
		&\partial_{y}(h_{3,k+2})_{0}-\partial_{z}(h_{2,k+2})_{0}=\partial_{y}\partial_{z}P_{0}^{k}-\partial_{z}\partial_{y}P_{0}^{k}=0,
		&\partial_{y}(h_{3,7})_{0}-\partial_{z}(h_{2,7})_{0}=\partial_{y}\partial_{z}P_{0}^{0}-\partial_{z}\partial_{y}P_{0}^{0}=0.
	\end{aligned}
\end{equation*}
Then (\ref{omega 10}) can be rewritten into
\begin{equation*}
	\partial_{t}\omega_{1,0}-\nu\triangle\omega_{1,0}+g_{1}+\partial_{y}(h_{3,3})_{0}-\partial_{z}(h_{2,3})_{0}=-\partial_{z}\Theta_{0}.
\end{equation*}
The energy estimate shows that 
\begin{equation}\label{omega10 energy}
	\begin{aligned}
		&\quad\partial_{t}\|\omega_{1,0}\|_{L^{2}}^{2}+2\nu\|\nabla\omega_{1,0}\|_{L^{2}}^{2}=-2\int_{\mathbb{R}\times\mathbb{T}}g_{1}\omega_{1,0}dydz
		\\&-2\int_{\mathbb{R}\times\mathbb{T}}\left[\partial_{y}(h_{3,3})_{0}-\partial_{z}(h_{2,3})_{0} \right]\omega_{1,0}dydz-2\int_{\mathbb{R}\times\mathbb{T}}\partial_{z}\Theta_{0}\omega_{1,0}dydz.
	\end{aligned}
\end{equation}
Noting that $g_{1}=(u_{2,0}\partial_{y}+u_{3,0}\partial_{z})\omega_{1,0}$ and $ \partial_{y}u_{2,0}+\partial_{z}u_{3,0}=0,$ by integration by parts, we get $-2\int_{\mathbb{R}\times\mathbb{T}}g_{1}\omega_{1,0}dydz=0.$ Besides, as $ \int_{\mathbb{T}}\partial_{z}\Theta_{0}dz=0, $ then (\ref{omega10 energy}) follows that
\begin{equation*}
	\begin{aligned}
		\partial_{t}\|\omega_{1,0}\|_{L^{2}}^{2}+2\nu\|\nabla\omega_{1,0}\|_{L^{2}}^{2}\leq& 2\left(\|(h_{3,3})_{0}\|_{L^{2}}+\|(h_{2,3})_{0}\|_{L^{2}} \right)\|\nabla\omega_{1,0}\|_{L^{2}}\\&+C\|\partial_{z}\Theta_{0}\|_{L^{2}}\|\partial_{z}\omega_{1,0}\|_{L^{2}}.
	\end{aligned}
\end{equation*}
Due to $ (h_{j,3})_{0}=(u_{\neq}\cdot\nabla u_{j,\neq})_{0} $,  by {Lemma \ref{lem 2}}, there holds
\begin{equation}\label{omega10}
	\begin{aligned}
		\|\omega_{1,0}\|_{Y_{0}}^{2}\leq&C\left(\|(u_{\rm in})_0\|_{H^{1}}^{2}+\nu^{-1}\|u_{\neq}\cdot\nabla u_{\neq}\|_{L^{2}L^{2}}^{2}+
		\nu^{-1}\|\nabla\Theta_{0}\|_{L^{2}L^{2}}^{2} \right)
		\\\leq&C\left(\|(u_{\rm in})_0\|_{H^{1}}^{2}+\nu^{-2}E_{4}^{4}+\nu^{-2}E_{3}^{2}\right).
	\end{aligned}
\end{equation}
Since $ \partial_{y}u_{2,0}+\partial_{z}u_{3,0}=0 $, we have
\begin{equation*}
	\begin{aligned}
		\|\omega_{1,0}\|_{L^{2}}^{2}=\|\partial_{y}u_{2,0}+\partial_{z}u_{3,0}\|_{L^{2}}^{2}+\|\partial_{y}u_{3,0}-\partial_{z}u_{2,0}\|_{L^{2}}^{2}=\|\nabla u_{2,0}\|_{L^{2}}^{2}+\|\nabla u_{3,0}\|_{L^{2}}^{2},
	\end{aligned}
\end{equation*}
along this with (\ref{omega10}), one deduces
\begin{equation}\label{u0'' end}
	\begin{aligned}
		\|\nabla u_{3,0}\|_{Y_{0}}^{2}\leq \|\omega_{1,0}\|_{Y_{0}}^{2}\leq C\left(\|(u_{\rm in})_0\|_{H^{1}}^{2}+\nu^{-2}E_{4}^{4}+\nu^{-2}E_{3}^{2} \right).
	\end{aligned}
\end{equation}
Due to $ \triangle u_{2,0}=-\partial_{z}\omega_{1,0}, $ then $ \triangle u_{2,0} $ satisfies
\begin{equation*}
	\partial_{t}\triangle u_{2,0}-\nu\triangle^{2}u_{2,0}-\partial_{z}g_{1}-\partial_{y}\partial_{z}(h_{3,3})_{0}+\partial_{z}^{2}(h_{2,3})_{0}=\partial_{z}^{2}\Theta_{0}.
\end{equation*}
The energy estimate indicates
\begin{equation}\label{u20''}
	\begin{aligned}
		\|\triangle u_{2,0}\|_{Y_{0}}^{2}\leq& C\big(\|(u_{\rm in})_0\|_{H^{2}}^{2}+\nu^{-1}\|g_{1}\|_{L^{2}L^{2}}^{2}+\nu^{-1}\|\partial_{z}\Theta_{0}\|_{L^{2}L^{2}}^{2}
		\\&+\nu^{-1}\|\partial_{z}h_{3,3}\|_{L^{2}L^{2}}^{2}+\nu^{-1}\|\partial_{z}h_{2,3}\|_{L^{2}L^{2}}^{2} \big).
	\end{aligned}
\end{equation}
Recall $ g_{1}=(u_{2,0}\partial_{y}+u_{3,0}\partial_{z})\omega_{1,0}, $ using Lemma \ref{zero mode}, we have
\begin{equation}\label{g1 22}
	\begin{aligned}
		\|g_{1}\|_{L^{2}L^{2}}^{2}\leq& \left(\|u_{2,0}\|_{L^{\infty}L^{\infty}}^{2}+\|u_{3,0}\|_{L^{\infty}L^{\infty}}^{2} \right)\|\nabla\omega_{1,0}\|_{L^{2}L^{2}}^{2}\leq C\nu^{-1}E_{2}^{2}\|\omega_{1,0}\|_{Y_{0}}^{2}.
	\end{aligned}
\end{equation}
By Lemma \ref{lem 2}, there holds
\begin{equation*}
	\|\partial_{z}h_{3,3}\|_{L^{2}L^{2}}^{2}+\|\partial_{z}h_{2,3}\|_{L^{2}L^{2}}^{2}\leq C\nu^{-1}E_{4}^{4}.
\end{equation*}
Based on the above estimates, we infer from (\ref{u20''}) that
\begin{equation}\label{u20'' end}
	\|\triangle u_{2,0}\|_{Y_{0}}^{2}\leq C\left(\|(u_{\rm in})_{0}\|_{H^{2}}^{2}+\nu^{-2}E_{2}^{2}\|\omega_{1,0}\|_{Y_{0}}^{2}+\nu^{-2}E_{4}^{4}+\nu^{-2}E_{3}^{2}  \right).
\end{equation}

As $ \partial_{y}\omega_{1,0} $ satisfies
\begin{equation*}
	\partial_{t}\partial_{y}\omega_{1,0}-\nu\triangle\partial_{y}\omega_{1,0}+\partial_{y}g_{1}+\partial_{y}^{2}(h_{3,3})_{0}-\partial_{z}\partial_{y}(h_{2,3})_{0}=-\partial_{y}\partial_{z}\Theta_{0},
\end{equation*}
 energy estimate gives
\begin{equation*}
	\begin{aligned}
		&\partial_{t}\|\partial_{y}\omega_{1,0}\|_{L^{2}}^{2}+\nu\|\nabla\partial_{y}\omega_{1,0}\|_{L^{2}}^{2}\\\leq&C\nu^{-1}\left(\|g_{1}\|_{L^{2}}^{2}+\|\partial_{y}h_{3,3}\|_{L^{2}}^{2}+\|\partial_{z}h_{2,3}\|_{L^{2}}^{2}+\|\partial_{y}\Theta_{0}\|_{L^{2}}^{2} \right),
	\end{aligned}
\end{equation*}
which follows that
\begin{equation*}
	\begin{aligned}
		&\partial_{t}\big(\min(\nu^{\frac23}+\nu t, 1)\|\partial_{y}\omega_{1,0}\|_{L^{2}}^{2} \big)+\nu\min(\nu^{\frac23}+\nu t, 1 )\|\nabla\partial_{y}\omega_{1,0}\|_{L^{2}}^{2}\\\leq&C\nu^{-1}\min(\nu^{\frac23}+\nu t, 1 )\left(\|g_{1}\|_{L^{2}}^{2}+\|\partial_{y}h_{3,3}\|_{L^{2}}^{2}+\|\partial_{z}h_{2,3}\|_{L^{2}}^{2}+\|\partial_{y}\Theta_{0}\|_{L^{2}}^{2} \right)+\nu\|\partial_{y}\omega_{1,0}\|_{L^{2}}^{2}\\\leq&C\nu^{-1}\left(\|g_{1}\|_{L^{2}}^{2}+\|\nabla h_{2,3}\|_{L^{2}}^{2}+\|\partial_{y}\Theta_{0}\|_{L^{2}}^{2} \right)+C\nu^{-\frac13}(1+\nu^{\frac13}t )\|\nabla h_{3,3}\|_{L^{2}}^{2}+\nu\|\nabla\omega_{1,0}\|_{L^{2}}^{2}.
	\end{aligned}
\end{equation*}
Recall that $ \triangle u_{3,0}=\partial_{y}\omega_{1,0} $, by using (\ref{g1 22}) and Lemma \ref{lem 2}, and we get
\begin{equation}\label{u30'' end}
	\begin{aligned}
		\|\min(\nu^{\frac23}+\nu t, 1 )^{\frac12}\triangle u_{3,0}\|_{Y_{0}}^{2}
		\leq C\big(\|u_{\rm in}\|_{H^{2}}^{2}
		+\nu^{-2}E_{4}^{4}+\nu^{-2}E_{3}^{2}+\|\omega_{1,0}\|_{Y_{0}}^{2} \big).
	\end{aligned}
\end{equation}
It follows from (\ref{U20 U30 end}), (\ref{u0'' end}), (\ref{u20'' end}) and (\ref{u30'' end}) that
\begin{equation*}
	\begin{aligned}\label{E12 end}
		E_{2}=&\|\triangle u_{2,0}\|_{Y_{0}}+\|u_{3,0}\|_{Y_{0}}+\|\nabla u_{3,0}\|_{Y_{0}}+\|\min(\nu^{\frac23}+\nu t, 1)^{\frac12}\triangle u_{3,0}\|_{Y_{0}}\\\leq&C\left(\|u_{\rm in}\|_{H^{2}}
		+\nu^{-1}E_{4}^{2}+\nu^{-1}E_{3} \right).
	\end{aligned}
\end{equation*}
\end{proof}

\subsection{Estimate $E_{3}$}
\begin{proposition}\label{prop:E22}
It holds that
\begin{equation*}
	E_{3}\leq C\left(\|\Theta_{\rm in}\|_{H^{2}}+\nu^{-1}E_{2}E_{3}+\nu^{-1}E_{4}E_{5} \right).
\end{equation*}
\end{proposition}
\begin{proof}
For
\begin{equation}\label{eq:Theta_0}
	\partial_{t}\Theta_{0}-\nu\triangle\Theta_{0}+(u\cdot\nabla\Theta)_{0}=0,
\end{equation} 
due to $<u_{0}\cdot\nabla\Theta_{0},\Theta_{0}>=0,$ we have
\begin{equation*}
	\begin{aligned}
		&\partial_{t}\|\Theta_{0}\|_{L^{2}}^{2}+2\nu\|\nabla\Theta_{0}\|_{L^{2}}^{2}
		=-2\int_{\mathbb{R}\times\mathbb{T}}\nabla\cdot(u_{\neq}\Theta_{\neq})_{0}\Theta_{0}dydz
		\leq 2\|u_{\neq}\Theta_{\neq}\|_{L^{2}}\|\nabla\Theta_{0}\|_{L^{2}}.
	\end{aligned}
\end{equation*}
Then there holds
\begin{equation*}\label{theta0 energy}
	\|\Theta_{0}\|_{Y_{0}}^{2}\leq C\left(\|\Theta_{\rm in}\|_{L^{2}}^{2}+\nu^{-1}\|u_{\neq}\Theta_{\neq}\|_{L^{2}L^{2}}^{2} \right),
\end{equation*}
which along with Lemma \ref{lem u theta} and 
$	\|u_{\neq}\Theta_{\neq}\|_{L^{2}L^{2}}^{2}\leq C\nu^{-1}E_{4}^{2}E_{5}^{2}$
yield that
\begin{equation}\label{Theta 0}
	\|\Theta_{0}\|_{Y_{0}}^{2}\leq C\left(\|\Theta_{\rm in}\|_{L^{2}}^{2}+\nu^{-2}E_{4}^{2}E_{5}^{2} \right).
\end{equation}

Multiplying $\triangle\Theta_0$ on both sides of \eqref{eq:Theta_0}, the energy estimate shows that 
\begin{equation*}
	\begin{aligned}
		\|\nabla\Theta_{0}\|_{Y_{0}}^{2}\leq C\left(\|\Theta_{\rm in}\|_{H^{1}}^{2}+\nu^{-1}\|u_0\cdot\nabla\Theta_{0}\|_{L^{2}L^{2}}^{2}
		+\nu^{-1}\|u_{\neq}\cdot\nabla\Theta_{\neq}\|_{L^{2}L^{2}}^{2} \right).
	\end{aligned}
\end{equation*}
By Lemma  \ref{zero mode} and Lemma \ref{lem u theta}, we obtain that 
\begin{equation*}
	\begin{aligned}
		&\|u_0\cdot\nabla\Theta_{0}\|_{L^{2}L^{2}}^{2}\leq C\nu^{-1}E_2^2E_3^2,\\
		&\|u_{\neq}\cdot\nabla\Theta_{\neq}\|_{L^{2}L^{2}}^{2}\leq C\nu^{-1}E_{4}^{2}E_{5}^{2},
	\end{aligned}
\end{equation*}
which implies that 
\begin{equation}\label{energy theta0}
	\|\nabla\Theta_{0}\|_{Y_{0}}^{2}\leq C\left(\|\Theta_{\rm in}\|_{H^{1}}^{2}+\nu^{-2}E_{2}^{2}E_{3}^{2}+\nu^{-2}E_{4}^{2}E_{5}^{2} \right).
\end{equation}

Taking $ \partial_{z}\nabla $ for (\ref{eq:Theta_0}), energy estimate shows that
\begin{equation*}
	\begin{aligned}
		\|\partial_{z}\nabla\Theta_{0}\|_{Y_{0}}^{2}\leq C\left(\|\Theta_{\rm in}\|_{H^{2}}^{2}+\nu^{-1}\|\partial_{z}(u\cdot\nabla\Theta)_{0}\|_{L^{2}L^{2}}^{2} \right).
	\end{aligned}
\end{equation*}
Using  Lemma \ref{lem u theta} and 
\begin{equation*}
	\begin{aligned}
		&\|\partial_{z}(u_{0}\cdot\nabla\Theta_{0})\|_{L^{2}L^{2}}^{2}
		\leq\|\partial_{z}(u_{2,0}\partial_{y}\Theta_{0})\|_{L^{2}L^{2}}^{2}+\|\partial_{z}(u_{3,0}\partial_{z}\Theta_{0})\|_{L^{2}L^{2}}^{2}
		\\ \leq&C\nu^{-1}\left(\|\Theta_{0}\|_{Y_{0}}^{2}+\|\partial_{z}\nabla\Theta_{0}\|_{Y_{0}}^{2} \right)
		\left(\|\triangle u_{2,0}\|_{Y_{0}}^{2}+\|\nabla u_{3,0}\|_{Y_{0}}^{2}+\|u_{3,0}\|_{Y_{0}}^{2} \right)
		\\ \leq& C\nu^{-1}E_{2}^{2}E_{3}^{2},
	\end{aligned}
\end{equation*}
one obtains
\begin{equation*}
	\begin{aligned}
		\|\partial_{z}(u\cdot\nabla\Theta)_{0}\|_{L^{2}L^{2}}^{2}\leq C\nu^{-1}E_{2}^{2}E_{3}^{2}+C\nu^{-1}E_{4}^{2}E_{5}^{2}.
	\end{aligned}
\end{equation*}
Therefore, we have
\begin{equation}\label{energy theta0'}
	\|\partial_{z}\nabla\Theta_{0}\|_{Y_{0}}^{2}\leq C\left(\|\Theta_{\rm in}\|_{H^{2}}^{2}+\nu^{-2}E_{2}^{2}E_{3}^{2}+\nu^{-2}E_{4}^{2}E_{5}^{2} \right).
\end{equation}
Combining (\ref{Theta 0}), (\ref{energy theta0}) and (\ref{energy theta0'}), there holds
\begin{equation*}
	\begin{aligned}
		E_{3}\leq C\left(\|\Theta_{\rm in}\|_{H^{2}}+\nu^{-1}E_{2}E_{3}+\nu^{-1}E_{4}E_{5} \right).
	\end{aligned}
\end{equation*}
\end{proof}

\section{Energy estimates for the non-zero modes}
\subsection{Estimate $ E_{4} $}
\begin{proposition}\label{prop:E3}
There holds
\begin{equation*}
	E_{4}^{2}\leq C\left(\|u_{\rm in}\|_{H^{2}}^{2}+E_{6}^{2}+\nu^{-2}E_{5}^{2}+E_{7}\right).
\end{equation*}
\end{proposition}
\begin{proof}
{\underline{\textbf{Step I. Estimate $ E_{4,1}. $}}}
Due to $ \triangle P^{N_{1}}=-2\partial_{x}u_{2} $ and $ \triangle P^{N_{3}}=\partial_{y}\Theta, $ we have
\begin{equation*}
	\begin{aligned}
		&\partial_{z}P^{N_{1}}=-2\partial_{x}\partial_{z}\triangle^{-1}u_{2,\neq}=-2\partial_{x}\partial_{z}\triangle^{-2}(\triangle u_{2,\neq}),\\
		&\Theta_{\neq}-\partial_{y}P^{N_{3}}_{\neq}
		=\triangle^{-1}(\partial_{x}^{2}\Theta_{\neq}+\partial_{z}^{2}\Theta_{\neq}).			
	\end{aligned}
\end{equation*}
Therefore, it follows from (\ref{23}) that 
\begin{equation*}
	\left\{
	\begin{aligned}
		&\mathcal{L}(\triangle u_{2,\neq})=-\triangle\left(\partial_{y}P^{1}+u_{1,0}\partial_{x}u_{2,\neq}+h_{2,\neq} \right)+\partial_{x}^{2}\Theta_{\neq}+\partial_{z}^{2}\Theta_{\neq},\\
		&\mathcal{L}u_{3,\neq}-2\partial_{x}\partial_{z}\triangle^{-2}(\triangle u_{2,\neq})+\partial_{z}P^{1}+u_{1,0}\partial_{x}u_{3,\neq}+h_{3,\neq}+\triangle^{-1}\partial_{y}\partial_{z}\Theta_{\neq}=0.
	\end{aligned}
	\right.
\end{equation*}
Applying Proposition \ref{prop Lf 1} to the above equations, we arrive at 
%	\begin{equation}\label{U2'' U3''}
	%		\begin{aligned}
		%			&\|\triangle u_{2,\neq}\|_{X_{a}}^{2}+\|(\partial_{x}^{2}+\partial_{z}^{2})u_{3,\neq}\|_{X_{a}}^{2}\\\leq&C\big(\|u_{\rm in}\|_{H^{2}}^{2}
		%			+\nu^{-1}\|e^{a\nu^{\frac13}t}\nabla\left(\partial_{y}P^{1}+u_{1,0}\partial_{x}u_{2,\neq}+h_{2,\neq} \right)\|_{L^{2}L^{2}}^{2}\\&+\nu^{-\frac13}\|e^{a\nu^{\frac13}t}
		%			(\partial_{x}^{2},\partial_{z}^{2})\Theta_{\neq}\|_{L^{2}L^{2}}^{2}
		%			+\nu^{-1}\|e^{a\nu^{\frac13}t}(\partial_{x},\partial_{z})\triangle^{-1}\partial_{y}\partial_{z}\Theta_{\neq}\|_{L^{2}L^{2}}^{2}
		%			\\&+\nu^{-1}\|e^{a\nu^{\frac13}t}(\partial_{x},\partial_{z})\left(\partial_{z}P^{1}+u_{1,0}\partial_{x}u_{3,\neq}+h_{3,\neq} \right)\|_{L^{2}L^{2}}^{2}\big)\\\leq&C\big(\|u_{\rm in}\|_{H^{2}}^{2}+\nu^{-1}\|e^{a\nu^{\frac13}t}\triangle P^{1}\|_{L^{2}L^{2}}^{2}+\nu^{-1}\|e^{a\nu^{\frac13}t}\nabla\left(u_{1,0}\partial_{x}u_{2,\neq}+h_{2,\neq} \right)\|_{L^{2}L^{2}}^{2}\\&+\nu^{-1}\|e^{a\nu^{\frac13}t}(\partial_{x},\partial_{z})(u_{1,0}\partial_{x}u_{3,\neq}+h_{3,\neq})\|_{L^{2}L^{2}}^{2}
		%			+\nu^{-\frac43}\|(\partial_{x}^{2},\partial_{z}^{2})\Theta_{\neq}\|_{X_{b}}^{2}\big).
		%		\end{aligned}
	%	\end{equation}
\begin{equation}\label{U2'' U3''}
	\begin{aligned}
		&\|\triangle u_{2,\neq}\|_{X_{a}}^{2}+\|(\partial_{x}^{2}+\partial_{z}^{2})u_{3,\neq}\|_{X_{a}}^{2}\\\leq&C\big(\|u_{\rm in}\|_{H^{2}}^{2}+\nu^{-1}\|e^{a\nu^{\frac13}t}\triangle P^{1}\|_{L^{2}L^{2}}^{2}+\nu^{-1}\|e^{a\nu^{\frac13}t}\nabla\left(u_{1,0}\partial_{x}u_{2,\neq}+h_{2,\neq} \right)\|_{L^{2}L^{2}}^{2}\\&+\nu^{-1}\|e^{a\nu^{\frac13}t}(\partial_{x},\partial_{z})(u_{1,0}\partial_{x}u_{3,\neq}+h_{3,\neq})\|_{L^{2}L^{2}}^{2}
		+\nu^{-\frac43}\|(\partial_{x}^{2},\partial_{z}^{2})\Theta_{\neq}\|_{X_{b}}^{2}\big).
	\end{aligned}
\end{equation}
Noting that $ h_{j}=\sum_{k=1}^{7}h_{j,k}$, by Lemma \ref{lem 2}, Lemma \ref{lem 1} and Lemma \ref{lem: zero and non-zero mode}, we get
\begin{equation*}
	\begin{aligned}
		&\|e^{a\nu^{\frac13}t}\nabla h_{2,\neq}\|_{L^{2}L^{2}}^{2}+\|e^{a\nu^{\frac13}t}(\partial_{x},\partial_{z})h_{3,\neq}\|_{L^{2}L^{2}}^{2}\\\leq&C\left(\nu^{-1}E_{2}^{2}E_{4}^{2}+\nu^{-1}E_{2}^{2}E_{7}+\nu^{-1}E_{4}^{4}+\nu E_{1}^{2}E_{7} \right).
	\end{aligned}
\end{equation*}
Moreover, using Lemma \ref{lem 1} and (\ref{U2'' U3''}), there is
\begin{equation}\label{e31 end}
	\begin{aligned}
		E_{4,1}^{2}\leq& C\left(\|\triangle u_{2,\neq}\|_{X_{a}}^{2}+\|(\partial_{x}^{2}+\partial_{z}^{2})u_{3,\neq}\|_{X_{a}}^{2} \right)
		\\\leq&C\big(\|u_{\rm in}\|_{H^{2}}^{2}+\nu^{-2}E_{2}^{2}E_{4}^{2}+E_{1}^{2}E_{4}E_{6}+\nu^{-2}E_{2}^{2}E_{7}
		+\nu^{-2}E_{4}^{4}+E_{1}^{2}E_{7}+\nu^{-\frac43}E_{5}^{2} \big).
	\end{aligned}
\end{equation}

{\underline{\textbf{Step II. Estimate $ E_{4,2}. $}}} 
Due to $ \partial_{x}\omega_{2,\neq}=-\partial_{z}\partial_{y}u_{2,\neq}-(\partial_{x}^{2}+\partial_{z}^{2})u_{3,\neq} $ and (\ref{e31 end}), there holds
\begin{equation}\label{end x omega2}
	\begin{aligned}
		&\|\partial_{x}\omega_{2,\neq}\|_{X_{a}}^{2}\leq C\left(\|\partial_{z}\partial_{y}u_{2,\neq}\|_{X_{a}}^{2}+\|(\partial_{x}^{2}+\partial_{z}^{2})u_{3,\neq}\|_{X_{a}}^{2} \right)\leq CE_{4,1}^{2}\\\leq& C\big(\|u_{\rm in}\|_{H^{2}}^{2}
		+\nu^{-2}E_{2}^{2}E_{4}^{2}+E_{1}^{2}E_{4}E_{6}+\nu^{-2}E_{2}^{2}E_{7}+\nu^{-2}E_{4}^{4}+E_{1}^{2}E_{7}+\nu^{-\frac43}E_{5}^{2} \big).
	\end{aligned}
\end{equation}
From $ (\ref{ini1})_{1}, $  we know $ \omega_{2,\neq}=\partial_{z}u_{1,\neq}-\partial_{x}u_{3,\neq} $  satisfies
\begin{equation}\label{eq:omega2}
	\mathcal{L}\omega_{2,\neq}=-\partial_{z}u_{2,\neq}+\partial_{x}(u\cdot\nabla u_{3})_{\neq}-\partial_{z}(u\cdot\nabla u_{1})_{\neq}.
\end{equation}
Taking $ \partial_{y} $ to (\ref{eq:omega2}) and noting that $ -\partial_{y}u_{2,\neq}=\partial_{x}u_{1,\neq}+\partial_{z}u_{3,\neq} $, we arrive
\begin{equation*}
	\begin{aligned}
		\mathcal{L}\partial_{y}\omega_{2,\neq}=(\partial_{x}^{2}+\partial_{z}^{2})u_{3,\neq}+\partial_{y}\partial_{x}(u\cdot\nabla u_{3})_{\neq}-\partial_{y}\partial_{z}(u\cdot\nabla u_{1})_{\neq}.
	\end{aligned}
\end{equation*}
Applying Proposition \ref{prop Lf}, there holds
\begin{equation}\label{energy y omega2}
	\begin{aligned}
		\|\partial_{y}\omega_{2,\neq}\|_{X_{a}}^{2}\leq&C\big(\|u_{\rm in}\|_{H^{2}}^{2}+\nu^{-\frac13}\|e^{a\nu^{\frac13}t}
		(\partial_{x}^{2}+\partial_{z}^{2})u_{3,\neq}\|_{L^{2}L^{2}}^{2}
		\\&+\nu^{-1}\|e^{a\nu^{\frac13}t}\partial_{x}(u\cdot\nabla u_{3})_{\neq}\|_{L^{2}L^{2}}^{2}
		+\nu^{-1}\|e^{a\nu^{\frac13}t}\partial_{z}(u\cdot\nabla u_{1})_{\neq}\|_{L^{2}L^{2}}^{2}
		\big).
	\end{aligned}
\end{equation}
Using (ii) of Lemma \ref{lem 2} and Lemma \ref{lem: u0 u_neq}, we get
\begin{equation}\label{t1}
	\begin{aligned}
		&\|e^{a\nu^{\frac13}t}\partial_{x}(u\cdot\nabla u_{3})_{\neq}\|_{L^{2}L^{2}}^{2}\\\leq&C\big(\|e^{a\nu^{\frac13}t}\partial_{x}(u_{0}\cdot\nabla u_{3,\neq})\|_{L^{2}L^{2}}^{2}+\|e^{a\nu^{\frac13}t}\partial_{x}(u_{\neq}\cdot\nabla u_{3,0})\|_{L^{2}L^{2}}^{2}+\|e^{a\nu^{\frac13}t}\partial_{x}(u_{\neq}\cdot\nabla u_{3,\neq})_{\neq}\|_{L^{2}L^{2}}^{2} \big)\\\leq&C\nu^{\frac13}E_{1}^{2}E_{4}E_{6}+C\nu^{-\frac53}E_{2}^{2}E_{4}^{2}+C\nu^{-\frac53}E_{4}^{4}
	\end{aligned}
\end{equation}
and
\begin{equation}\label{t2}
	\begin{aligned}
		&	\|e^{a\nu^{\frac13}t}\nabla(u\cdot\nabla u_{1})_{\neq}\|_{L^{2}L^{2}}^{2}\\
		\leq&C\big(\|e^{a\nu^{\frac13}t}\nabla(u_{0}\cdot\nabla u_{1,\neq})\|_{L^{2}L^{2}}^{2}+\|e^{a\nu^{\frac13}t}\nabla(u_{\neq}\cdot\nabla u_{1,0})\|_{L^{2}L^{2}}^{2}+\|e^{a\nu^{\frac13}t}\nabla(u_{\neq}\cdot\nabla u_{1,\neq})\|_{L^{2}L^{2}}^{2} \big)\\\leq&C\nu^{\frac13}E_{1}^{2}E_{4}E_{6}+C\nu^{-\frac53}E_{2}^{2}E_{4}^{2}+C\nu^{\frac13}E_{1}^{2}E_{7}+C\nu^{-\frac53}E_{4}^{4}.
	\end{aligned}
\end{equation}
Using (\ref{t1}) and (\ref{t2}), then (\ref{energy y omega2}) yields that
\begin{equation}\label{end y omega2}
	\begin{aligned}
		\|\partial_{y}\omega_{2,\neq}\|_{X_{a}}^{2}\leq C\big(&\|u_{\rm in}\|_{H^{2}}^{2}+\nu^{-\frac23}E_{1}^{2}E_{4}E_{6}
		+\nu^{-\frac83}E_{2}^{2}E_{4}^{2}+\nu^{-\frac83}E_{4}^{4}+\nu^{-\frac23}E_{1}^{2}E_{7}\\&+\nu^{-\frac83}E_{2}^{2}E_{7}+\nu^{-2}E_{5}^{2}\big).
	\end{aligned}
\end{equation}

Taking $ \partial_{z} $ for (\ref{eq:omega2}), we get
\begin{equation*}
	\begin{aligned}
		\mathcal{L}\partial_{z}\omega_{2,\neq}=-\partial_{z}^{2}u_{2,\neq}+\partial_{z}\partial_{x}(u\cdot\nabla u_{3})_{\neq}-\partial_{z}^{2}(u\cdot\nabla u_{1})_{\neq}.
	\end{aligned}
\end{equation*}
Applying Proposition \ref{prop Lf} to it and using (\ref{t1})-(\ref{t2}), one deduces
\begin{equation}\label{end z omega2}
	\begin{aligned}
		\|\partial_{z}\omega_{2,\neq}\|_{X_{a}}^{2}
		\leq&\big(\|u_{\rm in}\|_{H^{2}}^{2}+\nu^{-\frac23}\|\triangle u_{2,\neq}\|_{X_{a}}^{2}
		+\nu^{-1}\|e^{a\nu^{\frac13}t}\partial_{x}(u\cdot\nabla u_{3})_{\neq}\|_{L^{2}L^{2}}^{2}
		\\&+\nu^{-1}\|e^{a\nu^{\frac13}t}\partial_{z}(u\cdot\nabla u_{1})_{\neq}\|_{L^{2}L^{2}}^{2}\big)
		\\\leq&C\big(\|u_{\rm in}\|_{H^{2}}^{2}+\nu^{-\frac23}E_{1}^{2}E_{4}E_{6}+\nu^{-\frac83}E_{2}^{2}E_{4}^{2}
		+\nu^{-\frac83}E_{4}^{4}+\nu^{-\frac23}E_{1}^{2}E_{7}
		\big).
	\end{aligned}
\end{equation}
Collecting (\ref{end x omega2}), (\ref{end y omega2}) and (\ref{end z omega2}), we get
\begin{equation}\label{end e32}
	\begin{aligned}
		&E_{4,2}^{2}\leq C\nu^{\frac23}\|\nabla\omega_{2,\neq}\|_{X_{a}}^{2}\\
		\leq&C\big(\|u_{\rm in
		}\|_{H^{2}}^{2}+E_{1}^{2}E_{4}E_{6}+\nu^{-2}E_{2}^{2}E_{4}^{2}
		+\nu^{-2}E_{4}^{4}
		+E_{1}^{2}E_{7}+{\nu^{-2}}E_{2}^{2}E_{7}+{\nu^{-\frac43}}E_{5}^{2} \big).
	\end{aligned}
\end{equation}

Combining (\ref{e31 end}) with (\ref{end e32}), and using Proposition \ref{E5 E6}, one deduces
\begin{equation*}
	\begin{aligned}
		E_{4}^{2}\leq&2\left(E_{4,1}^{2}+E_{4,2}^{2} \right)\\\leq&C\big(\|u_{\rm in}\|_{H^{2}}^{2}
		+\nu^{-2}E_{2}^{2}E_{4}^{2}+E_{1}^{2}E_{4}E_{6}+\nu^{-2}E_{2}^{2}E_{7}+\nu^{-2}E_{4}^{4}+E_{1}^{2}E_{7}+\nu^{-\frac43}E_{5}^{2}\big).
	\end{aligned}
\end{equation*}
Taking $ \varepsilon_{0} $ small enough, then we arrive
\begin{equation*}\label{end e3}
	\begin{aligned}
		E_{4}^{2}\leq& C\left(\|u_{\rm in}\|_{H^{2}}^{2}+E_{6}^{2}+\nu^{-2}E_{5}^{2}+E_{7} \right).
	\end{aligned}
\end{equation*}
\end{proof}
\subsection{Estimate $ E_{5} $}
\begin{proposition}\label{propE4}
There holds 
\begin{equation*}
	E_{5}\leq C\big(\|\Theta_{\rm in}\|_{H^{2}}+\nu^{-\frac23}E_{3}E_{6}+\nu^{-1}E_{3}E_{4} \big).
\end{equation*}
\end{proposition}
\begin{proof}
According to (\ref{ini1}), we know $\partial_x^2\Theta_{\neq} $ satisfies 
\begin{equation}\label{theta xx}
	\begin{aligned}
		\mathcal{L}_{V}\partial_{x}^{2}\Theta_{\neq}+\partial_{x}^{2}(u_{\neq}\cdot\nabla\Theta_{\neq})_{\neq}
		+u_{2,0}\partial_{x}^{2}\partial_{y}\Theta_{\neq}+u_{3,0}\partial_{x}^{2}\partial_{z}\Theta_{\neq}
		+\nabla\cdot(\partial_{x}^{2}u_{\neq}\Theta_{0}).
	\end{aligned}
\end{equation}
By Lemma \ref{lem zero nonzero}, we have
$\|\Theta_{0}\|_{L^{\infty}L^{\infty}}^{2}\leq C\|(\partial_{z}\Theta_{0},\Theta_{0})\|_{L^{\infty}H^{1}}^{2}\leq C\left(\|\Theta_{0}\|_{Y_{0}}^{2}+\|\partial_{z}\nabla\Theta_{0}\|_{Y_{0}}^{2} \right),$
then applying Proposition \ref{Lvf 0} to (\ref{theta xx}), we get
\begin{equation}\label{xx}
	\begin{aligned}
		\|\partial_{x}^{2}\Theta_{\neq}\|_{X_{b}}^{2}
		\leq&C\Big( \|\Theta_{\rm in}\|_{H^{2}}^{2}
		+\nu^{-1}\|e^{b\nu^{\frac13}t}\partial_{x}^{2}(u_{\neq}\Theta_{\neq})\|_{L^{2}L^{2}}^{2}
		+\nu^{-\frac43}\|(u_{2,0},u_{3,0})\|_{L^{\infty}L^{\infty}}^{2}
		\|\partial_{x}^{2}\Theta_{\neq}\|_{X_{b}}^{2}\\&+\nu^{-\frac43}
		\left(\|\Theta_{0}\|_{Y_{0}}^{2}+\|\partial_{z}\nabla\Theta_{0}\|_{Y_{0}}^{2} \right)
		\left(\|\partial_{x}^{2}u_{2,\neq}\|_{X_{b}}^{2}+\|\partial_{x}^{2}u_{3,\neq}\|_{X_{b}}^{2} \right)
		\Big).
	\end{aligned}
\end{equation}
Using Lemma \ref{zero mode} and  Lemma \ref{lem u theta}, we get by (\ref{xx}) that
\begin{equation}\label{theta xx xb}
	\begin{aligned}
		\|\partial_{x}^{2}\Theta_{\neq}\|_{X_{b}}^{2}\leq& 
		C\left(\|\Theta_{\rm in}\|_{H^{2}}^{2}+\nu^{-2}E_{4}^{2}E_{5}^{2}
		+\nu^{-\frac43}E_{2}^{2}E_{5}^{2}+\nu^{-\frac43}E_{3}^{2}E_{6}^{2}
		\right).
	\end{aligned}
\end{equation}
Similarly, $\partial_{z}^{2}\Theta$ satisfies
\begin{equation}\label{theta zz}
	\begin{aligned}
		&\mathcal{L}_{V}\partial_{z}^{2}\Theta_{\neq}+\partial_{z}\widehat{u_{1,0}}\partial_{x}\partial_{z}\Theta_{\neq}+\partial_{z}(\partial_{z}\widehat{u_{1,0}}\partial_{x}\Theta_{\neq})+\partial_{z}^{2}(u_{\neq}\cdot\nabla\Theta_{\neq})_{\neq}
		+\partial_{z}^{2}\left(\widetilde{u_{1,0}}\partial_{x}\Theta_{\neq}\right)
		\\&+\partial_{z}^{2}(u_{2,0}\partial_{y}\Theta_{\neq})+\partial_{z}^{2}(u_{3,0}\partial_{z}\Theta_{\neq})+\partial_{z}^{2}(u_{2,\neq}\partial_{y}\Theta_{0})
		+\partial_{z}^{2}(u_{3,\neq}\partial_{z}\Theta_{0})=0.
	\end{aligned}
\end{equation}
Applying Proposition \ref{Lvf 0} to (\ref{theta zz}), one obtains
\begin{equation}\label{theta_neq zz}
	\begin{aligned}
		\|\partial_{z}^{2}\Theta_{\neq}\|_{X_{a}}^{2}\leq&C\Big(\|\Theta_{\rm in}\|_{H^{2}}^{2}
		+\nu^{-\frac13}\|e^{a\nu^{\frac13}t}\partial_{z}\widehat{u_{1,0}}\partial_{x}\partial_{z}\Theta_{\neq}\|_{L^{2}L^{2}}^{2}
		\\&+\nu^{-1}\|e^{a\nu^{\frac13}t}\partial_{z}\widehat{u_{1,0}}\partial_{x}\Theta_{\neq}\|_{L^{2}L^{2}}^{2}
		+\nu^{-1}\|e^{a\nu^{\frac13}t}\partial_{z}(u_{\neq}\cdot\nabla\Theta_{\neq})_{\neq}\|_{L^{2}L^{2}}^{2}
		\\&+\nu^{-1}\|e^{a\nu^{\frac13}t}\partial_{z}(u_{2,0}\partial_{y}\Theta_{\neq})\|_{L^{2}L^{2}}^{2}
		+\nu^{-1}\|e^{a\nu^{\frac13}t}\partial_{z}(u_{3,0}\partial_{z}\Theta_{\neq})\|_{L^{2}L^{2}}^{2}
		\\&+\nu^{-1}\|e^{a\nu^{\frac13}t}\partial_{z}(u_{2}\cdot\nabla\Theta_{0})\|_{L^{2}L^{2}}^{2}
		+\nu^{-1}\|e^{a\nu^{\frac13}t}\partial_{z}\left(\widetilde{u_{1,0}}\partial_{x}\Theta_{\neq} \right)\|_{L^{2}L^{2}}^{2}
		\Big).
	\end{aligned}
\end{equation}
Using Lemma \ref{lem u theta}, Lemma \ref{lem non zero mode 1} and Lemma \ref{sob_15}, we get by (\ref{theta_neq zz}) that
\begin{equation}\label{energy theta-neq zz}
	\begin{aligned}
		\|\partial_{z}^{2}\Theta_{\neq}\|_{X_{a}}^{2}\leq&C\left(\|\Theta_{\rm in}\|_{H^{2}}^{2}+E_{1}^{2}E_{5}^{2}+\nu^{-2}E_{2}^{2}E_{5}^{2}+\nu^{-2}E_{3}^{2}E_{4}^{2} \right).
	\end{aligned}
\end{equation}

Combining (\ref{theta xx xb}) with (\ref{energy theta-neq zz}), we have
\begin{equation*}
	\begin{aligned}
		E_{5}\leq C\left(\|\Theta_{\rm in}\|_{H^{2}}
		+\nu^{-1}E_{4}E_{5}+\nu^{-1}E_{2}E_{5}+\nu                                                ^{-\frac23}E_{3}E_{6}+E_{1}E_{5}+\nu^{-1}E_{3}E_{4} \right).
	\end{aligned}
\end{equation*}
As $E_{1}\leq\varepsilon_{0}, E_{2}\leq\varepsilon_{0}\nu $ and $ E_{4}\leq \varepsilon_{0}\nu, $ taking $ \varepsilon_{0} $ small enough such that $ C\varepsilon_{0}<\frac12, $ one deduces
\begin{equation}\label{e32 end}
	\begin{aligned}
		E_{5}\leq& C\big(\|\Theta_{\rm in}\|_{H^{2}}+\nu^{-\frac23}E_{3}E_{6}+\nu^{-1}E_{3}E_{4} \big).
	\end{aligned}
\end{equation}
\end{proof}

\subsection{Estimate $ E_{6} $} 
Motivated by \cite{wei2}, in order to deal $ E_{6}, $ a key idea is to introduce a new quantity $Q$  defined by
\begin{equation}\label{define kappa}
	Q=u_{2,\neq}+\kappa u_{3,\neq},\quad \kappa(t,y,z)=\frac{\partial_{z}V}{\partial_{y}V}=\frac{\partial_{z}\widehat{u_{1,0}}}{1+\partial_{y}\widehat{u_{1,0}}}.
\end{equation}
Then  $ Q $ satisfies 
\begin{equation*}
	\begin{aligned}
		\widetilde{\mathcal{L}_{V}}Q+H_{2}=(\partial_{t}\kappa-\nu\triangle\kappa)u_{3,\neq}-2\nu\nabla\kappa\cdot\nabla u_{3,\neq}+\Theta_{\neq}-(\partial_{y}+\kappa\partial_{z})\triangle^{-1}\partial_{y}\Theta_{\neq},
	\end{aligned}
\end{equation*}
where $ H_{2}=(h_{2}+\kappa h_{3})_{\neq} $ and
\begin{equation*}\label{widetilde Lv}
	\widetilde{\mathcal{L}_{V}}f=\mathcal{L}_{V}f-2(\partial_{y}+\kappa\partial_{z})\triangle^{-1}(\partial_{y}V\partial_{x}f).
\end{equation*}
%Notice that 
%\begin{equation*}
%\triangle\widetilde{\mathcal{L}_{V}}Q=\mathcal{L}_{V}\triangle Q+{\rm good~~terms},
%\end{equation*}
In addition, $ \triangle Q $ satisfies
\begin{equation}\label{eq Q''}
	\mathcal{L}_{V}\triangle Q=\triangle(-2\nu\nabla\kappa\cdot\nabla u_{3,\neq})+\triangle\left(\Theta_{\neq}-(\partial_{y}+\kappa\partial_{z})\triangle^{-1}\partial_{y}\Theta_{\neq} \right)+{\rm good~~terms}.
\end{equation}
To deal the singular term $ \triangle(-2\nu\nabla\kappa\cdot\nabla u_{3,\neq})$ and obtain a sharp result,  the following
decomposition introduced by \cite{wei2} is important 
\begin{equation}\label{kappa u3}
	\nabla\kappa\cdot\nabla u_{3,\neq}=\rho_{1}\nabla V\cdot\nabla u_{3,\neq}+\rho_{2}(\partial_{z}-\kappa\partial_{y})u_{3,\neq},
\end{equation}
where
\begin{equation}\label{rho1 rho2}
	\rho_{1}=\frac{\partial_{y}\kappa+\kappa\partial_{z}\kappa}{\partial_{y}V(1+\kappa^{2})},\quad\rho_{2}=\frac{\partial_{z}\kappa-\kappa\partial_{y}\kappa}{1+\kappa^{2}}.
\end{equation}
Since $\|\kappa\|_{H^3}$ can be enough small, 
$ (\partial_{z}-\kappa\partial_{y}) $ can be regard as a good derivative,
and the second term in (\ref{kappa u3}) is good.

In order to obtain a sharp threshold of velocity, 
we use the following quasi-linear
decomposition
\begin{equation*}
	Q=Q_{1}+\nu Q_{2}+Q_{3},
\end{equation*}
where $Q_1, Q_{2}$ and $Q_{3} $ satisfying
\begin{equation}\label{Q}
	\left\{
	\begin{aligned}
		&\mathcal{L}_{V}\triangle Q_{1}={\rm good~~terms},\\
		&\mathcal{L}_{V}Q_{2}=-\rho_{1}\nabla V\cdot\nabla u_{3,\neq},\\ 
		&\mathcal{L}_{V}Q_{3}=\Theta_{\neq}-(\partial_{y}+\kappa\partial_{z})\triangle^{-1}\partial_{y}\Theta_{\neq},
		\\&Q_{1}(0)=Q(0),~~~Q_{2}(0)=Q_{3}(0)=0.
	\end{aligned}
	\right.
\end{equation}
Furthermore, an additional quantity $Q_4$ is also needed satisfying
\begin{equation*}
	\mathcal{L}_{V}Q_{4}=-\nabla V\cdot\nabla u_{3,\neq}, ~~~ Q_{4}(0)=0.
\end{equation*}
Based on this decomposition, we will prove that 
\begin{equation*}
	\begin{aligned}
		E_{6}^2\leq C(E_{7}+\nu^{\frac43}\|\triangle u_{3,\neq}\|_{X_{b}}^{2})\leq C\left(\|u_{\rm in}\|_{H^{2}}^{2}+\nu^{-2}E_{4}^{4}+\nu^{-2}E_{5}^{2} \right),
	\end{aligned}
\end{equation*}
where $ E_{7} $ is the auxiliary norm  defined by
\begin{equation*}
	E_{7}=\sum_{j=2}^{3}\left(\|\partial_{x}^{2}u_{j,\neq}\|_{X_{b}}^{2}+\|\partial_{x}(\partial_{z}-\kappa\partial_{y})u_{j,\neq}\|_{X_{b}}^{2} \right)+\|\partial_{x}\nabla Q\|_{X_{b}}^{2}.
\end{equation*}

The following lemma gives the estimates of $ Q_{2}, Q_{3} $ and $ Q_{4} $ associated with good derivatives.
\begin{lemma}\label{lem W}
It holds that
\begin{itemize}
	\item[(i)] $\|\partial_{x}^{2}Q_{4}\|_{X_{b}}^{2}+\|\partial_{x}(\partial_{z}-\kappa\partial_{y})Q_{4}\|_{X_{b}}^{2}\leq C\nu^{-\frac43}E_{7},$
	\item[(ii)] 
	$\|\partial_{x}^{2}Q_{2}\|_{X_{b}}^{2}\leq C\varepsilon_{0}^{2}\nu^{-\frac43}E_{7}, \quad \|\partial_{x}\nabla Q_{2}\|_{X_{b}}^{2}\leq C\varepsilon_{0}^{2}\nu^{-2}E_{7},$
	\item[(iii)] 
	$\|\partial_{x}(Q_{2}-\rho_{1}Q_{4})\|_{X_{b}}^{2}\leq C\nu^{-\frac23}\varepsilon_{0}^{2}E_{7},$
	\item[(iv)] 
	$\|\partial_{x}^{2}Q_{3}\|_{X_{b}}^{2}\leq C\nu^{-\frac23}E_{5}^{2},\quad \|\partial_{x}\nabla Q_{3}\|_{X_{b}}^{2}\leq C\nu^{-\frac43}E_{5}^{2}.$
\end{itemize}
\end{lemma}
\begin{proof}
{\textbf{\underline{Estimate (i).}}}
Since $ \mathcal{L}_{V}Q_{4}=-\nabla V\cdot\nabla u_{3,\neq}, $  by Proposition \ref{Lvf}, then we obtain
\begin{equation}\label{W23}
	\begin{aligned}
		&\|\partial_{x}^{2}Q_{4}\|_{X_{b}}^{2}+\|\partial_{x}(\partial_{z}-\kappa\partial_{y})Q_{4}\|_{X_{b}}^{2}
		\leq C\nu^{-\frac13}\big(\|e^{b\nu^{\frac13}t}\partial_{x}^{2}(\nabla V\cdot\nabla u_{3,\neq})\|_{L^{2}L^{2}}^{2}\\
		&+\|e^{b\nu^{\frac13}t}\partial_{x}(\partial_{z}-\kappa\partial_{y})
		(\nabla V\cdot\nabla u_{3,\neq})\|_{L^{2}L^{2}}^{2} \big):=C\nu^{-\frac13}\left(I_{1}+I_{2} \right).
	\end{aligned}
\end{equation}
Recall that $ V=y+\widehat{u_{1,0}}(t,y,z),$ by Lemma \ref{kappa},
and we have
\begin{align}
	&\|\nabla V\|_{L^{\infty}}\leq 1+\|\nabla \widehat{u_{1,0}}\|_{L^{\infty}} \leq C\left(1+\|\widehat{u_{1,0}}\|_{H^{4}} \right)\leq C,\label{V1' infty}\\
	&\|(\partial_{z}-\kappa\partial_{y})\nabla V\|_{L^{\infty}}\leq C\left(1+\|\kappa\|_{H^{3}} \right)\|\widehat{u_{1,0}}\|_{H^{4}}\leq C.\label{V1'' infty}
\end{align}
Then for $ I_{1}, $ using (\ref{V1' infty}), we have
\begin{equation*}
	\begin{aligned}
		I_{1}\leq C\|\nabla V\|_{L^{\infty}L^{\infty}}\|e^{b\nu^{\frac13}t}\nabla\partial_{x}^{2}u_{3,\neq}\|_{L^{2}L^{2}}^{2}
		\leq C\nu^{-1}\|\partial_{x}^{2}u_{3,\neq}\|_{X_{b}}^{2}\leq C\nu^{-1}E_{7}.
	\end{aligned}
\end{equation*}
For $ I_{2}, $  using (\ref{V1' infty}) and (\ref{V1'' infty}), we deduce that
\begin{equation*}
	\begin{aligned}
		&\|\partial_{x}(\partial_{z}-\kappa\partial_{y})\left(\nabla V\cdot\nabla u_{3,\neq} \right)\|_{L^{2}}
		\leq C\left(\|\nabla\partial_{x}(\partial_{z}-\kappa\partial_{y})u_{3,\neq}\|_{L^{2}}+\|\nabla\partial_{x}^{2}u_{3,\neq}\|_{L^{2}} \right),
	\end{aligned}
\end{equation*}
which implies that 
\begin{equation*}
	\begin{aligned}
		I_{2}\leq&C\nu^{-1}\left(\|\partial_{x}(\partial_{z}-\kappa\partial_{y})u_{3,\neq}\|_{X_{b}}^{2}+\|\partial_{x}^{2}u_{3,\neq}\|_{X_{b}}^{2} \right)\leq C\nu^{-1}E_{7}.
	\end{aligned}
\end{equation*}
Combining the estimates of $ I_{1} $ and $ I_{2}, $ (\ref{W23}) gives the result of (i).

{\underline{\textbf{Estimate (ii).}}} Notice that
\begin{equation*}
	\mathcal{L}_{V}\partial_{x}^{2}Q_{2}=\partial_{x}^{2}\mathcal{L}_{V}Q_{2}=-\rho_{1}\nabla V\cdot\nabla\partial_{x}^{2}u_{3,\neq},
\end{equation*}	
and $ \partial_{x}^{2}Q_{2}(0)=0, $ 
applying Proposition \ref{Lvf 0}, then we get
\begin{equation*}
	\|\partial_{x}^{2}Q_{2}\|_{X_{b}}^{2}\leq C\nu^{-\frac13}\|e^{b\nu^{\frac13}t}\rho_{1}\nabla V\cdot\nabla\partial_{x}^{2}u_{3,\neq}\|_{L^{2}L^{2}}^{2}.
\end{equation*}
It follows from Lemma \ref{kappa} that 
\begin{equation}\label{rho1 infty}
	\|\rho_{1}\|_{L^{\infty}}\leq C\|\rho_{1}\|_{H^{2}}\leq C\|\widehat{u_{1,0}}\|_{H^{4}}\leq C\varepsilon_{0}.
\end{equation}
By using (\ref{V1' infty}), we obtain
\begin{equation*}
	\|\rho_{1}\nabla V\cdot\nabla\partial_{x}^{2}u_{3,\neq}\|_{L^{2}}\leq\|\rho_{1}\|_{L^{\infty}}\|\nabla V\|_{L^{\infty}}\|\nabla\partial_{x}^{2}u_{3,\neq}\|_{L^{2}}\leq C\varepsilon_{0}\|\nabla\partial_{x}^{2}u_{3,\neq}\|_{L^{2}},
\end{equation*}
which indicates that
\begin{equation*}
	\begin{aligned}
		\|\partial_{x}^{2}Q_{2}\|_{X_{b}}^{2}\leq&C\varepsilon_{0}^{2}\nu^{-\frac13}\|e^{b\nu^{\frac13}t}\nabla\partial_{x}^{2}u_{3,\neq}\|_{L^{2}L^{2}}^{2}\leq C\varepsilon_{0}^{2}\nu^{-\frac43}\|\partial_{x}^{2}u_{3,\neq}\|_{X_{b}}^{2}\leq C\varepsilon_{0}^{2}\nu^{-\frac43}E_{7}.
	\end{aligned}
\end{equation*}

For $ j\in\{1,2,3\}, $ there is
\begin{equation*}
	\begin{aligned}
		&\mathcal{L}_{V}\partial_{x}\partial_{j}Q_{2}=\partial_{x}\partial_{j}\mathcal{L}_{V}Q_{2}-\partial_{j}V\partial_{x}^{2}Q_{2}=-\partial_{j}\left(\rho_{1}\nabla V\cdot\nabla\partial_{x}u_{3,\neq} \right)-\partial_{j}V\partial_{x}^{2}Q_{2},\\
		&\partial_{x}\partial_{j}Q_{2}(0)=0.
	\end{aligned}
\end{equation*}
Applying Proposition \ref{Lvf 0} to the above equation, we have
\begin{equation}\label{tilde Q}
	\begin{aligned}
		\|\partial_{x}\partial_{j}Q_{2}\|_{X_{b}}^{2}\leq C\nu^{-\frac13}\|e^{b\nu^{\frac13}t}\partial_{j}V\partial_{x}^{2}Q_{2}\|_{L^{2}L^{2}}^{2}+C\nu^{-1}\|e^{b\nu^{\frac13}t}\rho_{1}\nabla V\cdot\nabla\partial_{x}u_{3,\neq}\|_{L^{2}L^{2}}^{2}.
	\end{aligned}
\end{equation}
Due to (\ref{V1' infty}) and (\ref{rho1 infty}), there holds
\begin{equation*}
	\begin{aligned}
		&\|\partial_{j}V\partial_{x}^{2}Q_{2}\|_{L^{2}}\leq \|\partial_{j}V\|_{L^{\infty}}\|\partial_{x}^{2}Q_{2}\|_{L^{2}}\leq C\|\partial_{x}^{2}Q_{2}\|_{L^{2}},\\
		&\|\rho_{1}\nabla V\cdot\nabla\partial_{x}u_{3,\neq}\|_{L^{2}}\leq \|\rho_{1}\|_{L^{\infty}}\|\nabla V\|_{L^{\infty}}\|\nabla\partial_{x}u_{3,\neq}\|_{L^{2}}\leq C\varepsilon_{0}\|\nabla\partial_{x}^{2}u_{3,\neq}\|_{L^{2}},
	\end{aligned}
\end{equation*}
which follow that
\begin{equation*}
	\begin{aligned}
		&\|e^{b\nu^{\frac13}t}\partial_{j}V\partial_{x}^{2}Q_{2}\|_{L^{2}L^{2}}^{2}\leq C\|e^{b\nu^{\frac13}t}\partial_{x}^{2}Q_{2}\|_{L^{2}L^{2}}^{2}\leq C\nu^{-\frac13}\|\partial_{x}^{2}Q_{2}\|_{X_{b}}^{2},\\
		&\|e^{b\nu^{\frac13}t}\rho_{1}\nabla V\cdot\nabla\partial_{x}u_{3,\neq}\|_{L^{2}L^{2}}^{2}\leq C\varepsilon_{0}^{2}\|e^{b\nu^{\frac13}t}\nabla\partial_{x}^{2}u_{3,\neq}\|_{L^{2}L^{2}}^{2}\leq C\varepsilon_{0}^{2}\nu^{-1}\|\partial_{x}^{2}u_{3,\neq}\|_{X_{b}}^{2}.
	\end{aligned}
\end{equation*}
Substituting the above estimations into (\ref{tilde Q}), we arrive at
\begin{equation*}
	\|\partial_{x}\nabla Q_{2}\|_{X_{b}}^{2}\leq C\nu^{-\frac23}\|\partial_{x}^{2}Q_{2}\|_{X_{b}}^{2}+C\varepsilon_{0}^{2}\nu^{-2}\|\partial_{x}^{2}u_{3,\neq}\|_{X_{b}}^{2}\leq C\varepsilon_{0}^{2}\nu^{-2}E_{7},
\end{equation*}
which gives the second inequality of (ii).

{\underline{\textbf{Estimate (iii).}}} 
Due to 
\begin{equation*}
	\begin{aligned}
		\mathcal{L}_{V}(\rho_{1}f)-\rho_{1}\mathcal{L}_{V}f=&(\partial_{t}\rho_{1}-\nu\triangle\rho_{1})f-2\nu\nabla\rho_{1}\cdot\nabla f\\=&(\partial_{t}\rho_{1}+\nu\triangle\rho_{1})f-2\nu\nabla\cdot(f\nabla\rho_{1}),
	\end{aligned}
\end{equation*}
and $ \mathcal{L}_{V}Q_{2}=\rho_{1}\mathcal{L}_{V}Q_{4}, $ there holds
\begin{equation*}
	\begin{aligned}
		\mathcal{L}_{V}\partial_{x}\left(Q_{2}-\rho_{1}Q_{4} \right)=&\partial_{x}\mathcal{L}_{V}\left(Q_{2}-\rho_{1}Q_{4} \right)=\partial_{x}\left(\rho_{1}\mathcal{L}_{V}Q_{4}-\mathcal{L}_{V}(\rho_{1}Q_{4}) \right)\\=&-\partial_{x}\left((\partial_{t}\rho_{1}+\nu\triangle\rho_{1}) Q_{4}-2\nu\nabla\cdot(Q_{4}\nabla\rho_{1})\right)\\=&-\triangle W+2\nu\nabla\cdot\left(\partial_{x}Q_{4}\nabla\rho_{1} \right),
	\end{aligned}
\end{equation*}	
where $ \triangle W=\left(\partial_{t}\rho_{1}+\nu\triangle\rho_{1} \right)\partial_{x}Q_{4}. $
Applying {Proposition \ref{Lvf 0}}, we get
\begin{equation}\label{Lv Q}
	\begin{aligned}
		&\|\partial_{x}(Q_{2}-\rho_{1}Q_{4} )\|_{X_{b}}^{2}\leq C\left(\nu^{-1}\|e^{b\nu^{\frac13}t}\nabla W\|_{L^{2}L^{2}}^{2}+\nu\|e^{b\nu^{\frac13}t}\partial_{x}Q_{4}\nabla\rho_{1}\|_{L^{2}L^{2}}^{2} \right).
	\end{aligned}
\end{equation}
It follows from Lemma \ref{kappa} that $ \|\nabla\rho_{1}\|_{H^{1}}\leq\|\rho_{1}\|_{H^{2}}\leq C\|\widehat{u_{1,0}}\|_{H^{4}}\leq C\varepsilon_{0} $ and
\begin{equation*}
	\begin{aligned}
		\|\partial_{t}\rho_{1}+\nu\triangle\rho_{1}\|_{L^{2}}\leq&\|\partial_{t}\rho_{1}\|_{L^{2}}+\nu\|\rho_{1}\|_{H^{2}}\\\leq&C\left(\|\partial_{t}\widehat{u_{1,0}}\|_{H^{2}}+\nu\|\widehat{u_{1,0}}\|_{H^{4}} \right)\leq C\nu\varepsilon_{0}.
	\end{aligned}
\end{equation*}
Combining it with Lemma \ref{sob_14}, one obtains
\begin{equation}\label{a}
	\begin{aligned}
		\|\nabla W\|_{L^{2}}=&\|\nabla\triangle^{-1}\triangle W\|_{L^{2}}=\|\nabla\triangle^{-1}[\left(\partial_{t}\rho_{1}+\nu\triangle\rho_{1} \right)\partial_{x}Q_{4} ]\|_{L^{2}}\\\leq&C\|\partial_{t}\rho_{1}+\nu\triangle\rho_{1}\|_{L^{2}}\left(\|\partial_{x}Q_{4}\|_{L^{2}}+\|\partial_{x}\left(\partial_{z}-\kappa\partial_{y} \right)Q_{4}\|_{L^{2}}\right)\\\leq&C\nu\varepsilon_{0}\left(\|\partial_{x}^{2}Q_{4}\|_{L^{2}}+\|\partial_{x}(\partial_{z}-\kappa\partial_{y})Q_{4}\|_{L^{2}}\right)
	\end{aligned}
\end{equation}
and
\begin{equation}\label{b}
	\begin{aligned}
		\|\partial_{x}Q_{4}\nabla\rho_{1}\|_{L^{2}}\leq&C\|\nabla\rho_{1}\|_{H^{1}}\left(\|\partial_{x}Q_{4}\|_{L^{2}}+\|\partial_{x}(\partial_{z}-\kappa\partial_{y}) Q_{4}\|_{L^{2}}\right)\\\leq&C\varepsilon_{0}\left(\|\partial_{x}^{2}Q_{4}\|_{L^{2}}+\|\partial_{x}(\partial_{z}-\kappa\partial_{y})Q_{4}\|_{L^{2}} \right).
	\end{aligned}
\end{equation}
Substituting (\ref{a}) and (\ref{b}) into (\ref{Lv Q}), and using the result of (i), we obtain
\begin{equation*}
	\begin{aligned}
		\|\partial_{x}(Q_{2}-\rho_{1}Q_{4})\|_{X_{b}}^{2}\leq&C\nu\varepsilon_{0}^{2}\left(\|e^{b\nu^{\frac13}t}\partial_{x}^{2}Q_{4}\|_{L^{2}L^{2}}^{2}+\|e^{b\nu^{\frac13}t}\partial_{x}(\partial_{z}-\kappa\partial_{y})Q_{4}\|_{L^{2}L^{2}}^{2} \right)\\\leq&C\nu\varepsilon_{0}^{2}\nu^{-\frac13}\left(\|\partial_{x}^{2}Q_{4}\|_{X_{b}}^{2}+\|\partial_{x}(\partial_{z}-\kappa\partial_{y})Q_{4}\|_{X_{b}}^{2} \right)\leq C\nu^{-\frac23}\varepsilon_{0}^{2}E_{7}.
	\end{aligned}
\end{equation*}

{\textbf{\underline{Estimate (iv).}}} Due to $ \mathcal{L}_{V}Q_{3}=\Theta_{\neq}-(\partial_{y}+\kappa\partial_{z})\triangle^{-1}\partial_{y}\Theta_{\neq}, $ there holds
\begin{equation}\label{Q3 xx}
	\mathcal{L}_{V}\partial_{x}^{2}Q_{3}=\partial_{x}^{2}\Theta_{\neq}-(\partial_{y}+\kappa\partial_{z})\triangle^{-1}\partial_{y}\partial_{x}^{2}\Theta_{\neq}.
\end{equation}
As $ \|\kappa\|_{H^{3}}\leq C $, then applying Proposition \ref{Lvf 0} to (\ref{Q3 xx}), we obtain
\begin{equation*}\label{Q3 Xb}
	\begin{aligned}
		\|\partial_{x}^{2}Q_{3}\|_{X_{b}}^{2}\leq&C\big(\nu^{-\frac13}\|e^{b\nu^{\frac13}t}\partial_{x}^{2}\Theta_{\neq}\|_{L^{2}L^{2}}^{2}+\nu^{-\frac13}\|e^{b\nu^{\frac13}t}(\partial_{y}+\kappa\partial_{z})\triangle^{-1}\partial_{y}\partial_{x}^{2}\Theta_{\neq}\|_{L^{2}L^{2}}^{2} \big)\\\leq&C\nu^{-\frac23}\|\partial_{x}^{2}\Theta_{\neq}\|_{X_{b}}^{2} \leq C\nu^{-\frac23}E_{5}^{2},
	\end{aligned}
\end{equation*}
which gives the first inequality of (iv).

Since
$\mathcal{L}_{V}\nabla\partial_{x}Q_{3}=-\nabla V\partial_{x}^{2}Q_{3}+\partial_{x}\nabla\Theta_{\neq}-\nabla(\partial_{y}+\kappa\partial_{z})\triangle^{-1}\partial_{y}\partial_{x}\Theta_{\neq},$
 applying {Proposition \ref{Lvf 0}} and  Lemma \ref{kappa},  there holds
\begin{equation*}
	\begin{aligned}
		\|\nabla\partial_{x}Q_{3}\|_{X_{b}}^{2}\leq&C\big(\nu^{-\frac13}\|e^{b\nu^{\frac13}t}\nabla V\partial_{x}^{2}Q_{3}\|_{L^{2}L^{2}}^{2}+\nu^{-1}\|e^{b\nu^{\frac13}t}\partial_{x}\Theta_{\neq}\|_{L^{2}L^{2}}^{2}\\&+\nu^{-1}\|e^{b\nu^{\frac13}t}(\partial_{y}+\kappa\partial_{z})\triangle^{-1}\partial_{y}\partial_{x}\Theta_{\neq}\|_{L^{2}L^{2}}^{2} \big)\\\leq&C\big(\nu^{-\frac23}\|\partial_{x}^{2}Q_{3}\|_{X_{b}}^{2}+\nu^{-\frac43}\|\partial_{x}^{2}\Theta_{\neq}\|_{X_{b}}^{2} \big)\leq C\nu^{-\frac43}E_{5}^{2}.
	\end{aligned}
\end{equation*}
\end{proof}

The following lemma provides estimates of regular part.
\begin{proposition}\label{E5 E6}
It holds that
\begin{equation*}
	E_6^2\leq C(E_{7}+\nu^{\frac43}\|\triangle u_{3,\neq}\|_{X_{b}}^{2})\leq C\left(\|u_{\rm in}\|_{H^{2}}^{2}+\nu^{-2}E_{4}^{4}+\nu^{-2}E_{5}^{2} \right),
\end{equation*}
\end{proposition}
\begin{proof}
{\underline{\textbf{Step I. Estimate $ \|\triangle u_{3,\neq}-2\partial_{x}Q_{4}\|_{X_{b}}. $}}}
As
\begin{equation*}
	\begin{aligned}
		\mathcal{L}_{V}\triangle u_{3,\neq}=&\partial_{t}\triangle u_{3,\neq}-\nu\triangle u_{3,\neq}+V\partial_{x}\triangle u_{3,\neq}\\=&\triangle\mathcal{L}_{V}u_{3,\neq}-\triangle V\partial_{x}u_{3,\neq}-2\nabla V\cdot\nabla\partial_{x}u_{3,\neq}\\=&\triangle\mathcal{L}_{V}u_{3,\neq}-\triangle V\partial_{x}u_{3,\neq}+2\mathcal{L}_{V}\partial_{x}Q_{4},
	\end{aligned}
\end{equation*}
which follows that
\begin{equation}\label{lv}
	\mathcal{L}_{V}\left(\triangle u_{3,\neq}-2\partial_{x}Q_{4} \right)=\triangle\mathcal{L}_{V}u_{3,\neq}-\triangle V\partial_{x}u_{3,\neq}.
\end{equation}
Thanks to (\ref{p5}), (\ref{U_neq}) and $ \triangle P^{N_{3}}_{\neq}=\partial_{y}\Theta_{\neq} $, there holds
\begin{equation*}
	\begin{aligned}
		\triangle\mathcal{L}_{V}u_{3,\neq}=&-\triangle\partial_{z}P^{5}-\triangle h_{3,\neq}-\partial_{z}\triangle P^{N_{3}}_{\neq}-\triangle\left(\widetilde{u_{1,0}}\partial_{x}u_{3,\neq}\right)\\=&2\partial_{z}(\partial_{y}V\partial_{x}Q)-\triangle h_{3,\neq}-\partial_{y}\partial_{z}\Theta_{\neq}-\triangle\left(\widetilde{u_{1,0}}\partial_{x}u_{3,\neq}\right),
	\end{aligned}
\end{equation*}
then (\ref{lv}) yields
\begin{equation*}
	\mathcal{L}_{V}\left(\triangle u_{3,\neq}-2\partial_{x}Q_{4} \right)=2\partial_{z}(\partial_{y}V\partial_{x}Q)-\triangle h_{3,\neq}-\partial_{y}\partial_{z}\Theta_{\neq}-\triangle V\partial_{x}u_{3,\neq}-\triangle\left(\widetilde{u_{1,0}}\partial_{x}u_{3,\neq}\right).
\end{equation*}
Applying {Proposition \ref{Lvf 0}} to the above equation, we obtain
\begin{equation}\label{u3 hatQ}
	\begin{aligned}
		\|\triangle u_{3,\neq}-2\partial_{x}Q_{4}\|_{X_{b}}^{2}\leq&C\big(\|(\triangle u_{3,{\rm in}})_{\neq}\|_{L^{2}}^{2}+\nu^{-\frac13}\|e^{b\nu^{\frac13}t}\triangle V\partial_{x}u_{3,\neq}\|_{L^{2}L^{2}}^{2}\\&+\nu^{-\frac13}\|e^{b\nu^{\frac13}t}\partial_{z}(\partial_{y}V\partial_{x}Q)\|_{L^{2}L^{2}}^{2}+\nu^{-1}\|e^{b\nu^{\frac13}t}\nabla h_{3,\neq}\|_{L^{2}L^{2}}^{2}\\&+\nu^{-1}\|e^{b\nu^{\frac13}t}\partial_{z}\Theta_{\neq}\|_{L^{2}L^{2}}^{2}+\nu^{-1}\|e^{b\nu^{\frac13}t}\nabla(\widetilde{u_{1,0}}\partial_{x}u_{3,\neq})\|_{L^{2}L^{2}}^{2} \big),
	\end{aligned}
\end{equation}
where we use $ \partial_{x}Q_{4}(0)=0. $ Due to $ \|\triangle V\|_{L^{\infty}}\leq C\|\widehat{u_{1,0}}\|_{H^{4}}\leq C, $ there holds
\begin{equation}\label{c}
	\begin{aligned}
		\|e^{b\nu^{\frac13}t}\triangle V\partial_{x}u_{3,\neq}\|_{L^{2}L^{2}}^{2}\leq&\|\triangle V\|_{L^{\infty}L^{\infty}}^{2}\|e^{b\nu^{\frac13}t}\partial_{x}u_{3,\neq}\|_{L^{2}L^{2}}^{2}\\\leq&C\|e^{b\nu^{\frac13}t}\partial_{x}^{2}u_{3,\neq}\|_{L^{2}L^{2}}^{2}\leq C\nu^{-\frac13}\|\partial_{x}^{2}u_{3,\neq}\|_{X_{b}}^{2}\leq C\nu^{-\frac13}E_{7}.
	\end{aligned}
\end{equation}
Moreover, for $ \|\partial_{y}V\|_{H^3}\leq C\left(1+\|\widehat{u_{1,0}}\|_{H^{4}} \right)\leq C, $ one deduces
\begin{equation}\label{d}
	\begin{aligned}
		&\|\partial_{z}(\partial_{y}V\partial_{x}Q)\|_{L^{2}}
		\leq\|\partial_{y}V\|_{H^3}\|\partial_{x}Q\|_{H^{1}}\leq C\|\nabla \partial_{x}Q\|_{L^{2}},\\
		&\|e^{b\nu^{\frac13}t}
		\partial_{z}(\partial_{y}V\partial_{x}Q)\|_{L^{2}L^{2}}^{2}
		\leq C\|e^{b\nu^{\frac13}t}\nabla\partial_{x}Q\|_{L^{2}L^{2}}^{2}
		\leq C\nu^{-\frac13}E_{7}.
	\end{aligned}
\end{equation}
Using Lemma \ref{lem zero nonzero} and Lemma \ref{lem nonzreo}, we have
\begin{equation}\label{e}
	\begin{aligned}
		&\|e^{b\nu^{\frac13}t}\nabla\left(\widetilde{u_{1,0}}\partial_{x}u_{3,\neq} \right)\|_{L^{2}L^{2}}^{2}\leq C\|\widetilde{u_{1,0}}\|_{L^{\infty}H^{2}}^{2}\|e^{b\nu^{\frac13}t}\partial_{x}\nabla u_{3,\neq}\|_{L^{2}L^{2}}^{2}\leq C\nu^{\frac43}\varepsilon_{0}^{2}\|\triangle u_{3,\neq}\|_{X_{b}}^{2}.
	\end{aligned}
\end{equation}

Collecting (\ref{u3 hatQ}), (\ref{c}), (\ref{d}) and (\ref{e}), we obtain
\begin{equation}\label{result u3 hatQ}
	\begin{aligned}
		\|\triangle u_{3,\neq}-2\partial_{x}Q_{4}\|_{X_{b}}^{2}\leq C\big(&\|u_{\rm in}\|_{H^{2}}^{2}
		+\nu^{-\frac23}E_{7}+\nu^{-\frac43}E_{5}^{2}+\nu^{-1}\|e^{b\nu^{\frac13}t}\nabla h_{3,\neq}\|_{L^{2}L^{2}}^{2}\\
		&+\nu^{\frac13}\varepsilon_{0}^{2}\|\triangle u_{3,\neq}\|_{X_{b}}^{2} \big).
	\end{aligned}
\end{equation}

{\underline{\textbf{Step II. Estimate $ \|\triangle Q_{1}\|_{X_{b}}. $}}}
Direct calculations indicate that
\begin{equation*}
	\widetilde{\mathcal{L}_{V}}Q+H_{2}=(\partial_{t}\kappa-\nu\triangle\kappa)u_{3,\neq}-2\nu\nabla\kappa\cdot\nabla u_{3,\neq}+\Theta_{\neq}-(\partial_{y}+\kappa\partial_{z})\triangle^{-1}\partial_{y}\Theta_{\neq}.
\end{equation*}	
Using the following decomposition
\begin{equation*}
	\nabla\kappa\cdot\nabla u_{3,\neq}=\rho_{1}\nabla V\cdot\nabla u_{3,\neq}+\rho_{2}(\partial_{z}-\kappa\partial_{y})u_{3,\neq},
\end{equation*}
and $ Q=Q_{1}+\nu Q_{2}+Q_{3}, $ we get
\begin{equation}\label{tilde Lv Q}
	\begin{aligned}
		&\widetilde{\mathcal{L}_{V}}Q_{1}+H_{2}-(\partial_{t}\kappa-\nu\triangle\kappa)u_{3,\neq}\\=&-2\nu\left(\rho_{1}\nabla V\cdot\nabla u_{3,\neq}+\rho_{2}(\partial_{z}-\kappa\partial_{y})u_{3,\neq} \right)+\Theta_{\neq}-(\partial_{y}+\kappa\partial_{z})\triangle^{-1}\partial_{y}\Theta_{\neq}-\nu\widetilde{\mathcal{L}_{V}}Q_{2}-\widetilde{\mathcal{L}_{V}}Q_{3}\\=&-2\nu\rho_{2}(\partial_{z}-\kappa\partial_{y})u_{3,\neq}-\nu J_{11}-K_{11},
	\end{aligned}
\end{equation}
where 
\begin{equation*}
	J_{11}=\widetilde{\mathcal{L}_{V}}Q_{2}+2\rho_{1}\nabla V\cdot\nabla u_{3,\neq},\quad K_{11}=\widetilde{\mathcal{L}_{V}}Q_{3}-\left[\Theta_{\neq}-(\partial_{y}+\kappa\partial_{z})\triangle^{-1}\partial_{y}\Theta_{\neq} \right].
\end{equation*}
Applying {Proposition \ref{tilde Lv}} to (\ref{tilde Lv Q}), there holds
\begin{equation}\label{tildeQ''}
	\begin{aligned}
		\|\triangle Q_{1}\|_{X_{b}}^{2}\leq&C\big( \|\triangle Q_{\rm in}\|_{L^{2}}^{2}+\nu^{-1}\|e^{b\nu^{\frac13}t}\nabla H_{2}\|_{L^{2}L^{2}}^{2}+\nu\|e^{b\nu^{\frac13}t}\nabla J_{11}\|_{L^{2}L^{2}}^{2}\\&+\nu^{-1}\|e^{b\nu^{\frac13}t}\nabla K_{11}\|_{L^{2}L^{2}}^{2}+\nu\|e^{b\nu^{\frac13}t}\nabla\left(\rho_{2}(\partial_{z}-\kappa\partial_{y})u_{3,\neq} \right)\|_{L^{2}L^{2}}^{2}\\&+\nu^{-1}\|e^{b\nu^{\frac13}t}\nabla\left((\partial_{t}\kappa-\nu\triangle\kappa)u_{3,\neq} \right)\|_{L^{2}L^{2}}^{2}
		\big).
	\end{aligned}
\end{equation}
It follows from {Lemma \ref{kappa}} that $ \|\rho_{2}\|_{H^{2}}\leq C\|\widehat{u_{1,0}}\|_{H^{4}}\leq C\varepsilon_{0} $ and
\begin{equation*}
	\|\partial_{t}\kappa-\nu\triangle\kappa\|_{H^{1}}\leq \|\partial_{t}\kappa\|_{H^{1}}+\nu\|\kappa\|_{H^{3}}\leq C\left(\|\partial_{t}\widehat{u_{1,0}}\|_{H^{2}}+\nu\|\widehat{u_{1,0}}\|_{H^{4}} \right)\leq C\nu\varepsilon_{0},
\end{equation*}
which imply that 
\begin{equation}\label{11}
	\begin{aligned}
		\|\nabla\left(\rho_{2}(\partial_{z}-\kappa\partial_{y})u_{3,\neq} \right)\|_{L^{2}}\leq C\|\rho_{2}\|_{H^{2}}\|\nabla(\partial_{z}-\kappa\partial_{y})u_{3,\neq}\|_{L^{2}}\leq C\varepsilon_{0}\|\nabla\partial_{x}(\partial_{z}-\kappa\partial_{y})u_{3,\neq}\|_{L^{2}}.
	\end{aligned}
\end{equation}
Besides, using {Lemma \ref{sob_14}}, one obtains that 
\begin{equation}\label{12}
	\begin{aligned}
		&\|\nabla\left((\partial_{t}\kappa-\nu\triangle\kappa)u_{3,\neq} \right)\|_{L^{2}}
		\leq C\nu\varepsilon_{0}\left(\|\nabla\partial_{x}^{2}u_{3,\neq}\|_{L^{2}}+\|\nabla\partial_{x}(\partial_{z}-\kappa\partial_{y})u_{3,\neq}\|_{L^{2}} \right).
	\end{aligned}
\end{equation}
By Lemma \ref{kappa}, we have 
\begin{equation}\label{con}
	\begin{aligned}
		&\|\rho_{1}\|_{L^{\infty}}\leq C\|\rho_{1}\|_{H^{2}}\leq C\|\widehat{u_{1,0}}\|_{H^{4}}\leq C\varepsilon_{0},\\& \|\kappa\|_{H^{3}}\leq C\|\widehat{u_{1,0}}\|_{H^{4}}\leq C,\quad \|\partial_{y}V\|_{L^{\infty}}\leq C.
	\end{aligned}
\end{equation}
For $ K_{11} $ satisfying
\begin{equation*}
	\begin{aligned}
		K_{11}=&\widetilde{\mathcal{L}_{V}}Q_{3}-\left[\Theta_{\neq}-(\partial_{y}+\kappa\partial_{z})\triangle^{-1}\partial_{y}\Theta_{\neq} \right]\\=&\mathcal{L}_{V}Q_{3}-2(\partial_{y}+\kappa\partial_{z})\triangle^{-1}(\partial_{y}V\partial_{x}Q_{3})-\left[\Theta_{\neq}-(\partial_{y}+\kappa\partial_{z})\triangle^{-1}\partial_{y}\Theta_{\neq} \right]\\=&-2(\partial_{y}+\kappa\partial_{z})\triangle^{-1}(\partial_{y}V\partial_{x}Q_{3}),
	\end{aligned}
\end{equation*}
then using (\ref{con}), we get
%	\begin{equation*}
	%		\begin{aligned}
		%			\|\nabla K_{11}\|_{L^{2}}\leq&2\|(\partial_{y}+\kappa\partial_{z})\triangle^{-1}(\partial_{y}V\partial_{x}Q_{3})\|_{H^{1}}\\\leq&C\left(\|\partial_{y}\triangle^{-1}(\partial_{y}V\partial_{x}Q_{3})\|_{H^{1}}+\|\kappa\|_{H^{3}}\|\partial_{z}\triangle^{-1}(\partial_{y}V\partial_{x}Q_{3})\|_{H^{1}} \right)\\\leq&C\|\triangle^{-1}(\partial_{y}V\partial_{x}Q_{3})\|_{H^{2}}\\\leq&C\|\partial_{y}V\partial_{x}Q_{3}\|_{L^{2}}\leq C\|\partial_{y}V\|_{L^{\infty}}\|\partial_{x}Q_{3}\|_{L^{2}}\leq C\|\partial_{x}Q_{3}\|_{L^{2}}.
		%		\end{aligned}
	%	\end{equation*}
\begin{equation*}
	\begin{aligned}
		\|\nabla K_{11}\|_{L^{2}}\leq&2\|(\partial_{y}+\kappa\partial_{z})\triangle^{-1}(\partial_{y}V\partial_{x}Q_{3})\|_{H^{1}}
		\\\leq&C\|\partial_{y}V\partial_{x}Q_{3}\|_{L^{2}}\leq C\|\partial_{y}V\|_{L^{\infty}}
		\|\partial_{x}Q_{3}\|_{L^{2}}\leq C\|\partial_{x}Q_{3}\|_{L^{2}}.
	\end{aligned}
\end{equation*}
From this, along with (iv) of Lemma \ref{lem W}, one deduces
\begin{equation}\label{13}
	\begin{aligned}
		\nu^{-1}\|e^{b\nu^{\frac13}t}\nabla K_{11}\|_{L^{2}L^{2}}^{2}\leq&C\nu^{-\frac43}\|\partial_{x}^{2}Q_{3}\|_{X_{b}}^{2}\leq C\nu^{-2}E_{5}^{2}.
	\end{aligned}
\end{equation}
Next we estimate $ J_{11}. $ We can rewrite $ J_{11} $ as follows:
\begin{equation}\label{J11}
	\begin{aligned}
		J_{11}=&\mathcal{L}_{V}Q_{2}-2(\partial_{y}+\kappa\partial_{z})\triangle^{-1}(\partial_{y}V\partial_{x}Q_{2})+2\rho_{1}\nabla V\cdot\nabla u_{3,\neq}\\=&-2(\partial_{y}+\kappa\partial_{z})\triangle^{-1}(\partial_{y}V\partial_{x}Q_{2})+\rho_{1}\nabla V\cdot\nabla u_{3,\neq}\\=&-(\partial_{y}+\kappa\partial_{z})\triangle^{-1}(\partial_{y}VJ_{12})-(\partial_{y}+\kappa\partial_{z})\triangle^{-1}(\partial_{y}V\rho_{1}\triangle u_{3,\neq})+\rho_{1}\nabla V\cdot\nabla u_{3,\neq},
	\end{aligned}
\end{equation}
where $$ J_{12}=2\partial_{x}Q_{2}-\rho_{1}\triangle u_{3,\neq}=2\partial_{x}(Q_{2}-\rho_{1}Q_{4})-\rho_{1}\left(\triangle u_{3,\neq}-2\partial_{x}Q_{4} \right). $$
Using (\ref{con}), we arrive
\begin{equation*}
	\begin{aligned}
		J_{12}\leq C\left(\|\partial_{x}(Q_{2}-\rho_{1}Q_{4})\|_{L^{2}}+\varepsilon_{0}\|\triangle u_{3,\neq}-2\partial_{x}Q_{4}\|_{L^{2}} \right)
	\end{aligned}
\end{equation*}
and
\begin{equation}\label{J12 0}
	\begin{aligned}
		&\|(\partial_{y}+\kappa\partial_{z})\triangle^{-1}(\partial_{y}VJ_{12})\|_{H^{1}}\\\leq&C\left(\|\partial_{y}\triangle^{-1}(\partial_{y}VJ_{12})\|_{H^{1}}+\|\kappa\|_{H^{3}}\|\partial_{z}\triangle^{-1}(\partial_{y}VJ_{12})\|_{H^{1}} \right)\\\leq&C\|J_{12}\|_{L^{2}}\leq C\left(\|\partial_{x}(Q_{2}-\rho_{1}Q_{4})\|_{L^{2}}+\varepsilon_{0}\|\triangle u_{3,\neq}-2\partial_{x}Q_{4}\|_{L^{2}} \right).
	\end{aligned}
\end{equation}	
Thanks to $ \nabla V\cdot\nabla u_{3,\neq}=\partial_{y}V(\partial_{y}+\kappa\partial_{z})u_{3,\neq}, $ we write
\begin{equation*}
	\begin{aligned}
		&-(\partial_{y}+\kappa\partial_{z})\triangle^{-1}(\partial_{y}V\rho_{1}\triangle u_{3,\neq})+\rho_{1}\nabla V\cdot\nabla u_{3,\neq}\\=&-\left[(\partial_{y}+\kappa\partial_{z})\triangle^{-1}, \partial_{y}V\rho_{1} \right]\triangle u_{3,\neq}\\=&-\left[\partial_{y}\triangle^{-1}, \partial_{y}V\rho_{1} \right]\triangle u_{3,\neq}-\kappa\left[\partial_{z}\triangle^{-1},\partial_{y}V\rho_{1} \right]\triangle u_{3,\neq}.
	\end{aligned}
\end{equation*}
Note that $ \|\kappa\|_{H^{3}}\leq C $ and $ \|\partial_{y}V\rho_{1}\|_{H^{2}}\leq \|\partial_{y}V\|_{L^{\infty}}\|\rho_{1}\|_{H^{2}}\leq C\varepsilon_{0}, $ by Lemma \ref{sob_14}, 
and we get
\begin{equation}\label{J12 1}
	\begin{aligned}
		&\|-(\partial_{y}+\kappa\partial_{z})\triangle^{-1}\left(\partial_{y}V\rho_{1}\triangle u_{3,\neq} \right)+\rho_{1}\nabla V\cdot\nabla u_{3,\neq}\|_{H^{1}}\\\leq&\|\left[\partial_{y}\triangle^{-1},\partial_{y}V\rho_{1} \right]\triangle u_{3,\neq}\|_{H^{1}}+\|\kappa\|_{L^{\infty}}\|\left[\partial_{z}\triangle^{-1},\partial_{y}V\rho_{1} \right]\triangle u_{3,\neq}\|_{H^{1}}\\\leq&C\varepsilon_{0}\left(\|\nabla\partial_{x}^{2}u_{3,\neq}\|_{L^{2}}+\|\partial_{x}(\partial_{z}-\kappa\partial_{y})\nabla u_{3,\neq}\|_{L^{2}} \right).
	\end{aligned}
\end{equation}	
It follows from (\ref{J11}), (\ref{J12 0}) and (\ref{J12 1}) that
\begin{equation}\label{J11 H1}
	\begin{aligned}
		\|\nabla J_{11}\|_{L^{2}}&\leq C\left(\|\partial_{x}(Q_{2}-\rho_{1}Q_{4})\|_{L^{2}}+\varepsilon_{0}\|\triangle u_{3,\neq}-2\partial_{x}Q_{4}\|_{L^{2}} \right)\\&+C\varepsilon_{0}\left(\|\nabla\partial_{x}^{2}u_{3,\neq}\|_{L^{2}}+\|\nabla\partial_{x}(\partial_{z}-\kappa\partial_{y})u_{3,\neq}\|_{L^{2}} \right).
	\end{aligned}
\end{equation}

Collecting (\ref{result u3 hatQ}), (\ref{11}), (\ref{12}), (\ref{13}), (\ref{J11 H1}) and (iii) of {Lemma \ref{lem W}}, we get
\begin{equation*}
	\begin{aligned}
		&\nu\|e^{b\nu^{\frac13}t}\nabla J_{11}\|_{L^{2}L^{2}}^{2}+\nu^{-1}\|e^{b\nu^{\frac13}t}\nabla K_{11}\|_{L^{2}L^{2}}^{2}\\&+\nu\|e^{b\nu^{\frac13}t}\nabla\left(\rho_{2}(\partial_{z}-\kappa\partial_{y})u_{3,\neq} \right)\|_{L^{2}L^{2}}^{2}+\nu^{-1}\|e^{b\nu^{\frac13}t}\nabla\left((\partial_{t}\kappa-\nu\triangle\kappa)u_{3,\neq} \right)\|_{L^{2}L^{2}}^{2}\\\leq&C\nu^{\frac23}\left(\|\partial_{x}\left(Q_{2}-\rho_{1}Q_{4} \right)\|_{X_{b}}^{2}+\varepsilon_{0}^{2}\|\triangle u_{3,\neq}-2\partial_{x}Q_{4}\|_{X_{b}}^{2} \right)\\&+C\varepsilon_{0}^{2}\left(\|\partial_{x}^{2}u_{3,\neq}\|_{X_{b}}^{2}+\|\partial_{x}(\partial_{z}-\kappa\partial_{y})u_{3,\neq}\|_{X_{b}}^{2} \right)+C\nu^{-2}E_{5}^{2}\\\leq&C\varepsilon_{0}^{2}\left(E_{7}+\|u_{
		\rm in}\|_{H^{2}}^{2} \right)+C\nu^{-\frac13}\|e^{a\nu^{\frac13}t}\nabla h_{3,\neq}\|_{L^{2}L^{2}}^{2}+C\nu^{-2}E_{5}^{2},
	\end{aligned}
\end{equation*}
substituting it into (\ref{tildeQ''}), one obtains
\begin{equation}\label{bar Q''}
	\begin{aligned}
		\|\triangle Q_{1}\|_{X_{b}}^{2}\leq C&\big(\|u_{\rm in}\|_{H^{2}}^{2}+\nu^{-1}\|e^{b\nu^{\frac13}t}\nabla H_{2}\|_{L^{2}L^{2}}^{2}+\varepsilon_{0}^{2}E_{7}\\&+\nu^{-\frac13}\|e^{b\nu^{\frac13}t}\nabla h_{3,\neq}\|_{L^{2}L^{2}}^{2}+\nu^{-2}E_{5}^{2} \big).
	\end{aligned}
\end{equation}

{\underline{\textbf{Step III. Estimate $ \|\partial_{x}^{2}u_{j,\neq}\|_{X_{b}}+\|\partial_{x}(\partial_{z}-\kappa\partial_{y})u_{j,\neq}\|_{X_{b}}. $}}}	
Let $ P^{5}=P^{5,1}+P^{5,2}+P^{5,3}, $ where
\begin{equation*}
	\triangle P^{5,1}=-2\partial_{y}V\partial_{x}Q_{1},\quad \triangle P^{5,2}=-2\nu\partial_{y}V\partial_{x}Q_{2},\quad \triangle P^{5,3}=-2\partial_{y}V\partial_{x}Q_{3}.
\end{equation*}	
From (\ref{U_neq}), we get
\begin{equation*}
	\left\{
	\begin{array}{lr}
		\mathcal{L}_{V}u_{2,\neq}+h_{2,\neq}=\Theta_{\neq}-\triangle^{-1}\partial_{y}^{2}\Theta_{\neq}-\partial_{y}P^{5,1}-\partial_{y}P^{5,2}-\partial_{y}P^{5,3}-\widetilde{u_{1,0}}\partial_{x}u_{2,\neq},\\
		\mathcal{L}_{V}u_{3,\neq}+h_{3,\neq}=-\triangle^{-1}\partial_{y}\partial_{z}\Theta_{\neq}-\partial_{z}P^{5,1}-\partial_{z}P^{5,2}-\partial_{z}P^{5,3}{-\widetilde{u_{1,0}}\partial_{x}u_{3,\neq}}.
	\end{array}
	\right.
\end{equation*}
For $ j\in\{2,3\},$  $\widetilde{u_{1,0}}\partial_{x}u_{J,\neq}$ can be regarded as a perturbation,
by Proposition \ref{Lvf}, then there holds
\begin{equation}\label{Uj0}
	\begin{aligned}
		&\|\partial_{x}^{2}u_{j,\neq}\|_{X_{b}}^{2}+\|\partial_{x}(\partial_{z}-\kappa\partial_{y})u_{j,\neq}\|_{X_{b}}^{2}\\
		\leq&C\big(\|u_{j,\rm in}\|_{H^{2}}^{2}+\|e^{b\nu^{\frac13}t}\triangle\partial_{j}P^{5,1}
		\|_{L^{2}L^{2}}^{2}+\nu^{-1}\|e^{b\nu^{\frac13}t}\partial_{x}h_{j}\|_{L^{2}L^{2}}^{2}\\&+\nu^{-\frac13}\|e^{b\nu^{\frac13}t}\partial_{x}\nabla\partial_{j}P^{5,2}\|_{L^{2}L^{2}}^{2}+\|e^{b\nu^{\frac13}t}\triangle\partial_{j}P^{5,3}\|_{L^{2}L^{2}}^{2}+\nu^{-\frac43}\|\partial_{x}^{2}\Theta_{\neq}\|_{X_{b}}^{2}
		\big),
	\end{aligned}
\end{equation}	
where we use
\begin{equation*}
	\|\partial_{x}^{2}f\|_{L^{2}}+\|\partial_{x}(\partial_{z}-\kappa\partial_{y})f\|_{L^{2}}\leq C(1+\|\kappa\|_{L^{\infty}})\|\partial_{x}\nabla f\|_{L^{2}}\leq C\|\partial_{x}\nabla f\|_{L^{2}}.
\end{equation*}
Notice that
\begin{equation}\label{Uj 0}
	\begin{aligned}
		\|e^{b\nu^{\frac13}t}\triangle\partial_{j}P^{5,1}\|_{L^{2}L^{2}}^{2}\leq C\|e^{b\nu^{\frac13}t}\nabla\partial_{x}Q_{1}\|_{L^{2}L^{2}}^{2}\leq C\|\triangle Q_{1}\|_{X_{b}}^{2}.
	\end{aligned}
\end{equation}
Moreover, for
\begin{equation*}
	\begin{aligned}
		\|\partial_{x}\nabla\partial_{j}P^{5,2}\|_{L^{2}}\leq\|\partial_{x}\triangle P^{5,2}\|_{L^{2}}\leq 2\nu\|\partial_{y}V\|_{L^{\infty}}\|\partial_{x}^{2}Q_{2}\|_{L^{2}}\leq C\nu\|\partial_{x}^{2}Q_{2}\|_{L^{2}},
	\end{aligned}
\end{equation*}
we also have
\begin{equation}\label{Uj 1}
	\|e^{b\nu^{\frac13}t}\partial_{x}\nabla\partial_{j}P^{5,2}\|_{L^{2}L^{2}}^{2}\leq C\nu^{2}\|e^{b\nu^{\frac13}t}\partial_{x}^{2}Q_{2}\|_{L^{2}L^{2}}^{2}\leq C\nu^{\frac53}\|\partial_{x}^{2}Q_{2}\|_{X_{b}}^{2}.
\end{equation}
Due to 
\begin{equation*}
	\begin{aligned}
		\|\triangle\partial_{j}P^{5,3}\|_{L^{2}}\leq C\|\partial_{y}V\|_{L^{\infty}}\|\nabla\partial_{x}Q_{3}\|_{L^{2}}\leq C\|\nabla\partial_{x}Q_{3}\|_{L^{2}},
	\end{aligned}
\end{equation*}
using (iv) of {Lemma \ref{lem W}}, there holds
\begin{equation}\label{Uj 2}
	\begin{aligned}
		\|e^{b\nu^{\frac13}t}\triangle\partial_{j}P^{5,3}\|_{L^{2}L^{2}}^{2}\leq C\|e^{b\nu^{\frac13}t}\nabla\partial_{x}Q_{3}\|_{L^{2}L^{2}}^{2}\leq C\nu^{-1}\|\partial_{x}^{2}Q_{3}\|_{X_{b}}^{2}\leq C\nu^{-\frac53}E_{5}^{2}.
	\end{aligned}
\end{equation}
Substituting (\ref{Uj 0})-(\ref{Uj 2}) into (\ref{Uj0}), we conclude that 
\begin{equation}\label{result Uj}
	\begin{aligned}
		&\|\partial_{x}^{2}u_{j,\neq}\|_{X_{b}}^{2}+\|\partial_{x}(\partial_{z}-\kappa\partial_{y})u_{j,\neq}\|_{X_{b}}^{2}
		\\\leq&C\big(\|u_{\rm in}\|_{H^{2}}^{2}+\|\triangle Q_{1}\|_{X_{b}}^{2}+\nu^{\frac43}\|\partial_{x}^{2}Q_{2}\|_{X_{b}}^{2}
		+\nu^{-\frac53}E_{5}^{2}+\nu^{-1}\|e^{b\nu^{\frac13}t}\partial_{x}h_{j}\|_{L^{2}L^{2}}^{2} \big).
	\end{aligned}
\end{equation}

{\underline{\textbf{Step IV. Estimate $ E_{7}+\nu^{\frac43}\|\triangle u_{3,\neq}\|_{X_{b}}^{2} $ and $ E_{6}. $}}} 
As $ Q=Q_{1}+\nu Q_{2}+Q_{3}, $ we get
\begin{equation*}
	\begin{aligned}
		\|\partial_{x}\nabla Q\|_{X_{b}}\leq& \nu\|\partial_{x}\nabla Q_{2} \|_{X_{b}}
		+\|\partial_{x}\nabla Q_{1}\|_{X_{b}}+\|\partial_{x}\nabla Q_{3}\|_{X_{b}}
		\\\leq& \nu\|\partial_{x}\nabla Q_{2}\|_{X_{b}}+\|\triangle Q_{1}\|_{X_{b}}
		+\|\partial_{x}\nabla Q_{3}\|_{X_{b}},
	\end{aligned}
\end{equation*}
which along with (\ref{result Uj}) give that 
\begin{equation*}
	\begin{aligned}
		E_{7}\leq&C\big(\|u_{\rm in}\|_{H^{2}}^{2}+\|\triangle Q_{1}\|_{X_{b}}^{2}
		+\nu^{\frac43}\|\partial_{x}^{2}Q_{2}\|_{X_{b}}^{2}+\nu^{-\frac53}E_{5}^{2}
		+\|\partial_{x}\nabla Q_{3}\|_{X_{b}}^{2}\\&+\nu^{2}\|\partial_{x}\nabla Q_{2}\|_{X_{b}}^{2}
		+\nu^{-1}\|e^{b\nu^{\frac13}t}(\partial_{x}h_{2}, \partial_{x}h_{3})\|_{L^{2}L^{2}}^{2}\big).
	\end{aligned}
\end{equation*}
By (\ref{bar Q''}), (ii) of Lemma \ref{lem W} and (iv) of Lemma \ref{lem W}, the above inequality indicates that
\begin{equation}\label{E6 1}
	\begin{aligned}
		E_{7}\leq& C\big(\|u_{\rm in}\|_{H^{2}}^{2}+\nu^{-1}\|e^{b\nu^{\frac13}t}\nabla H_{2}\|_{L^{2}L^{2}}^{2}+\nu^{-2}E_{5}^{2}\\&+\nu^{-\frac13}\|e^{b\nu^{\frac13}t}\nabla h_{3,\neq}\|_{L^{2}L^{2}}^{2}+\nu^{-1}\|e^{b\nu^{\frac13}t}(\partial_{x}h_{2},\partial_{x}h_{3})\|_{L^{2}L^{2}}^{2} \big).
	\end{aligned}
\end{equation}
Using (\ref{result u3 hatQ}) and (i) of Lemma \ref{lem W}, we find
\begin{equation}\label{U3''}
	\begin{aligned}
		\|\triangle u_{3,\neq}\|_{X_{b}}^{2}\leq&C\left(\|\triangle u_{3,\neq}-2\partial_{x}Q_{4}\|_{X_{b}}^{2}+\|\partial_{x}Q_{4}\|_{X_{b}}^{2} \right)
		\\\leq&C\big(\|u_{\rm in}\|_{H^{2}}^{2}+\nu^{-\frac43}E_{7}+\nu^{-2}E_{5}^{2}+\nu^{-1}\|e^{b\nu^{\frac13}t}\nabla h_{3,\neq}\|_{L^{2}L^{2}}^{2} \big),
	\end{aligned}
\end{equation}
Combining (\ref{E6 1}) and (\ref{U3''}), one gets 
\begin{equation*}\label{E6 u3''}
	\begin{aligned}
		E_{7}+\nu^{\frac43}\|\triangle u_{3,\neq}\|_{X_{b}}^{2}\leq&C\big(\|u_{\rm in}\|_{H^{2}}^{2}+\nu^{-\frac23}E_{5}^{2}+E_{7}+\|e^{b\nu^{\frac13}t}\nabla h_{3,\neq}\|_{L^{2}L^{2}}^{2} \big)\\\leq&C\big(\|u_{\rm in}\|_{H^{2}}^{2}+\nu^{-1}\|e^{b\nu^{\frac13}t}\nabla H_{2}\|_{L^{2}L^{2}}^{2}+\nu^{-2}E_{5}^{2}\\&+\nu^{-\frac13}\|e^{b\nu^{\frac13}t}\nabla h_{3,\neq}\|_{L^{2}L^{2}}^{2}+\nu^{-1}\|e^{b\nu^{\frac13}t}\partial_{x}h_{3}\|_{L^{2}L^{2}}^{2}
		\big),
	\end{aligned}
\end{equation*}
where we use 
$\|e^{b\nu^{\frac13}t}\partial_{x}h_{2}\|_{L^{2}L^{2}}^{2}
\leq C\big(\|e^{b\nu^{\frac13}t}\nabla H_{2}\|_{L^{2}L^{2}}^{2}
+\|e^{b\nu^{\frac13}t}\partial_{x}h_{3}\|_{L^{2}L^{2}}^{2} \big).$

Next we denote
\begin{equation*}
	\begin{aligned}
		I_{k}=&\|e^{b\nu^{\frac13}t}\nabla H_{2,k}\|_{L^{2}L^{2}}^{2}+\|e^{b\nu^{\frac13}t}\partial_{x}h_{3,k}\|_{L^{2}L^{2}}^{2}+\nu^{\frac23}\|e^{b\nu^{\frac13}t}\nabla(h_{3,k})_{\neq}\|_{L^{2}L^{2}}^{2},
		\quad k=1,\cdots,7. 
	\end{aligned}
\end{equation*}	
As $ (h_{j,4})_{\neq}=0 $ for $ j\in \{2,3\},$  $ I_{4}=0. $

For $ I_{1}, $ using Lemma \ref{lem: zero and non-zero mode}, we get
\begin{equation*}
	I_{1}\leq C\nu^{-1}E_{2}^{2}\big(E_{7}+\nu^{\frac43}\|\triangle u_{3,\neq}\|_{X_{b}}^{2} \big).
\end{equation*}	
Due to
\begin{equation*}
	\begin{aligned}
		&\|\nabla H_{2,k}\|_{L^{2}}+\|\partial_{x}h_{3,k}\|_{L^{2}}+\nu^{\frac23}\|\nabla(h_{3,k})_{\neq}\|_{L^{2}}
		%			\\\leq&C\left(\|(h_{2,k})_{\neq}\|_{H^{1}}+\|\kappa\|_{H^{3}}\|(h_{3,k})_{\neq}\|_{H^{1}}+\|(h_{3,k})_{\neq}\|_{H^{1}} \right)
		\leq C\left(\|\nabla(h_{2,k})_{\neq}\|_{L^{2}}+\|\nabla(h_{3,k})_{\neq}\|_{L^{2}} \right),
	\end{aligned}
\end{equation*}
there holds
\begin{equation}\label{Ik}
	I_{k}\leq C\big(\|e^{b\nu^{\frac13}t}\nabla(h_{2,k})_{\neq}\|_{L^{2}L^{2}}^{2}+\|e^{b\nu^{\frac13}t}\nabla(h_{3,k})_{\neq}\|_{L^{2}L^{2}}^{2} \big).
\end{equation}

For $ I_{2}, I_{5}, I_{6} $ and $ I_{7}, $ by Lemma \ref{lem 2}, Lemma \ref{lem: zero and non-zero mode} and (\ref{Ik}), one obtains
\begin{equation*}
	I_{2}\leq C\nu^{-1}E_{2}^{2}E_{7},~~~ I_{5}\leq C\nu^{-1}E_{2}^{2}E_{7},~~~
	I_{6}\leq C\nu^{-1}E_{4}^{4},~~~  I_{7}\leq C\nu E_{1}^{2}E_{7}.
\end{equation*}

For $ I_{3}, $ by Lemma \ref{kappa} and (\ref{widehat u10 H2}), we find 
\begin{equation*}
	\begin{aligned}
		&\|\kappa\|_{H^{1}}\leq C\|\widehat{u_{1,0}}\|_{H^{2}}\leq CE_{1}\nu t\leq C\nu t,\\
		&\|\kappa\|_{H^{3}}\leq C\|\widehat{u_{1,0}}\|_{H^{4}}\leq C. 
	\end{aligned}
\end{equation*}
 Then, it holds
\begin{equation*}
	\begin{aligned}
		&\|\nabla H_{2,3}\|_{L^{2}}\leq\|\nabla h_{2,3}\|_{L^{2}}+C\|\kappa\|_{H^{2}}\|(h_{3,3})_{\neq}\|_{H^{1}}\\\leq&\|\nabla h_{2,3}\|_{L^{2}}+C\|\kappa\|_{H^{1}}^{\frac12}\|\kappa\|_{H^{3}}^{\frac12}\|(h_{3,3})_{\neq}\|_{H^{1}}\leq\|\nabla h_{2,3}\|_{L^{2}}+C(t\nu)^{\frac12}\|\nabla h_{3,3}\|_{L^{2}},
	\end{aligned}
\end{equation*}
which along with Lemma \ref{lem 2} show that 
\begin{equation*}
	\begin{aligned}
		I_{3}\leq C\big(\|e^{2a\nu^{\frac13}t}\nabla h_{2,3}\|_{L^{2}L^{2}}^{2}+\nu^{\frac23}\|e^{2a\nu^{\frac13}t}\nabla h_{3,3}\|_{L^{2}L^{2}}^{2}+\|e^{2a
			\nu^{\frac13}t}\partial_{x}h_{3,3}\|_{L^{2}L^{2}}^{2} \big)\leq C\nu^{-1}E_{4}^{4},
	\end{aligned}
\end{equation*}
where we use the fact that $ (t\nu)^{\frac12}e^{b\nu^{\frac13}t}\leq C\nu^{\frac13}te^{2a\nu^{\frac13}t}. $

Collecting the estimates of $ I_{1}-I_{7}, $ we conclude that
\begin{equation*}
	\begin{aligned}
		&\|e^{b\nu^{\frac13}t}\nabla H_{2}\|_{L^{2}L^{2}}^{2}+\|e^{b\nu^{\frac13}t}\partial_{x}h_{3}\|_{L^{2}L^{2}}^{2}+\nu^{\frac23}\|e^{b\nu^{\frac13}t}\nabla h_{3,\neq}\|_{L^{2}L^{2}}^{2}\\\leq&C\nu^{-1}E_{2}^{2}\left(E_{7}+\nu^{\frac43}\|\triangle u_{3,\neq}\|_{X_{b}}^{2} \right)+C\nu^{-1}E_{4}^{4}+C\nu E_{1}^{2}E_{7}.
	\end{aligned}
\end{equation*}
Noting that $ E_{1}\leq \varepsilon_{0} $ and $ E_{2}\leq \varepsilon_{0}\nu $, and taking $ \varepsilon_{0} $ small enough,  we conclude that
\begin{equation*}
	\begin{aligned}
		E_{6}^{2}\leq C(E_{7}+\nu^{\frac43}\|\triangle u_{3,\neq}\|_{X_{b}}^{2})\leq C\left(\|u_{\rm in}\|_{H^{2}}^{2}+\nu^{-2}E_{4}^{4}+\nu^{-2}E_{5}^{2} \right).
	\end{aligned}
\end{equation*}
\end{proof}

%%%%%%%%%%%%%%%%%%%%%/

%%%%%%%%%%%%%%%%%%%%%%%%%

\appendix
\section{Space-time estimates}
Firstly, we recall the following space-time estimates.
\begin{proposition}[{\bf Proposition 4.1} in \cite{wei2}]\label{prop Lf}
Let $ f $ satisfy
\begin{equation*}
	\partial_{t}f-\nu\triangle f+y\partial_{x}f=\partial_{x}f_{1}+f_{2}+\nabla\cdot f_{3},
\end{equation*}
for $ t\in[0,T]. $ If $ P_{0}f=P_{0}f_{1}=P_{0}f_{2}=P_{0}f_{3}=0, $ then for $ a\geq 0, $ there holds
\begin{equation*}
	\begin{aligned}
		\|f\|_{X_{a}}^{2}\leq C\left(\|f(0)\|_{L^{2}}^{2}+\|e^{a\nu^{\frac13}t}\nabla f_{1}\|_{L^{2}L^{2}}^{2}+\nu^{-\frac13}\|e^{a\nu^{\frac13}t}f_{2}\|_{L^{2}L^{2}}^{2}+\nu^{-1}\|e^{a\nu^{\frac13}t}f_{3}\|_{L^{2}L^{2}}^{2} \right).
	\end{aligned}
\end{equation*}
\end{proposition}

The following proposition can be derived from Proposition 4.3 in \cite{wei2}, and we omit it.
\begin{proposition}\label{prop Lf 1}
Let $ (f,g) $ satisfy
\begin{equation*}
	\left\{
	\begin{array}{lr}
		\partial_{t}f-\nu\triangle f+y\partial_{x}f=\triangle f_{1}+f_{2}, \\
		\partial_{t}g-\nu\triangle g+y\partial_{x}g-2\partial_{x}\partial_{z}\triangle^{-2}f=g_{1}, 	\end{array}
	\right.
\end{equation*}
for $ t\in[0,T]. $ If $ P_{0}f=P_{0}f_{1}=P_{0}f_{2}=P_{0}g=P_{0}g_{1}=0, $ then for $ a\geq 0, $ it holds that
\begin{equation*}
	\begin{aligned}
		&\|f\|_{X_{a}}^{2}+\|(\partial_{x}^{2}+\partial_{z}^{2})g\|_{X_{a}}^{2}\leq C\big(\|f(0)\|_{L^{2}}^{2}+\|g(0)\|_{H^{2}}^{2}+\nu^{-1}\|e^{a\nu^{\frac13}t}\nabla f_{1}\|_{L^{2}L^{2}}^{2}\\&+\nu^{-\frac13}\|e^{a\nu^{\frac13}t}f_{2}\|_{L^{2}L^{2}}^{2}+\nu^{-1}\|e^{a\nu^{\frac13}t}(\partial_{x},\partial_{z})g_{1}\|_{L^{2}L^{2}}^{2}  \big).
	\end{aligned}
\end{equation*}
\end{proposition}
The following lemma is important in proving Proposition \ref{Lvf 0} and Proposition \ref{Lvf}.
\begin{lemma}\label{kappa}
The definitions of $\kappa, \rho_{1}$ and $ \rho_{2} $ are given by (\ref{define kappa}) and (\ref{rho1 rho2}), then there hold
$$	\|\kappa\|_{H^{1}}\leq C\|\widehat{u_{1,0}}\|_{H^{2}},\quad \|\kappa\|_{H^{3}}\leq C\|\widehat{u_{1,0}}\|_{H^{4}},\quad \|\partial_{t}\kappa\|_{H^{1}}\leq C\|\partial_{t}\widehat{u_{1,0}}\|_{H^{2}},$$
$$ \|\rho_{1}\|_{H^{2}}+\|\rho_{2}\|_{H^{2}}\leq C\|\widehat{u_{1,0}}\|_{H^{4}},\quad \|\partial_{t}\rho_{1}\|_{L^{2}}\leq C\|\partial_{t}\widehat{u_{1,0}}\|_{H^{2}}. $$
\end{lemma}
\begin{proposition}[{\bf Proposition 4.4} in \cite{wei2}]\label{Lvf 0}
Let $ f $ satisfy
\begin{equation*}
	\mathcal{L}_{V}f=\partial_{x}f_{1}+f_{2}+\nabla\cdot f_{3},
\end{equation*}
for $ t\in[0,T]. $ Assume that
$\|\widehat{u_{1,0}}\|_{H^{4}}+\|\partial_{t}\widehat{u_{1,0}}\|_{H^{2}}/\nu<\delta_{2},$
for some small constant $ \delta_{2} $ independent of $ \nu $ and $ T. $ If $ P_{0}f=P_{0}f_{1}=P_{0}f_{2}=P_{0}f_{3}=0, $ then for $ a\geq 0, $ it holds that
\begin{equation*}
	\begin{aligned}
		\|f\|_{X_{a}}^{2}\leq C\big(\|f(0)\|_{L^{2}}^{2}+\|e^{a\nu^{\frac13}t}\nabla f_{1}\|_{L^{2}L^{2}}^{2}+\nu^{-\frac13}\|e^{a\nu^{\frac13}t}f_{2}\|_{L^{2}L^{2}}^{2}+\nu^{-1}\|e^{a\nu^{\frac13}t}f_{3}\|_{L^{2}L^{2}}^{2}
		\big).
	\end{aligned}
\end{equation*}
\end{proposition}

\begin{proposition}[{\bf Proposition 4.7} in \cite{wei2}]\label{Lvf}
Let $ f $ satisfy
\begin{equation*}
	\mathcal{L}_{V}f=f_{1}+f_{2}+f_{3},
\end{equation*}
for $ t\in[0,T]. $ Assume that
$\|\widehat{u_{1,0}}\|_{H^{4}}+\|\partial_{t}\widehat{u_{1,0}}\|_{H^{2}}/\nu<\delta_{3},$
for some small constant $ \delta_{3} $ independent of $ \nu $ and $ T. $ If $ P_{0}f=P_{0}f_{1}=P_{0}f_{2}=P_{0}f_{3}=0, $ then for $ a\geq 0, $ it holds that
\begin{equation*}
	\begin{aligned}
		\|\partial_{x}^{2}f\|_{X_{a}}^{2}+\|\partial_{x}(\partial_{z}-\kappa\partial_{y})f\|_{X_{a}}^{2}&\leq C\big(\|f(0)\|_{H^{2}}^{2}+\|e^{a\nu^{\frac13}t}\triangle f_{1}\|_{L^{2}L^{2}}^{2}+\nu^{-\frac13}\|e^{a\nu^{\frac13}t}\partial_{x}^{2}f_{2}\|_{L^{2}L^{2}}^{2}\\&+\nu^{-\frac13}\|e^{a\nu^{\frac13}t}\partial_{x}(\partial_{z}-\kappa\partial_{y})f_{2}\|_{L^{2}L^{2}}^{2}+\nu^{-1}\|e^{a\nu^{\frac13}t}\partial_{x}f_{3}\|_{L^{2}L^{2}}^{2}		
		\big).
	\end{aligned}
\end{equation*}	
\end{proposition}

\begin{proposition}[{\bf Proposition 4.9} in \cite{wei2}]\label{tilde Lv}
Let $ f $ satisfy $ \widetilde{\mathcal{L}_{V}}f=f_{1} $ for $ t\in[0, T]. $ Assume that
$\|\widehat{u_{1,0}}\|_{H^{4}}+\|\partial_{t}\widehat{u_{1,0}}\|_{H^{2}}/\nu<\delta_{4},$
for some small constant $ \delta_{4}\in(0, \delta_{3}) $ independent of $ \nu $ and $ T. $ If $ P_{0}f=P_{0}f_{1}=0, $ then for $ a\geq 0, $ it holds that
\begin{equation*}
	\|\triangle f\|_{X_{a}}^{2}\leq C\left(\|\triangle f(0)\|_{L^{2}}^{2}+\nu^{-1}\|e^{a\nu^{\frac13}t}\nabla f_{1}\|_{L^{2}L^{2}}^{2} \right).
\end{equation*}
\end{proposition}

\section{Sobolev embeddings}
The following lemma can be used to estimate the $L^{\infty}$ norm for the zero mode, which can be found in  Lemma 3.1 of \cite{CWW}.
\begin{lemma}\label{lem zero nonzero}
For a given function $f(x,y,z)$ and $f_0=\frac{1}{|\mathbb{T}|}\int_{\mathbb{T}}{f}(t,x,y,z)dx,$ we have
\begin{equation}\label{sob_result_1}
	\begin{aligned}
		&\|f_0\|_{L^{\infty}}\leq C\big(\|\partial_yf_0\|^{\frac{1}{2}}_{L^2}\|f_0\|^{\frac{1}{2}}_{L^2}+\|\partial_y\partial_zf_0\|^{\frac{1}{2}}_{L^2}
		\|\partial_zf_0\|^{\alpha-\frac{1}{2}}_{L^2}
		\|f_0\|^{1-\alpha}_{L^2}\big),\\
		&\|f_0\|_{L^{\infty}}
		\leq 
		C\big(\|\partial_yf_0\|^{\frac{1}{2}}_{L^2}\|f_0\|^{\frac{1}{2}}_{L^2}+\|\partial_y\partial_zf_0\|^{\alpha-\frac{1}{2}}_{L^2}
		\|\partial_zf_0\|^{\frac{1}{2}}_{L^2}
		\|\partial_yf_0\|^{1-\alpha}_{L^2}\big),\\
		&\|f_{0}\|_{L^{\infty}_{z}L^2_y}\leq C\left(\|f_0\|_{L^2}+\|\partial_zf_0\|_{L^2}^{\alpha}\|f_0\|_{L^2}^{1-\alpha}\right),\\
		&\|f_{0}\|_{L^{\infty}_yL^{2}_{z}}
		\leq \|\partial_yf_0\|_{L^2}^{\frac{1}{2}}
		\|f_0\|_{L^2}^{\frac{1}{2}},
	\end{aligned}
\end{equation}	
where $\alpha$ is a constant with $\alpha\in(\frac{1}{2},1].$
\end{lemma}
The following lemma can be used to estimate the $L^{\infty}$ norm for the non-zero mode, and can be found in Lemma 3.2 of \cite{CWW}. 
\begin{lemma}\label{lem nonzreo}
For a given function $ f=f(x,y,z) $, if $ P_{0}f=0, $ it holds that
\begin{equation}\label{f_neq}
	\begin{aligned}
		&\|f\|_{L^{\infty}}\leq C\big(\|\partial_{y}\partial_{z}f\|_{L^{2}}^{\frac12}\|\partial_{x}\partial_{z}f\|_{L^{2}}^{\alpha-\frac12}
		\|\partial_{x}^{2}f\|_{L^{2}}^{\alpha-\frac12}\|\partial_{x}f\|_{L^{2}}^{\frac32-2\alpha}
		+\|\partial_{x}\partial_{y}f\|_{L^{2}}^{\frac12}\|\partial_{x}f\|_{L^{2}}^{\alpha-\frac12}\|f\|_{L^{2}}^{1-\alpha} \big),
		\\&\|f\|_{L^{\infty}_{y,z}L^{2}_{x}}\leq C\big(\|\partial_{y}f\|_{L^{2}}^{\frac12}\|f\|_{L^{2}}^{\frac12}
		+\|\partial_{z}f\|_{L^{2}}^{\frac12}\|\partial_{z}\partial_{y}f\|_{L^{2}}^{\alpha-\frac12}\|\partial_{y}f\|_{L^{2}}^{1-\alpha} \big),
		\\&\|f\|_{L^{\infty}_{x,y}L^{2}_{z}}\leq C\|\partial_{x}f\|_{L^{2}}^{\frac12}\|\partial_{x}\partial_{y}f\|_{L^{2}}^{\alpha-\frac12}
		\|\partial_{y}f\|_{L^{2}}^{1-\alpha},\\&
		\|f\|_{L^{\infty}_{x}L^{2}_{y,z}}\leq C\|\partial_{x}f\|_{L^{2}}^{\alpha}\|f\|_{L^{2}}^{1-\alpha}
		,\\&\|f\|_{L^{\infty}_{z}L^{2}_{x,y}}\leq C\left(\|f\|_{L^{2}}+\|\partial_{z}f\|_{L^{2}}^{\alpha}\|f\|_{L^{2}}^{1-\alpha} \right),\\&
		\|f\|_{L^{\infty}_{y}L^{2}_{x,z}}\leq C\|\partial_{y}f\|_{L^{2}}^{\frac12}\|f\|_{L^{2}}^{\frac12},
		\\&\|f\|_{L^{\infty}_{x,z}L^{2}_{y}}\leq\big(\|\partial_{x}f\|_{L^{2}}^{\alpha}\|f\|_{L^{2}}^{1-\alpha}+\|\partial_{x}\partial_{z}f\|_{L^{2}}^{\frac12}\|\partial_{x}f\|_{L^{2}}^{\alpha-\frac12}\|\partial_{z}f\|_{L^{2}}^{\alpha-\frac12}\|f\|_{L^{2}}^{\frac32-2\alpha} \big),
	\end{aligned}
\end{equation}
where $ \alpha\in (\frac12,\frac34] $ for $ (\ref{f_neq})_{1}$ and $(\ref{f_neq})_{7}, $ and $ \alpha\in(\frac12, 1] $ for others.
\end{lemma}

%Lemma \ref{sob_13} - \ref{sob_15} can be found in \cite{Chen1} and \cite{wei2}, so we omitted the proof.
%\begin{lemma}[{\bf Lemma 11.1} in \cite{Chen1}]\label{sob_13}
%	For given functions $ f_{1} $ and $ f_{2} $, it holds that 
%	\begin{equation*}
%		\begin{aligned}
	%			&\|\nabla(f_{1}f_{2})\|_{L^{2}}\leq C\big(\|(\partial_{z}f_{1},f_{1})\|_{H^{2}}\|(\partial_{x}f_{2},f_{2})\|_{L^{2}}+\|(\partial_{x}\partial_{z}f_{1},\partial_{x}f_{1},\partial_{z}f_{1},f_{1} )\|_{L^{2}}\|f_{2}\|_{H^{2}} \big).
	%		\end{aligned}
%	\end{equation*}	
%\end{lemma}
From Lemma \ref{lem nonzreo}, we can derive the following Lemma \ref{sob ineq}-\ref{sob_04} (see also Lemma 5.1-5.4 in \cite{wei2}).
\begin{lemma}\label{sob ineq}
If $ \partial_{x}f_{1}=0, $ there holds
\begin{equation*}
	\begin{aligned}
		&\|f_{1}f_{2}\|_{H^{2}}\leq  C\|f_{1}\|_{H^{1}}\left(\|f_{2}\|_{H^{2}}
		+\|\partial_{z}f_{2}\|_{H^{2}}\right)+C\|f_{1}\|_{H^{3}}\left(\|f_{2}\|_{L^{2}}+\|\partial_{z}f_{2}\|_{L^{2}} \right),\\
		&\|f_{1}f_{2}\|_{H^{3}}\leq C\|f_{1}\|_{H^{1}}\left(\|f_{2}\|_{H^{3}}+\|\partial_{z}f_{2}\|_{H^{3}} \right)+C\|f_{1}\|_{H^{3}}\left(\|f_{2}\|_{H^{1}}+\|\partial_{z}f_{2}\|_{H^{1}} \right).
	\end{aligned}
\end{equation*}
\end{lemma}

\begin{lemma}\label{sob_15}
For given functions $ f_{1}, f_{2} $ and $ j\in\{1,3\}, $ there holds
\begin{equation*}
	\|f_{1}\partial_{j}f_{2}\|_{L^{2}}\leq C\left(\|\partial_{j}f_{1}\|_{L^{2}}+\|f_{1}\|_{L^{2}} \right)\|\triangle f_{2}\|_{L^{2}},
\end{equation*}
\begin{equation*}
	\|f_{1}f_{2}\|_{L^{2}}+\|\partial_{j}(f_{1}f_{2})\|_{L^{2}}\leq C\left(\|\partial_{j}f_{1}\|_{L^{2}}+\|f_{1}\|_{L^{2}} \right)\|f_{2}\|_{H^{2}}.
\end{equation*}
\end{lemma}

\begin{lemma}\label{sob_04}
It holds that for $ j\in\{1,3\} $
\begin{equation*}
	\begin{aligned}
		&\|f_{1}f_{2}\|_{L^{2}}\leq C\left(\|\partial_{j}f_{1}\|_{H^{1}}+\|f_{1}\|_{H^{1}} \right)\|f_{2}\|_{L^{2}}+C\left(\|\partial_{j}f_{1}\|_{L^{2}}+\|f_{1}\|_{L^{2}} \right)\|f_{2}\|_{H^{1}},\\
		&\|\partial_{j}(f_{1}f_{2})\|_{L^{2}}\leq C\|(\partial_{j}f_{1},f_{1})\|_{H^{1}}\|(\partial_{j}f_{2},f_{2})\|_{L^{2}}+C\|(\partial_{j}f_{1},f_{1})\|_{L^{2}}\|(\partial_{j}f_{2},f_{2})\|_{H^{1}}.
	\end{aligned}
\end{equation*}
If $ P_{0}f_{1}=0, $ then for $ s=1,2,3, $ there holds
\begin{equation*}
	\|f_{1}f_{2}\|_{H^{s}}\leq C\|\partial_{x}f_{1}\|_{H^{s+1}}\|f_{2}\|_{L^{2}}+C\|\partial_{x}f_{1}\|_{L^{2}}\|f_{2}\|_{H^{s+1}}.
\end{equation*}
\end{lemma}

\begin{lemma}[{\bf Lemma 5.5} in \cite{wei2}]\label{sob_14}
Assume that $ \|\widehat{u_{1,0}}\|_{H^{4}}<\delta_{1} $ with $ \delta_{1} $ as in {\rm Lemma \ref{kappa}}.
If $ \partial_{x}f_{1}=0, P_{0}f_{2}=0, $ then it holds that
\begin{equation*}
	\|f_{1}f_{2}\|_{L^{2}}\leq C\|f_{1}\|_{H^{1}}\left(\|f_{2}\|_{L^{2}}+\|(\partial_{z}-\kappa\partial_{y})f_{2}\|_{L^{2}} \right),
\end{equation*}
\begin{equation*}
	\|\nabla\triangle^{-1}(f_{1}f_{2})\|_{L^{2}}\leq C\|f_{1}\|_{L^{2}}\left(\|f_{2}\|_{L^{2}}+\|(\partial_{z}-\kappa\partial_{y})f_{2}\|_{L^{2}} \right),
\end{equation*}
\begin{equation*}
	\|\nabla(f_{1}f_{2})\|_{L^{2}}\leq C\|f_{1}\|_{H^{1}}\left(\|f_{2}\|_{H^{1}}+\|(\partial_{z}-\kappa\partial_{y})f_{2}\|_{H^{1}} \right),
\end{equation*}
and for $ j=2,3, $
\begin{equation*}
	\|[\partial_{j}\triangle^{-1},f_{1}]\triangle f_{2}\|_{H^{1}}\leq C\|f_{1}\|_{H^{2}}\left(\|\nabla f_{2}\|_{L^{2}}+\|(\partial_{z}-\kappa\partial_{y})\nabla f_{2}\|_{L^{2}} \right).
\end{equation*}
\end{lemma}

\section*{Acknowledgement}
The authors would like to thank Professors Zhifei Zhang and  Weiren Zhao  for some helpful communications.
W. Wang was supported by National Key R\&D Program of China (No. 2023YFA1009200) and NSFC under grant 12471219 and 12071054.

\section*{Declaration of competing interest}
The authors declare that they have no known competing financial interests
or personal relationships that could have appeared to influence the work reported in this paper.

\end{document}